\documentclass[ruled]{article}
\usepackage[latin9]{inputenc}
\usepackage[a4paper]{geometry}
\usepackage{verbatim}
\usepackage{enumitem}
\usepackage[algo2e]{algorithm2e}
\usepackage{algorithm}
\usepackage{algorithmic}
\usepackage{amsmath}
\usepackage{amsthm}
\usepackage{amssymb}
\usepackage{hyperref}
\usepackage[numbers]{natbib}

\usepackage[resetlabels]{multibib}
\newcites{app}{Reference}

\newcommand{\assign}{:=}
\newcommand{\cdummy}{\cdot}

\newcommand{\tmop}[1]{\ensuremath{\operatorname{#1}}}

\newcommand{\smod}{\texttt{SMOD}}
\newcommand{\spl}{\texttt{SPL}}
\newcommand{\spp}{\texttt{SPP}}

\newcommand{\sgd}{\texttt{SGD}}
\newcommand{\segd}{\texttt{SEGD}}
\newcommand{\sepl}{\texttt{SEPL}}
\newcommand{\sepp}{\texttt{SEPP}}
\newcommand{\extra}{\texttt{SEMOD}}

\global\long\def\vertiii#1{\left\vert \kern-0.25ex  \left\vert \kern-0.25ex  \left\vert #1\right\vert \kern-0.25ex  \right\vert \kern-0.25ex  \right\vert }%

\global\long\def\brbra#1{\big(#1\big)}%

\global\long\def\til#1{\tilde{#1}}%

\global\long\def\trans{\textrm{T}}%

\global\long\def\Expe{\mathbb{E}}%

\global\long\def\argmin{\operatornamewithlimits{argmin}}%

\global\long\def\sign{\operatornamewithlimits{sign}}%

\global\long\def\prox{\mathrm{prox}}%

\global\long\def\and{\mathrm{and}}%

\global\long\def\raw{\rightarrow}%

\global\long\def\vep{\varepsilon}%

\global\long\def\Ebb{\mathbb{E}}%

\global\long\def\Rbb{\mathbb{R}}%

\global\long\def\Ncal{\mathcal{N}}%

\global\long\def\Ocal{\mathcal{O}}%

\global\long\def\Scal{\mathcal{S}}%

\global\long\def\Tcal{\mathcal{T}}%

\global\long\def\Xcal{\mathcal{X}}%

\theoremstyle{plain}
\newtheorem{thm}{\protect\theoremname}[section]
\theoremstyle{definition}

\theoremstyle{plain}
\newtheorem{lem}[thm]{\protect\lemmaname}
\theoremstyle{remark}
\newtheorem*{rem*}{\protect\remarkname}
\newtheorem{rem}{\protect\remarkname}
\providecommand{\definitionname}{Definition}
\providecommand{\lemmaname}{Lemma}
\providecommand{\remarkname}{Remark}
\providecommand{\theoremname}{Theorem}

\global\long\def\trans{^{\mathrm T}}

\usepackage{graphicx}
\usepackage{multirow}
\usepackage{optidef}

\makeatother

\providecommand{\definitionname}{Definition}
\providecommand{\lemmaname}{Lemma}
\providecommand{\remarkname}{Remark}
\providecommand{\theoremname}{Theorem}

\newif\ifOneCol
\OneColtrue

\usepackage[toc,page,header]{appendix}
\usepackage{minitoc}

\author{%
  Qi Deng$^1$\footnote{QD was partially supported by National Natural Science Foundation of China (Grant 11831002, 72150001)}~~~~~~~~~~~~~~~~~~Wenzhi Gao$^2$  \\\\
  Shanghai University of Finance and Economics\\ \\
  $^1$\texttt{qideng@sufe.edu.cn}~~~~$^2$\texttt{gwz@163.shufe.edu.cn}  \\
  }
\date{}

\begin{document}

\title{Minibatch and Momentum Model-based Methods for Stochastic Weakly Convex Optimization}

\maketitle

\begin{abstract}
Stochastic model-based methods have received increasing attention
lately due to their appealing robustness to the stepsize selection
and provable efficiency guarantee.
We make two important extensions for  improving model-based methods on stochastic weakly convex optimization. 
First, we propose new minibatch model-based methods by involving a set of samples to approximate the model function in each iteration. For the first time, we show that stochastic algorithms achieve linear speedup over the batch size even for non-smooth and non-convex (particularly, weakly convex) problems. To this end, we develop a novel sensitivity analysis of the proximal mapping involved in each algorithm iteration. Our analysis appears to be of independent interests in more general settings.
Second, motivated by the success of momentum stochastic gradient descent, we propose a new stochastic extrapolated model-based method, greatly extending the classic Polyak momentum technique to a wider class of stochastic algorithms for weakly convex optimization. The rate of convergence to some natural stationarity condition is established over a fairly flexible range of extrapolation terms.

While mainly focusing on weakly convex optimization, we also extend our work to convex optimization. 
We apply the minibatch and extrapolated model-based methods to stochastic convex optimization, for which we provide a new complexity bound and promising linear speedup in batch size. Moreover, an accelerated model-based method based on Nesterov's momentum is presented, for which we establish an optimal complexity bound for reaching optimality.
\end{abstract}

\section{Introduction}

In this paper, we are interested in the following stochastic optimization problem:
\begin{mini}
{x\in\Xcal}{f(x)= \Expe_{\xi\sim\Xi} \big[f(x,\xi)\big]}{}{}\label{prob:main}
\end{mini}
where $f(\cdot,\xi)$ stands for the loss function,
 sample $\xi$ follows certain distribution $\Xi$, and $\Xcal$
is a closed convex set. We assume that $f(\cdot,\xi)$ is  weakly convex, namely, the sum of $f(x,\xi)$ and a  quadratic function $\frac{\lambda}{2}\|x\|^2$ is convex ($\lambda>0$). This type of non-smooth non-convex functions can be found in a variety of machine learning applications,  such as phase retrieval, robust PCA and low rank decomposition \citep{charisopoulos2019low}. To solve problem~(\ref{prob:main}), we consider the stochastic model-based method ({\smod}, \citep{duchi2018stochastic,davis2019stochasticweakly,asi2019the}), which comprises a large class of stochastic algorithms (including stochastic (sub)gradient descent, proximal point, among others).
Recent work \citep{duchi2018stochastic,davis2019stochasticweakly} show that {\smod} exhibits  promising convergence property: both asymptotic convergence and
rates of convergence to certain stationarity measure have been established for the {\smod} family. In addition, empirical results~\citep{davis2019stochasticweakly, duchi2019solving} indicate that {\smod} exhibits remarkable robustness to hyper-parameter tuning and often outperforms {\sgd}.

Despite much recent progress, our understanding of model-based methods for weakly convex optimization is still quite limited. Particularly, it  is still unknown whether {\smod} is competitive against modern {\sgd} used in practice.
We highlight some important remaining questions. 
First, despite the appealing robustness and stable convergence, the {\smod} family is sequential in nature. 
It is unclear whether  minibatching, which is immensely used in training learning models, can improve the performance of {\smod} when the problem is non-smooth.
Particularly, the current best complexity bound $\Ocal(\frac{L^{2}}{\vep^{4}})$ from \citep{davis2019stochasticweakly}, which is regardless of batch size,
is unsatisfactory.
Were this bound tight, 
a sequential algorithm (using one sample per iteration) would be optimal: it offers the highest processing speed per iteration as well as the
best iteration complexity. 
Therefore, it is  crucial to know whether minibatching can
 improve the complexity bound of the {\smod} family or the current bound is tight.
 Second, in modern applications, momentum technique has been playing a vital role in large-scale non-convex optimization (see \citep{yan2018a,RN312}). 
In spite of its effectiveness, to the best of our knowledge,
momentum technique has been provably efficient only in \textbf{1)} unconstrained smooth optimization \citep{liu2020improved,defazio2020understanding, gitman2019understanding} and \textbf{2)} non-smooth optimization with a simple constraint \citep{pmlr-v119-mai20b}, which constitute only a portion of the interesting applications.
From the practical aspect, it is peculiarly desirable to know whether momentum  technique is applicable beyond in {\sgd} and whether it can benefit the {\smod} algorithm family in the non-smooth and non-convex setting.

\textbf{Contributions.}
Our work is motivated by the aforementioned challenge to make {\smod} more practically efficient. We summarize the contributions as follows. First, we extend {\smod} to the minibatch setting and develop sharper rates of convergence to stationarity. Leveraging the tool of algorithm stability (\citep{bousquet2002stability, shalev2010learnability,hardt2016train}), we provide a nearly complete recipe on when minibatching would be helpful even in presence of non-smoothness. Our theory implies that  stochastic proximal point  and stochastic prox-linear are inherently parallelizable: both algorithms achieve linear speedup over the minibatch size.
To the best of our knowledge, this is the first time
  that these minibatch stochastic algorithms are proven to exhibit such an acceleration even for \emph{non-smooth} and \emph{non-convex} (particularly, \emph{weakly convex}) optimization.
 Moreover, our theory recovers the complexity of minibatch (proximal) {\sgd} in~\citep{davis2019stochasticweakly}, showing that  (proximal) {\sgd} enjoys the same linear speedup by minibatching for smooth composite problems with non-smooth regularizers or with constrained domain. 
 
Second, we present new extrapolated model-based methods by incorporating a Polyak-type momentum term. 
We develop a unified Lyapunov analysis to show that a worst-case complexity of $\Ocal(1/\vep^{4})$ holds for all momentum {\smod} algorithms. 
To the best of our knowledge, these are the first complexity results of momentum  stochastic  prox-linear and stochastic proximal point  for non-smooth non-convex optimization.
Since our analysis offers complexity guarantees for momentum {\sgd} and its proximal variant, our work appears to be more general than a recent study~\citep{pmlr-v119-mai20b}, which only proves the convergence of momentum projected {\sgd}.
Proximal {\sgd} is more advantageous in composite optimization, where the non-smooth term is often involved via its proximal operator rather than the subgradient.
For example, in the Lasso problem, it is often favorable to invoke the proximal operator of $\ell_1$ function (Soft-Thresholding) to enhance solution sparsity.
We summarize  the complexity results in Table~\ref{tab:complexity}.

Third, we develop  new convergence results of {\smod} for convex optimization, showing that minibatch extrapolated {\smod} achieves a promising linear speedup over the batch size under some mild condition. Specifically, to obtain some $\vep$-optimal solution, our proposed method exhibits an $\Ocal(1/\vep+1/(m\vep^2))$ complexity bound in the worst case. Moreover, we develop a new minibatch {\smod} based on Nesterov's momentum, achieving the $\Ocal(1/\vep^{1/2}+1/(m\vep^2))$ optimal complexity bound. Note that a similar complexity result, explicitly relying on the smoothness assumption, has been shown in a recent study~\citep{chadha2021accelerated}. Compared to this work, our analysis makes weaker assumptions, showing that  smoothness is not a must-have for many model-based algorithms, such as {\spl} and {\spp}, to get sharper complexity bound.

\begin{table}[!h]
\centering

\caption{Complexity of  {\smod} to
reach $\protect\Ebb\,\|\nabla_{1/\rho}f\|\le\protect\vep$ (M: minibatch; E: Extrapolation, $m$: batch size) \label{tab:complexity}
}
\begin{tabular}{c|c|c|c}
\hline 
\hline 
Algorithms  & Problem  & Current Best  & Ours\tabularnewline
\hline 
{M + \sgd}  & $f$: non-smooth  & $\Ocal(1/\vep^{4})$\cite{davis2019stochasticweakly}  & $\Ocal(1/\vep^{4})$\tabularnewline
M + Prox. {\sgd}  & $f=\ell+\omega$; $\ell$:smooth  & $\Ocal(1/(m\vep^{4})+1/\vep^{2})$\cite{davis2019stochasticweakly}  & $\Ocal(1/(m\vep^{4})+1/\vep^{2})$ \tabularnewline
M + \spl /\spp  & $f$: non-smooth  & $\Ocal(1/\vep^{4})$\cite{davis2019stochasticweakly}  & $\Ocal(1/(m\vep^{4})+1/\vep^{2})$ \tabularnewline
\hline 
{E + {\sgd}}  & $f$: non-smooth  & $\Ocal(1/\vep^{4})$\cite{pmlr-v119-mai20b}  & $\Ocal(1/\vep^{4})$\tabularnewline
{E + Prox. {\sgd}}  & $f=\ell+\omega$; $\ell$:smooth  & ---  & $\Ocal(1/\vep^{4})$ \tabularnewline
E + \spl /\spp  & $f$: non-smooth  & ---  & $\Ocal(1/\vep^{4})$\tabularnewline
\hline 
{M + E + {\sgd}}  & $f$: non-smooth  & $\Ocal(1/\vep^{4})$\cite{pmlr-v119-mai20b}  & $\Ocal(1/\vep^{4})$\tabularnewline
{M + E + Prox. {\sgd}}  & $f=\ell+\omega$; $\ell$:smooth  & ---  & $\Ocal(1/(m\vep^{4})+1/\vep^{2})$ \tabularnewline
M + E + \spl /\spp  & $f$: non-smooth  & ---  & $\Ocal(1/(m\vep^{4})+1/\vep^{2})$\tabularnewline
\hline 
\hline
\end{tabular}
\end{table}

\textbf{Other related work.}
 For smooth and composite optimization, it is well known that {\sgd}  can be linearly accelerated by minibatching (c.f. \citep{dekel2012optimal,saeed-lan-nonconvex-2013,takavc2015distributed}).
 Minibatch model-based methods have been studied primarily in the convex setting.
\citet{asi2020minibatch} investigates the speedups of minibatch stochastic model-based methods in the convex smooth, restricted strongly convex and convex interpolation settings, respectively. 
Since their assumptions differ from ours, the technique does not
readily apply to the non-convex setting. 
\citet{chadha2021accelerated} studies the accelerated minibatch model-based methods for convex smooth and convex interpolated problems. 
The interpolation setting, where the model can perfectly fit the data, is not considered in our paper.
Algorithm stability~\citep{bousquet2002stability,shalev2010learnability}---an important technique for analyzing the generalization performance of stochastic algorithms \citep{hardt2016train, bassily2020stability}, is the key tool to obtain some of our convergence results.
In contrast to the traditional work, our paper employs the stability argument to obtain sharper optimization convergence rates (with respect to the batch size). 
See Section~\ref{sec:minibatch}.
As noted by an anonymous reviewer, a similar idea of using stability analysis was proposed by~\citet{wang2017memory}, albeit with a different motivation from distributed stochastic optimization.
Robustness and fast convergence of model-based methods have been shown on various
 statistical learning problems \citep{charisopoulos2019low,duchi2019solving,asi2019the,berrada2019deep,frerix2018proximal,botev2017practical}.
 \citet{drusvyatskiy2018efficiency}
give a complete complexity analysis of the accelerated proximal-linear
methods for deterministic optimization. \citet{zhang2021stochastic} further improve the convergence rates of prox-linear methods on certain finite-sum and stochastic problems by using variance-reduction.  Momentum and accelerated methods for convex stochastic optimization can be referred from \citep{loizou2020momentum, sebbouh2020convergence}.
The study~\citep{defazio2020understanding, liu2020improved, yan2018a}
develop the convergence rate of stochastic momentum method
for smooth non-convex optimization.

\section{Background\label{sec:background}}
Throughout the paper, we use $\|\cdot\|$ to denote the Euclidean
norm and $\langle\cdot,\cdot\rangle$ to denote the Euclidean inner
product.
We assume that $f(x)$ is bounded below. i.e., $\min_{x}f(x)>-\infty$.
The subdifferential ${\partial}f(x)$ of function $f(x)$ is the
set of vectors $v\in\Rbb^{d}$ that satisfy:
$
f(y)\ge f(x)+\langle v,y-x\rangle+o(\|x-y\|),\text{as } y\raw x.
$
Any such vector in $\partial f(x)$ is called a subgradient and is denoted by $f^{\prime}(x)\in\partial f(x)$ for simplicity. 
We say that a point $x$ is stationary if $0\in \partial f(x) + N_\Xcal(x)$, where the normal cone $N_\Xcal(x)$ is defined as $N_\Xcal(x)\triangleq\{d: \langle d, y-x\rangle \le 0, \forall y\in \Xcal\}$. For a set $S$, define the set distance to $0$ by: $\|\Scal\|_{-}\triangleq\inf \{\|x-0\|, x\in \Scal\}$. It is natural to use the quantity $\|\partial f(x)+N_\Xcal(x)\|_{-}$ to measure the stationarity of point $x$.

\textbf{Moreau-envelope.} The $\mu$-Moreau-envelope
of $f$ is defined  by
$
f_{\mu}(x)\triangleq\min_{y\in\Xcal}\big\{ f(y)+\frac{1}{2\mu}\|x-y\|^{2}\big\}
$
and the  proximal mapping associated with $f(\cdot)$ is defined by
$
\prox_{\mu f}(x)\triangleq\argmin_{y\in\Xcal}\big\{ f(y)+\frac{1}{2\mu}\|x-y\|^{2}\big\}.
$
Assume that
$f(x)$ is $\lambda$-weakly convex, then for $\mu<\lambda^{-1}$, 
the Moreau envelope $f_{\mu}(\cdot)$ is differentiable and its gradient
is $\nabla f_{\mu}(x)=\mu^{-1}(x-\prox_{\mu f}(x))$.

The {\smod} family  iteratively computes the  proximal
map associated with a model function $f_{x^{k}}(\cdot,\xi_{k})$:%
\begin{equation}
x^{k+1}=\argmin_{x\in\Xcal}\left\{ f_{x^{k}}(x,\xi_{k})+\frac{\gamma_{k}}{2}\|x-x^{k}\|^{2}\right\} ,\label{eq:model-update}
\end{equation}
where $\{\xi_{k}\}$ are i.i.d. samples. Typical algorithms
and the accompanied models are described below.	

\textbf{Stochastic (Proximal) Gradient Descent}: consider the composite function $f(x,\xi)=\ell(x,\xi)+\omega(x)$ where  $\ell(x,\xi)$ is a data-driven and weakly-convex loss term and $\omega(x)$ is a convex regularizer such as $\ell_1$-penalty. {\sgd} applies the model function:
\begin{equation}
f_{y}(x,\xi)=\ell(y,\xi)+\big\langle \ell^{\prime}(y,\xi),x-y\big\rangle +\omega(x).\label{eq:model-linear}
\end{equation}

\textbf{Stochastic Prox-linear ({\spl})}: consider the composition function
$f(x,\xi)=h(C(x,\xi))$ where $h(\cdot,\xi)$ is convex continuous
 and $C(x,\xi)$ is a continuously differentiable map. We
perform partial linearization to obtain the model
\begin{equation}
f_{y}(x,\xi)=h\bigl(C(y,\xi)+\langle\nabla C(y,\xi),x-y\rangle\bigr).\label{eq:model-prox-linear}
\end{equation}

\textbf{Stochastic Proximal Point ({\spp})}: compute~(\ref{eq:model-update})
with full stochastic function:
\begin{equation}
f_{y}(x,\xi)=f(x,\xi).\label{eq:model-full}
\end{equation}

Throughout the paper, we assume that $f(x,\xi)$ is continuous and $\mu$-weakly convex, and that the model
function $f_{x}(\cdot,\cdot)$ satisfies the following assumptions \citep{davis2019stochasticweakly}.
\begin{enumerate}[label=\textbf{A\arabic*:},ref=A\arabic*]
\item \label{ass:model-weakly-cvx}For any $\xi\sim\Xi$, the model function $f_{x}(y,\xi)$
is $\lambda$-weakly convex in $y$ ($\lambda\ge0$). 
\item \label{ass:unbiased}Tightness condition: 
$
f_{x}(x,\xi)=f(x,\xi),\ \forall x\in\Xcal,\,  \xi\sim\Xi.
$
\item \label{ass:one-side-quad}One-sided quadratic approximation: 
$
f_{x}(y,\xi)-f(y,\xi)  \le\frac{\tau}{2}\|x-y\|^{2},\ \forall x,y\in\Xcal, \xi\sim\Xi.
$
\item \label{ass:model-lip}Lipschitz continuity:
 There exists $ L>0$ that
$
 f_{x}(z,\xi)-f_{x}(y,\xi) \leq L\thinspace\|z-y\|,$ for any $x, y, z\in\Xcal,\, \xi\sim\Xi. \\
$
\end{enumerate}
\begin{rem}
Assumption~\ref{ass:unbiased} is quite standard and will be used only in the convergence proof.
Combining
\ref{ass:model-weakly-cvx} and \ref{ass:one-side-quad},
we immediately have that $f(x,\xi)$ is  $(\lambda+\tau)$-weakly convex.  Thus, it suffices to assume that $\mu<\tau+\lambda$. 
Assumptions~\ref{ass:unbiased}-\ref{ass:model-lip} can be slightly relaxed by replacing the uniform bound with a bound on expectation  over $\xi$, leading to only a minor adjustment to the  analysis. 
\end{rem}
Denote $\hat{x}\triangleq\prox_{f/\rho}(x)=\argmin_y\big\{f(y)+\frac{\rho}{2}\|y-x\|^2\big\}$
for some $\rho>\mu$. \citet{davis2019stochasticweakly} revealed a striking feature of Moreau envelope to characterize stationarity:
\begin{equation*}
\|\hat{x}-x\| =\rho^{-1}\|\nabla f_{1/\rho}(x)\|,\ \textrm{and}\ 
\|\partial f(\hat{x})+N_{\Xcal}(\hat{x})\|_{-} \le\|\nabla f_{1/\rho}(x)\|.
\end{equation*}
Namely,
a point $x$ with small gradient norm $\|\nabla f_{1/\rho}(x)\|$ stays
in the proximity of a nearly-stationary point $\hat{x}$. With this
observation, they show the first complexity result of {\smod} for non-smooth non-convex optimization: 
$
\min_{1\le k\le K}\Ebb[\|\nabla f_{1/\rho}(x^{k})\|]^{2}\le\Ocal(\frac{L}{\sqrt{K}}).
$
Note that this rate is regardless of the size of minibatches since it does not explicitly use any information of the samples other than the Lispchitzness of the model function. 
Due to this limitation, it remains unclear whether  minibatching can further improve the  convergence rate of {\smod}.

\section{{\smod} with minibatches\label{sec:minibatch}}

In this section, we present a minibatch {\smod}
method which takes a small batch of i.i.d. samples to estimate
the model function. The overall procedure is detailed
in Algorithm\,\ref{alg:mini-batch}. Within each iteration, Algorithm\,\ref{alg:mini-batch}
forms a stochastic model function $f_{x^{k}}(\cdot,B_{k})=\frac{1}{m_{k}}\sum_{i=1}^{m_{k}}f_{x^{k}}(x,\xi_{k,i})$
parameterized at $x^{k}$ by sampling over $m_{k}$ i.i.d. samples
$B_{k}=\xi_{k,1},\ldots,\xi_{k,m_{k}}$. Then it performs proximal update to get the next iterate $x^{k+1}$.  We will illustrate  the main convergence results of Algorithm~\ref{alg:mini-batch} and leave all the proof details in  Appendix sections.
But first, let us present an additional assumption.
\begin{enumerate}[label=\textbf{A\arabic*:},ref=A\arabic*,start=5]
\item \label{ass:two-sides-quad}Two-sided quadratic bound: for any $x,y\in \Xcal$,  $\xi\sim \Xi$,
$
\big\vert f_{x}(y,\xi)-f(y,\xi)\big\vert\le\frac{\tau}{2}\|x-y\|^{2}.
$\\

\end{enumerate}

\begin{rem}
 Assumption~\ref{ass:two-sides-quad} is   vital  for our improved convergence analysis. While it is slightly stronger than \ref{ass:one-side-quad}, \ref{ass:two-sides-quad} is indeed satisfied by the {\smod} family in most contexts: \textbf{1)} For {\spp}, \ref{ass:two-sides-quad} is trivially satisfied by taking $f_x(y,\xi)=f(y,\xi)$.
 \textbf{2)} For {\spl}, we minimize a composition function $f(x,\xi)=h(C_\xi(x))$ where $h(\cdot)$ is a $c_1$-Lipschitz convex function and $C_\xi(\cdot)$ is a $c_2$-Lipschitz smooth map.
 In view of (\ref{eq:model-prox-linear}), \ref{ass:two-sides-quad} is verified with
$
|f_x(y,\xi)-f(y,\xi)| \le  c_1\big\|C_\xi(y)-C_\xi(x)-\nabla C_\xi(x)\trans(y-x)\big\| \le  \frac{c_1 c_2}{2} \|x-y\|^2.
$
\textbf{3)} For {\sgd}, \ref{ass:two-sides-quad} is satisfied if $\ell(\cdot,\xi)$ is $c_3$-Lipschitz smooth for some $c_3>0$, as
$
|f_x(y,\xi)-f(y,\xi)| \le  |\ell(y,\xi)-\ell(x,\xi)-\nabla \ell(x,\xi)\trans(y-x)| \le  \frac{c_3}{2} \|x-y\|^2.
$
 We note that \ref{ass:two-sides-quad} is not satisfied by {\sgd} when the loss $\ell(\cdot,\xi)$ is also non-smooth. Unfortunately, there seems to be little hope to accelerate {\sgd} in such a case since the convergence rate of {\sgd} already matches the rate of deterministic subgradient method.
\end{rem}

\begin{algorithm}[tb]
   \caption{Stochastic Model-based Method with Minibatches ({\smod})\label{alg:mini-batch}}
   \label{alg:mini-batch}
\begin{algorithmic}
   \STATE {\bfseries Input:}  $x^{1}, \gamma_k$;
   \FOR{$k=1$ {\bfseries to} $K$}
   \STATE Sample a minibatch $B_{k}=\{\xi_{k,1},\ldots,\xi_{k,m_{k}}\}$ and update $x^{k+1}$ by solving
   \begin{equation}
\min_{x\in\Xcal}\, \left\{\frac{1}{m_{k}}\sum_{i=1}^{m_{k}}f_{x^{k}}\big(x,\xi_{k,i}\big)+\frac{\gamma_{k}}{2}\big\Vert x-x^{k}\big\Vert^{2} \right\} \label{eq:mb-prox-map}
\end{equation}
   \ENDFOR
\end{algorithmic}
\end{algorithm}

We present an improved complexity analysis of {\smod} by leveraging the framework of algorithm stability \citep{bousquet2002stability,shalev2010learnability}.
In stark contrast to its standard application in characterizing the algorithm generalization performance, stability analysis is applied to determine how the variation of a minibatch affects the
\emph{estimation of the model function} in each algorithm iteration.

\textbf{Notations.} Let $B=\{\xi_{1},\xi_{2},\ldots,\xi_{m}\}$
be a batch of i.i.d. samples and $B_{(i)}=B\setminus \{\xi_i\}\cup \{\xi_{i}^{\prime}\}$
by replacing $\xi_{i}$ with an i.i.d. copy $\xi_{i}^{\prime}$, and $B^{\prime}= \{\xi_{1}^{\prime},\xi_{2}^{\prime},\ldots,\xi_{m}^{\prime}\}$.
Let $h(\cdot,\xi)$ be a stochastic model function, and denote $h(y,B)=\frac{1}{m}\sum_{i=1}^{m}h(y,\xi_{i})$.
The stochastic proximal mapping associated with $h(\cdot, B)$ is defined by $\prox_{\rho h}(x,B)\triangleq\argmin_{y\in\Xcal}\big\{ h(y,B)+\frac{1}{2\rho}\|y-x\|^{2}\big\}$
for some $\rho>0$. We denote $x_B^+\triangleq\prox_{\rho h}(x,B)$ for brevity.
We say that the stochastic proximal mapping $\prox_{\rho h}$ is \textit{$\varepsilon$-stable}
if, for any $x \in \Xcal$, we have
\begin{equation}
\big|\Ebb_{B,B^{\prime},i}\big[h(x_{B_{(i)}}^+,\xi_{i}^{\prime})-h(x_B^+,\xi_{i}^{\prime})\big]\big|\leq\varepsilon,
\label{eq:def-prox-stability}
\end{equation}

where $i$ is an index chosen from $\{1,2,\ldots,m\}$ uniformly
at random.

The next lemma exploits the stability of proximal mapping associated with
the model function.
\begin{lem}
\label{lem:model-uni-stable} Let $f_{z}(\cdot,B)$ be a stochastic
model function under the assumptions\,\ref{ass:model-weakly-cvx}-\ref{ass:model-lip}.
For $\gamma\in(\lambda,\infty)$, vectors $z$ and $y$, the proximal mapping
$\prox_{f_{z}/\gamma}(y,B)=\argmin_{x\in\Xcal}\big\{ f_{z}(x,B)+\frac{\gamma}{2}\|x-y\|^{2}\big\}$
is $\vep$-stable with
$
\vep=\frac{2L^{2}}{m(\gamma-\lambda)}.
$
\end{lem}
Applying Lemma~\ref{lem:model-uni-stable}, we  obtain the  error bound for approximating the full model function  in the next theorem.
\begin{thm}
\label{thm:stable-mod-func}Under all the assumptions of Lemma~\ref{lem:model-uni-stable}, we have
\begin{equation}
\big\vert\mathbb{E}_{B_{k}}\big[f_{x^{k}}(x^{k+1},B_{k})-\mathbb{E}_{\xi}f_{x^{k}}(x^{k+1},\xi)\vert\sigma_{k}\big]\big\vert\leq\vep_{k},\
\varepsilon_{k}=\tfrac{2L^{2}}{m_{k}(\gamma_{k}-\lambda)}.\label{eq:epsi-stable}
\end{equation}
where $\sigma_{k}$ is the $\sigma$-algebra generating $\{B_{i}\}_{1\le i \le k-1}$.
\end{thm}
Note that since $x^{k+1}$ is dependent on $B_k$, $f_{x^{k}}(x^{k+1},B_{k})$ is not an unbiased estimator of  $\mathbb{E}_{\xi}[f_{x^{k}}(x^{k+1},\xi)]$.
	However, the stability argument identifies that the expected approximation error is a decreasing function of batch size $m_k$. This observation is the key to the sharp analysis of minibatch stochastic algorithms. With all the tools at our hands, we obtain the key descent property in
the following theorem.
\begin{thm}
\label{thm:main-mbsmod} Suppose that $\rho>\lambda+\tau$, $\gamma_{k}\ge\rho+\tau$, ~\ref{ass:two-sides-quad} and all the assumptions in Lemma~\ref{lem:model-uni-stable} hold. Let $\Expe_{k}[\cdot]$ abbreviates $\Ebb_{B_{k}}\big[\cdot|\sigma_{k}\big]$
and $\vep_{k}$ be given by\,(\ref{eq:epsi-stable}),
then we have 
\begin{equation}
\frac{(\rho-\lambda-\tau)}{\rho(\gamma_{k}+\rho-2\lambda-\tau)}\|\nabla f_{1/\rho}(x^{k})\|^{2}\le f_{1/\rho}(x^{k})-\mathbb{E}_{k}\big[f_{1/\rho}(x^{k+1})\big]+\frac{\rho\vep_{k}}{\gamma_{k}+\rho-2\lambda-\tau}.\label{eq:main-bound-mini}
\end{equation}
\end{thm}

Next, we specify the rate of convergence to stationarity using
a constant stepsize policy.
\begin{thm}
\label{thm:rate-mb-smod}Under the assumptions of Theorem\,\ref{thm:main-mbsmod},
let $\Delta=f_{1/\rho}(x^{1})-\min_{x}f(x)$, $m_{k}=m$, and $\gamma_{k}=\gamma=\max\{\rho+\tau,\lambda+\eta\}$
where $\eta=\frac{\sqrt{K}}{\alpha_0\sqrt{m}}$ and $\alpha_0\in (0,\infty)$.
Let $k^{*}$ be an index chosen in $\{1,2,\ldots,K\}$ uniformly, then we have
\begin{equation}
\Ebb\big[\|\nabla f_{1/\rho}(x^{k^{*}})\|^{2}\big]\le\frac{\rho}{\rho-\lambda-\tau}\bigg[\frac{(2\rho-\lambda)\Delta}{K}+ \Big(\frac{\Delta}{\alpha_0}+2\alpha_0\rho L^2\Big) \frac{1}{\sqrt{mK}}\bigg].\label{eq:mb-mid-04}
\end{equation}
\end{thm}

\begin{rem}
The performance of {\smod} depends on $\alpha_0$ and batch size $m$.  
\eqref{eq:mb-mid-04} implies that when batch size is fixed, the best rate is obtained at $\alpha_0^*=\sqrt{\tfrac{\Delta}{2\rho}}\frac{1}{L}$.  
Since both $\Delta$ and $L$ are unknown, hyper-parameter tuning over $\alpha_0$ is required to obtain good empirical performance.  For the simplicity of theoretical analysis, let us take $\alpha_0=\alpha_0^*$. 
Hence, to obtain an iterate whose Moreau envelop has expected gradient norm smaller than $\vep$, the total iteration count
is
$
\Tcal_\vep=\max \big\{\Ocal(\frac{\Delta}{\vep^{2}}),\Ocal(\frac{L^{2}\Delta}{m\vep^{4}})\big\}.
$
For small batch size $m$ (i.e. $m=o(1/\vep^2)$), 
the second term in $\max(,)$ dominates the bound $\Tcal_\vep$, yielding a total complexity of $\mathcal{O} (
{\frac{L^2 \Delta}{m\varepsilon^4}} )$. Note that this complexity bound is better than the $\mathcal{O} (
{\frac{L^2 \Delta}{\varepsilon^4}} )$ bound~\citep{davis2019stochasticweakly} by a factor of $m$.
\end{rem}
\begin{rem}
	Theorem~\ref{thm:rate-mb-smod} implies that {\sgd} can be accelerated by minibatching on the smooth composite problems (\ref{eq:model-linear}) but leaves out the more general problems where $\ell(x,\xi)$ is non-smooth and weakly convex. In the latter case, showing any improved rate of minibatch {\sgd} is substantially more challenging. Without additional knowledge, the  $\mathcal{O} (\frac{L^2 \Delta}{\varepsilon^4})$ complexity of {\sgd}  already matches the  best result for deterministic subgradient method (c.f. \citep{davis2019stochasticweakly}). It remains unknown whether such $\Ocal(1/\vep^4)$ bound is tight or not, and a possible direction to obtain sharper complexity bound is by exploiting the non-smooth structure information such as  sharpness.
\end{rem}
\textbf{Solving the subproblems.} {\sgd} is embarrassingly parallelizable by simply averaging the stochastic subgradients. We highlight how to solve the proximal subproblems for {\spl} and {\spp}. Consider the composition function $f(x,\xi)=h(C(x,\xi))$ where $h(a)=|a|$. For {\spl}, it is easy to transform the corresponding subproblem to an $\Ocal(m_k)$-dimensional quadratic program (QP) in the dual space (e.g.~\citep{asi2020minibatch}). The dual QP can be efficiently solved in parallel, for example, by a fast interior point solver.  For {\spp}, we show that the subproblem can be solved by a deterministic prox-linear method at a rapid linear convergence rate. Note that the {\spp} subproblem is especially well-conditioned because our stepsize policy ensures a large strongly convex parameter $\gamma-\lambda$.
	We refer to the appendix for more technical details. 

\section{{\smod} with momentum\label{sec:momentum}}
We present a new model-based method by incorporating
an additional extrapolation term, and we record this stochastic extrapolated
model-based method in Algorithm\,\ref{alg:semod}. Each iteration
of Algorithm\,\ref{alg:semod} consists of two steps, first, an extrapolation step
is performed to get an auxiliary update $y^{k}$. Then a random sample $\xi_k$ is collected and the proximal
mapping, associated with the model function $f_{x^k}(\cdot,\xi_k)$,  is computed at $y^{k}$ to obtain the new point $x^{k+1}$.
For ease of exposition, we take constant values of stepsize and extrapolation term.
\begin{algorithm}[tb]
   \caption{Stochastic Extrapolated Model-Based Method~({\extra})}\label{alg:semod}
\begin{algorithmic}
   \STATE {\bfseries Input:} $x^{0}$, $x^{1}$, $\beta$, $\gamma$;
   \FOR{$k=1$ {\bfseries to} $K$}
   \STATE Sample data $\xi^{k}$ and update:
   \begin{align}
y^{k} & =x^{k}+\beta(x^{k}-x^{k-1}) \label{eq:update-yk}\\
x^{k+1} & =\argmin_{x\in\Xcal} \left\{ f_{x^{k}}(x,\xi^{k})+\frac{\gamma}{2}\|x-y^{k}\|^{2} \right\} \label{eq:update-xk}
\end{align}
   \ENDFOR
\end{algorithmic}
\end{algorithm}

Note that Algorithm\,\ref{alg:semod} can be interpreted as an extension
of the momentum {\sgd}  by replacing the gradient descent step
with a broader class of proximal mappings.
To see this intuition, we combine (\ref{eq:update-yk}) and (\ref{eq:update-xk})
to get
\ifOneCol
\begin{equation}\label{eq:combined}
x^{k+1} =\argmin_{x\in\Xcal}\Big\{ f_{x^{k}}(x,\xi^{k})+\gamma\beta\langle x^{k-1}-x^{k},x-x^{k}\rangle +\frac{\gamma}{2}\|x-x^{k}\|^{2}\Big\},
\end{equation}
\else
\begin{align}
x^{k+1} & =\argmin_{x\in\Xcal}\ f_{x^{k}}(x,\xi^{k})+\gamma\beta\langle x^{k-1}-x^{k},x-x^{k}\rangle \nonumber\\
& \quad \quad \quad \quad  +\frac{\gamma}{2}\|x-x^{k}\|^{2}, \label{eq:combined}
\end{align}
\fi
If we choose the linear model (\ref{eq:model-linear}), i.e.,
$f_{x^{k}}(x,\xi^{k})=f(x^k,\xi^{k})+\langle f^{\prime}(x^k,\xi^{k}),x-x^{k}\rangle,$
and assume $\Xcal=\Rbb^{d}$, then the update\,(\ref{eq:combined})
has the following form:
\begin{equation}
x^{k+1}=x^{k}-\gamma^{-1}f^{\prime}(x^k,\xi^{k})-\beta(x^{k-1}-x^{k}).\label{eq:extra-mid-13}
\end{equation}
Define $v^{k}\triangleq\gamma(x^{k-1}-x^{k})$ and apply it to (\ref{eq:extra-mid-13}),
then Algorithm\,\ref{alg:semod} reduces to the heavy-ball method
\begin{align}
v^{k+1} & =f^{\prime}(x^k,\xi^{k})+\beta v^{k},\label{eq:hb1}\\
x^{k+1} & =x^{k}-\gamma^{-1}v^{k+1}.\label{eq:hb2}
\end{align}
Despite such relation, the gradient averaging view~(\ref{eq:hb1}) only applies to {\sgd} for unconstrained optimization, which limits the use of standard analysis of heavy-ball method (\citep{yan2018a}) for our problem.
To overcome this issue, we present a unified convergence analysis which can deal with all the model functions and is amenable to both constrained and composite problems.

Our theoretical analysis of  Algorithm\,\ref{alg:semod}
relies on a different potential function from the one in the previous section. Let us define the auxiliary variable
\begin{equation}
\label{eq:zk}
z^{k}\triangleq x^{k}+\frac{\beta}{1-\beta}(x^{k}-x^{k-1}).
\end{equation}
The following lemma proves some approximate descent property by adopting
the potential function 
\label{eq:fzk-potential}$
f_{1 / \rho} (z^k)  +\frac{\rho(\gamma\beta+\rho\beta^{2}\theta^{-2})}{2(\gamma\theta-\lambda\theta)}\|x^{k}-x^{k-1}\|^{2}$
and measuring the quantity of $\|\nabla f_{1/\rho}(z^{k})\|$.
\begin{lem}
\label{lem:extra-2} Assume that $\rho\ge2(\tau+\lambda)$ and $\beta\in[0,1)$.
Let $\theta=1-\beta$. Then we have 
\begin{align}\label{eq:extra-decent-prop}
 \frac{(\rho-\lambda\theta)}{2\rho(\gamma\theta-\lambda\theta)}\|\nabla f_{1/\rho}(z^{k})\|^{2} 
 & \le f_{1/\rho}(z^{k})-\Ebb_{k}\big[f_{1/\rho}(z^{k+1})\big]+\frac{\rho L^{2}}{(\gamma\theta^{2}-\rho\beta^{2}\theta^{-1})(\gamma\theta^{2}-\lambda\theta^{2})}\nonumber \\
&\quad  +  \frac{\rho(\gamma\beta+\rho\beta^{2}\theta^{-2})}{2(\gamma\theta-\lambda\theta)}\big(\|x^{k}-x^{k-1}\|^{2}-\Ebb_{k}[\|x^{k+1}-x^{k}\|^{2}]\big)\nonumber\\
& - \frac{\rho(\gamma-\rho\beta^{2}\theta^{-3})}{4(\gamma-\lambda)} \Ebb_{k}[\|x^{k+1}-x^{k}\|^{2}].
\end{align}
\end{lem}
Invoking Lemma\,\ref{lem:extra-2} and specifying the stepsize policy,
we obtain the main convergence result of Algorithm\,\ref{alg:semod}
in the following theorem.
\begin{thm}
\label{thm:extra-2-1}Under assumptions of Lemma\,\ref{lem:extra-2},
if we choose $x^{1}=x^{0}$, and set 
$
\gamma=\gamma_{0}\theta^{-1}\sqrt{K}+\lambda+\rho\beta^{2}\theta^{-3}
$
for some $\gamma_{0}>0$, then 
\begin{equation}
\Ebb[\|\nabla f_{1/\rho}(z^{k^{*}})\|^{2}]\le\frac{2\rho}{\rho-\lambda}\bigg[\frac{\rho\beta^{2}\theta^{-2}\Delta}{K}+\Big(\gamma_{0}\Delta+\frac{\rho L^{2}}{\theta\gamma_{0}}\Big)\frac{1}{\sqrt{K}}\bigg]\label{eq:extra-main-3}
\end{equation}
where $k^{*}$ is an index chosen in $\{1,2,\ldots,K\}$ uniformly
at random.
\end{thm}

\begin{rem}
Despite the fact that convergence is established for all $\gamma_{0}>0$,
we can see that the optimal $\gamma_{0}$ would be $\gamma_{0}=\sqrt{\frac{\rho}{\Delta\theta}}L$,
which gives the bound 
$
\Ebb[\thinspace\|\nabla f_{1/\rho}(z^{k^{*}})\|^{2}]\le\frac{2\rho}{\rho-\lambda}\big(\frac{\rho\beta^{2}\theta^{-2}\Delta}{K}+2L\sqrt{\frac{\rho \Delta}{\theta K}}\big).
$
In practice, we can set $\gamma_{0}$ to a suboptimal value and obtain a possibly loose upper-bound.
\end{rem}
\begin{rem}
Since $z^k$ is an extrapolated solution, it may not be feasible. It is desirable to show optimality guarantee at iterates $x^k$.  Note that using Lemma~\ref{lem:extra-2} and the parameters in Theorem~\ref{thm:extra-2-1}, it is easy to show that $\Ebb [\|x^{k^*}-x^{k^*-1}\|^2]=\Ocal(\frac{1}{K})$. 
Based on (\ref{eq:zk}) we have $\|z^{k^*}-x^{k^*}\|^2=\beta^2\theta^{-2}\Ebb [\|x^{k^*}-x^{k^*-1}\|^2]=\Ocal(\frac{1}{K})$. 
Using Lipschitz smoothness of Moreau envelop, we can show $\Ebb[\|\nabla f_{1/\rho}(x^{k^*})\|^2]$ converges at the same $\Ocal(\frac{1}{\sqrt{K}})$ rate as is shown in Theorem~\ref{thm:extra-2-1}.
\end{rem}
Combining momentum and minibatching,  we develop a minibatch version of Algorithm~\ref{alg:semod} that takes a batch of samples $B_k$ in each iteration. The convergence analysis of this minibatch {\extra} is more involving. We leave the details in the Appendix but informally state the main result below.
\begin{thm}[Informal]
In the minibatch {\extra}, suppose that \ref{ass:two-sides-quad} holds, the batch size $|B_k|=m$ and  $\gamma=\Ocal(\sqrt{\frac{K}{m}})$, then $\Ebb[\|\nabla f_{1/\rho}(z^{k^{*}})\|^{2}]=\Ocal\brbra{\frac{1}{K}+\sqrt{\frac{1}{mK}}}$.
\end{thm}

\section{{\smod} for convex optimization}
Besides the study on non-convex optimization,  we also apply model-based methods to stochastic convex optimization.
Due to the space limit, we highlight main theoretical results but defer all the technical details to the Appendix section.
We show that if certain  assumption adapted from \ref{ass:two-sides-quad} for the convex setting holds,  the function gap of minibatching {\extra} will converge at a rate of $
\Ocal\Big(\frac{1}{K}+\frac{1}{\sqrt{mK}}\Big).$ In view of this result, the deterministic part of our rate is consistent with the best  $\Ocal{(\frac{1}{K}})$ rate for the heavy-ball method. For example, see \citep{diakonikolas2021generalized,ghadimi2015global}.
Moreover, the stochastic part of the rate is improved from the  $\Ocal(\frac{1}{\sqrt{K}})$ rate of Theorem~4.4~\citep{davis2019stochasticweakly} by a factor of $\sqrt{m}$.

An important question arises naturally: Can we further improve the  convergence rate of model-based methods for stochastic convex optimization?  Due to the widely known limitation of heavy-ball type momentum, it would be interesting to consider Nesterov's acceleration.
To this end, we present a model-based method with Nesterov type momentum. Thanks to the stability argument, we obtain the following improved rate of convergence: 
$
\Ocal\Big(\frac{1}{K^2}+\frac{1}{\sqrt{mK}}\Big).
$
We note that a similar convergence rate for minibatching model-based methods is obtained in a recent paper~\citep{chadha2021accelerated}. However, their result requires the assumption that the stochastic function is Lipschitz smooth while our assumption is much weaker.

\section{Experiments\label{sec:Experi}}
In this section, we examine the empirical performance of our proposed methods through experiments on the problem of robust phase retrieval. (Additional experiments on 
blind deconvolution are given in Appendix section).
 Given a set of vectors $a_i\in \Rbb^d$ and nonnegative scalars $b_i\in\Rbb_+$,  the goal of phase retrieval is to recover the true signal $x^*$ from the measurement $b_i=|\langle a_i, x^*\rangle|^2$. 
Due to the potential corruption in the dataset, we consider the following penalized formulation 
\begin{mini}
{x\in\Rbb^d}{ \frac{1}{n} \sum_{i = 1}^n \big\vert \langle a_i,  x \rangle^2 -
  b_i \big\vert }{}{}\label{pb:phase}
\end{mini}
where we impose $\ell_1$-loss to promote robustness and stability (cf. \cite{duchi2019solving, davis2019stochasticweakly,pmlr-v119-mai20b}).

\textbf{Data Preparation.} We conduct experiments on both synthetic and real datasets.

\textbf{1) Synthetic data.} Synthetic data is generated following the setup in \cite{pmlr-v119-mai20b}.
We set $n = 300, d = 100$ and select $x^{\ast}$
from unit sphere uniformly at random. Moreover, we generate $A = Q D$ where $Q \in \mathbb{R}^{n
\times d}, q_{i j} \sim \Ncal (0, 1)$ and $D \in \mathbb{R}^d$ is a diagonal matrix whose diagonal entries are evenly distributed in $[1/\kappa, 1]$. Here $\kappa \geq 1$ plays the role of condition number (large $\kappa$ makes problem hard).  The measurements are generated by $b_i = \langle a_i, x^{\ast} \rangle^2 + \delta_i \zeta_i$ ($1\le i\le n$) with $\zeta_i \sim \Ncal (0, 25)$, $\delta_i
\sim \tmop{Bernoulli} (p_{\tmop{fail}})$, where $p_{\tmop{fail}}\in[0,1]$ controls the fraction of corrupted observations on expectation.\

 \textbf{2) Real data.} We consider  \texttt{zipcode}, a dataset of $16 \times 16$ handwritten digits collected from \cite{RN113}. Following the setup in \cite{duchi2019solving},  let $H\in \mathbb{R}^{256 \times 256}$ be a normalized Hadamard matrix such that $h_{ij} \in \left\{\frac{1}{16}, -\frac{1}{16} \right\}, H=H{\trans}$ and $H=H^{-1}$. Then we generate $k=3$ diagonal sign matrices $S_1, S_2, S_3$ such that each diagonal element of $S_k$ is uniformly sampled from $\{-1, 1\}$. Last we set $A = \left[H S_1, HS_2, HS_3 \right]{\trans} \in \mathbb{R}^{(3\times 256)\times 256}$.
As for the true signal and measurements, each image is represented by a data matrix $X\in \mathbb{R}^{16\times16}$ and gets vectorized to $x^{\ast}=\text{vec}(X)$. To simulate the case of corruption, we set measurements $b=\phi_{p_{\text{fail}}}(Ax^*)$, where $\phi_{p_{\text{fail}}}(\cdot)$ denotes element-wise squaring and setting a fraction $p_\text{fail}$ of entries to $0$ on expectation.

In the first experiment, we illustrate that {\smod} methods enjoy linear speedup in the size of minibatches and exhibit strong robustness to the stepsize policy.  We conduct comparison on {\spl} and {\sgd} and describe the detailed experiment setup as follows.

\textbf{1) Dataset generation.} We generate four testing cases: the synthetic datasets with $(\kappa, p_\text{fail})=(10, 0.2)$, and  $(10, 0.3)$; \texttt{zipcode} with digit images of id 2 and 24;

\textbf{2) Initial point.} We set the initial point $x^1(=x^0)\sim\Ncal(0, I_d)$ for synthetic data and $x^1 = x^* + \Ncal(0, I_d)$ for \texttt{zipcode};

\textbf{3) Stopping criterion.} We set the stopping criterion to be $f(x^k) \leq 1.5 \hat{f}$, where $\hat{f} = f(x^*)$ is the corrupted objective evaluated at the true signal $x^*$;

\textbf{3) Stepsize.} We set the parameter $\gamma=\alpha_0^{-1}\sqrt{K/m}$ where $m$ is the batch size;
For synthetic dataset, we test 10 evenly spaced $\alpha_0$ values in range $ [10^{-1}, 10^{2}]$ on logarithmic scale, and for \texttt{zipcode} dataset we set such range of $\alpha_0 $ to $ [10^1, 10^{3}]$;

\textbf{4) Maximum iteration.} We set the maximum number of epochs to be 200 and 400 respectively for minibatch and momentum related tests;

\textbf{5) Batch size.} We take minibatch size $m$ from the range $\{1, 4, 8, 16, 32, 64\}$;

\textbf{6) Sub-problems} The solution to the proximal sub-problems is left in the appendix.

For each algorithm, speedup from minibatching is quantified as $T_1^*/T_m^*$ where $T_m^*$ is the total number of iterations for reaching the desired accuracy, with batch size $m$ and the best initial stepsize $\alpha_0$ among values specified above. Specially, if an algorithm fails to reach desired accuracy after running out of 400 epochs, we set its iteration number to the maximum.

\begin{figure*}[!htb]
\centering
\includegraphics[scale=0.20]{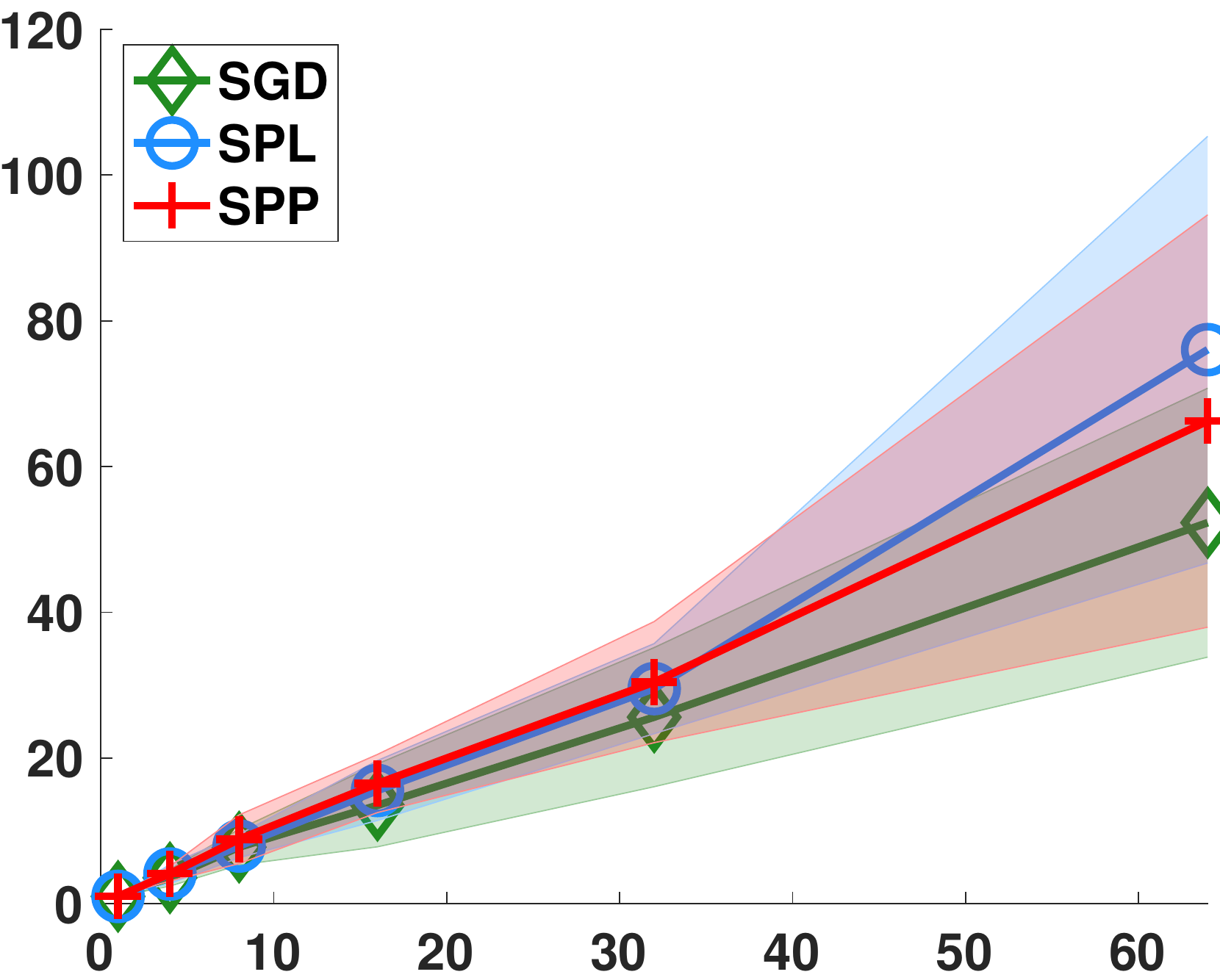} \includegraphics[scale=0.20]{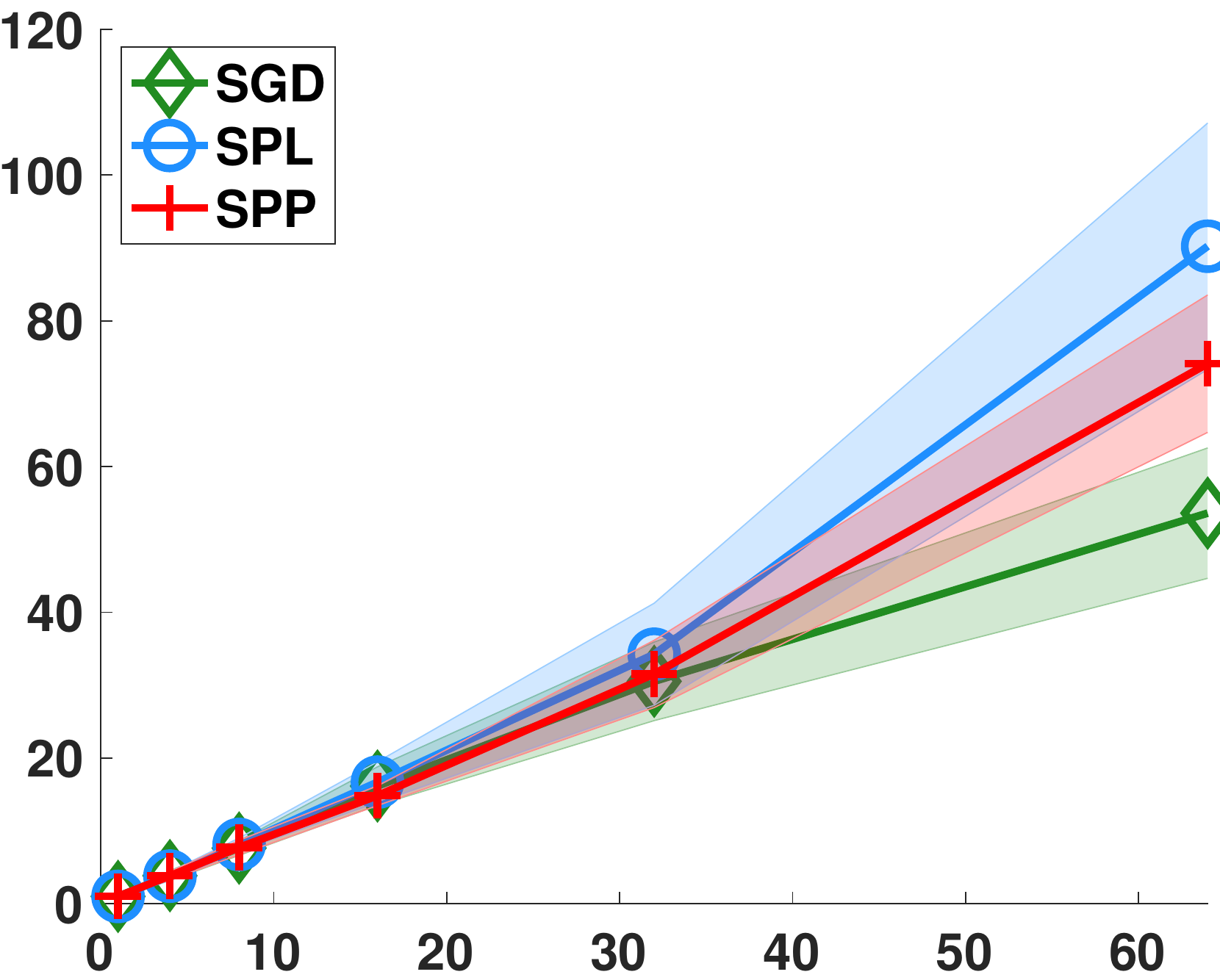}\includegraphics[scale=0.20]{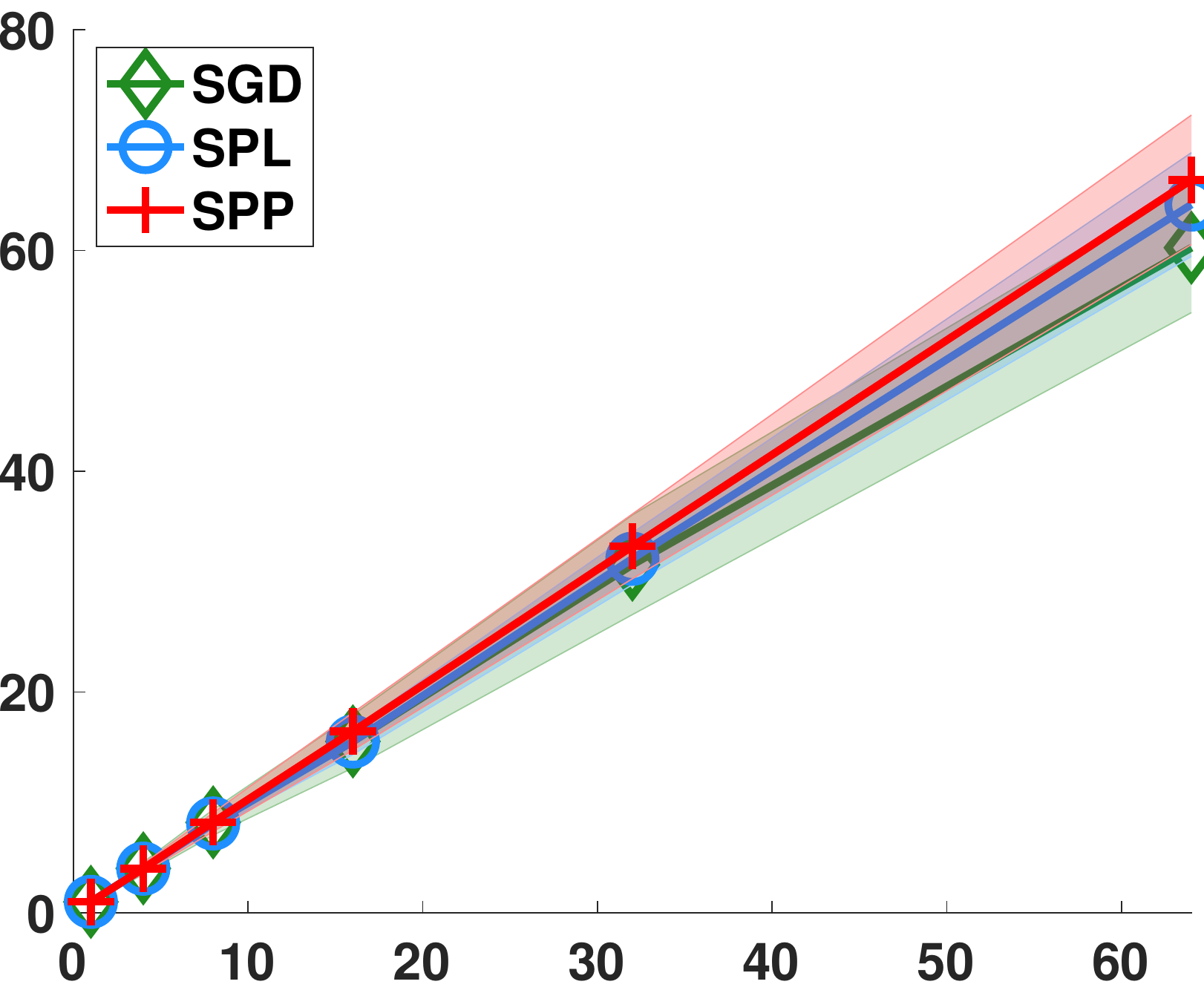}\includegraphics[scale=0.20]{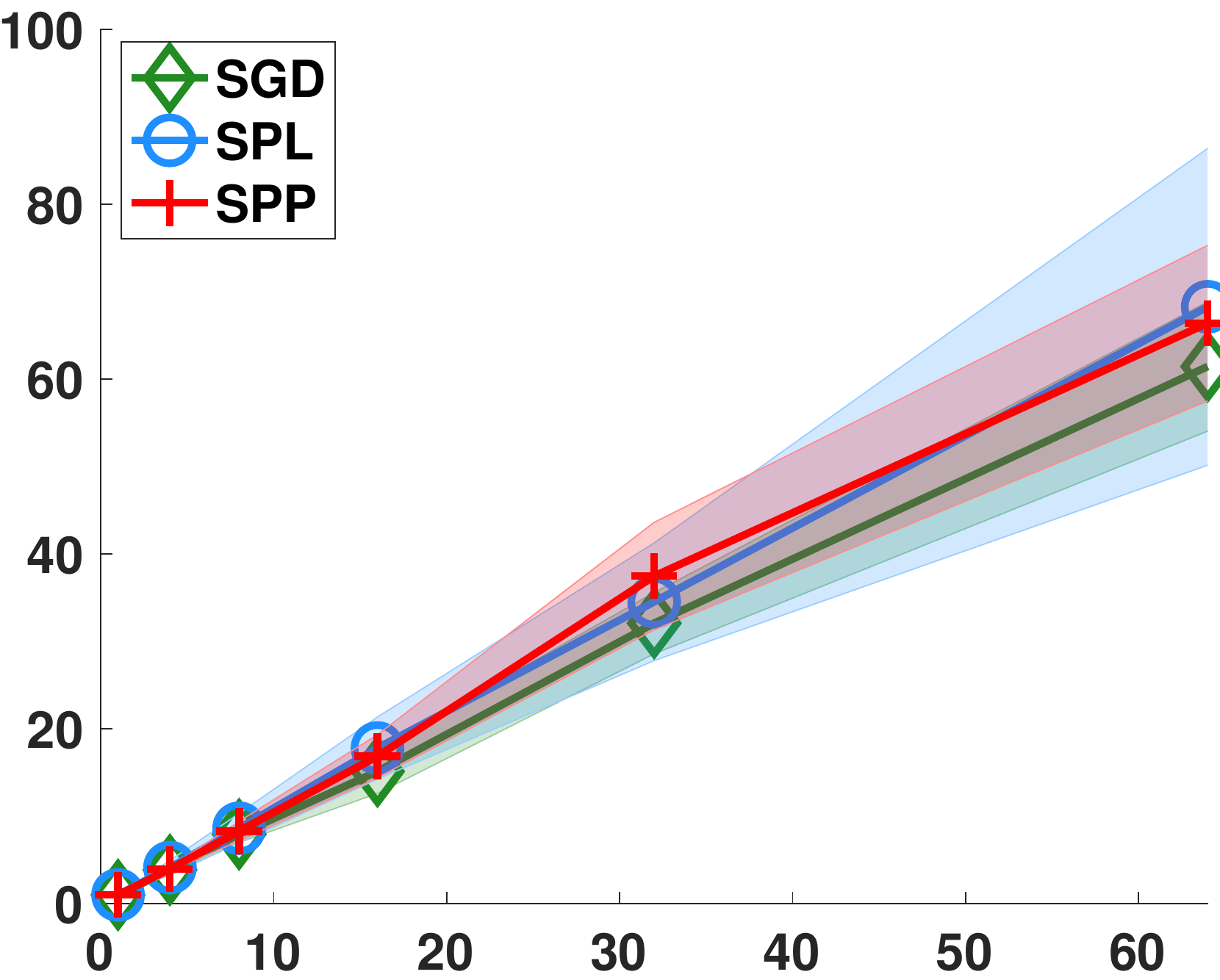}
	\caption{Speedup over minibatch sizes. The left two are for synthetic datasets  $\kappa=10, p_\text{fail} \in \{0.2, 0.3\}$; Digit datasets: digit image (id:24) with $p_{\text{fail}}\in \{0.2, 0.3\}$. \label{exp1:mb-speedup-1}}
\end{figure*}
 
\begin{figure*}[!htb]
\centering
\includegraphics[scale=0.20]{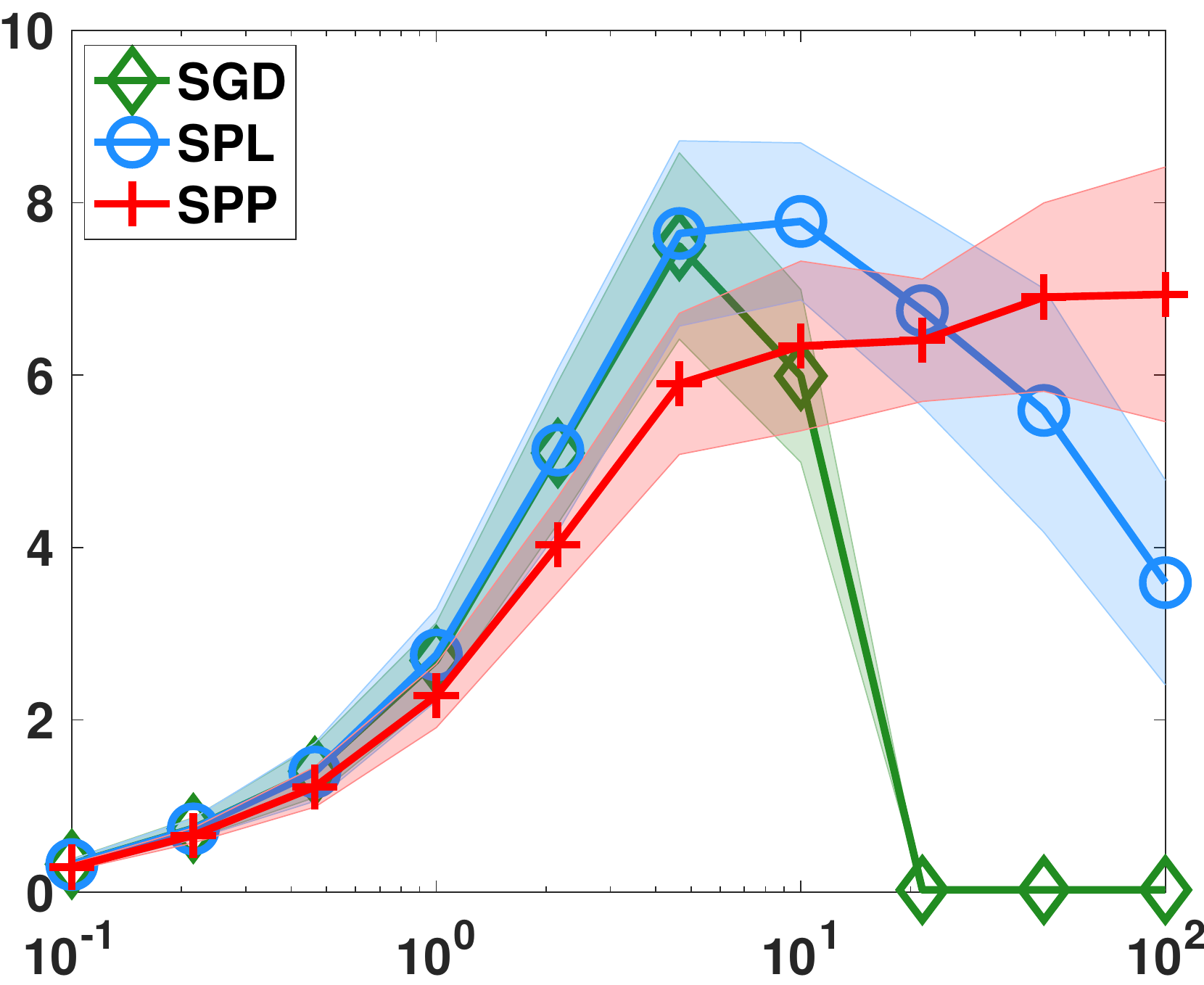} \includegraphics[scale=0.20]{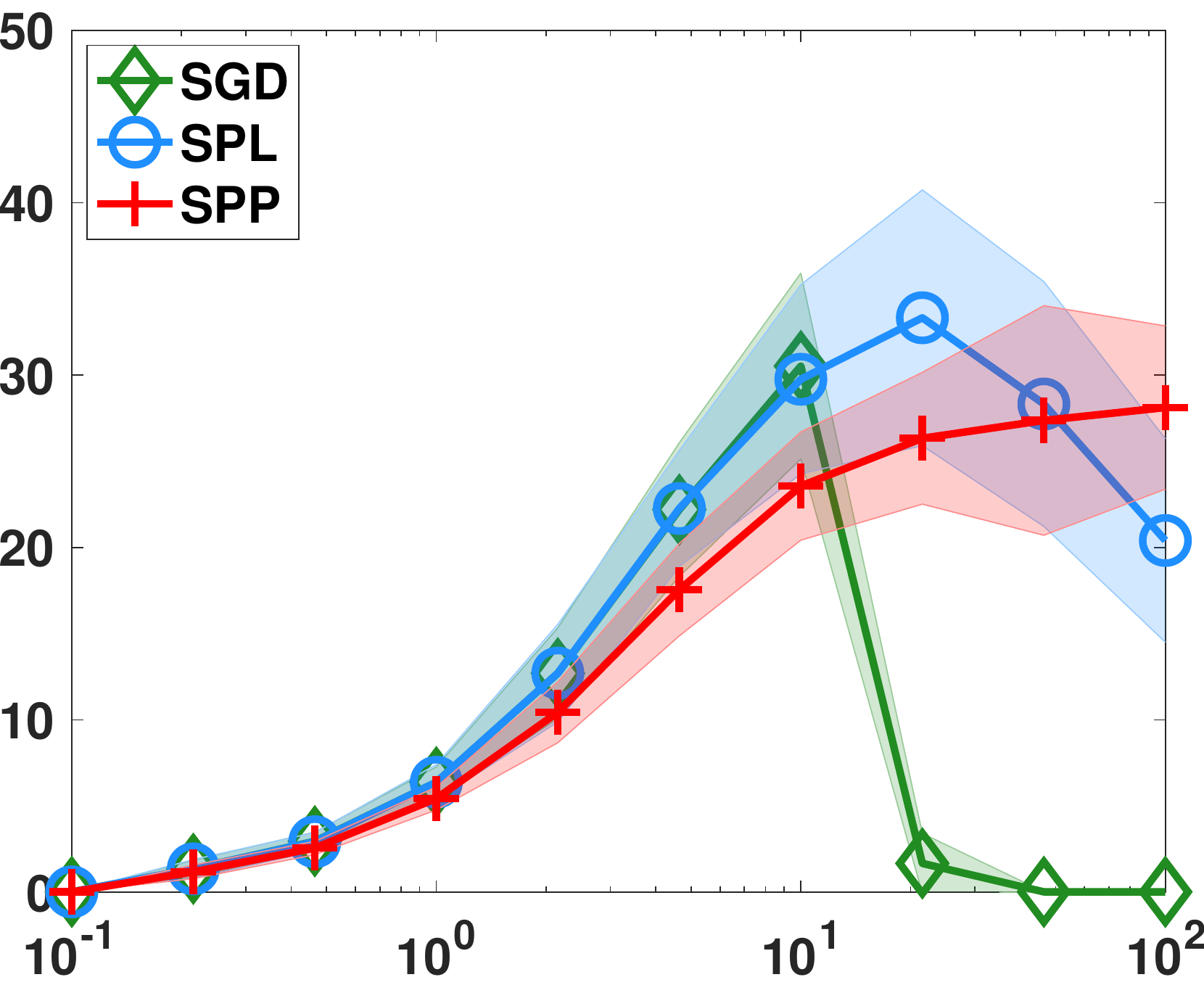}\includegraphics[scale=0.20]{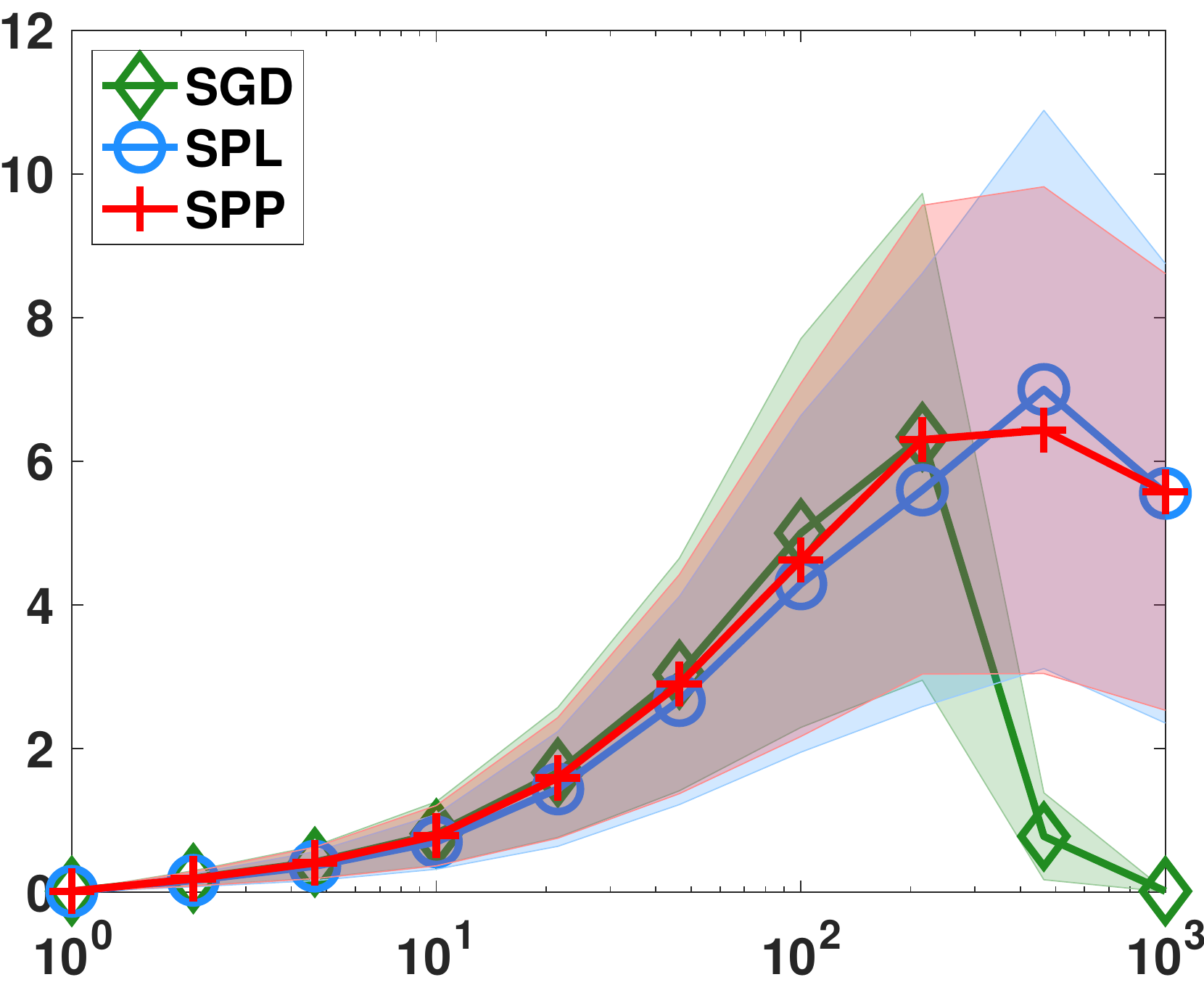}\includegraphics[scale=0.20]{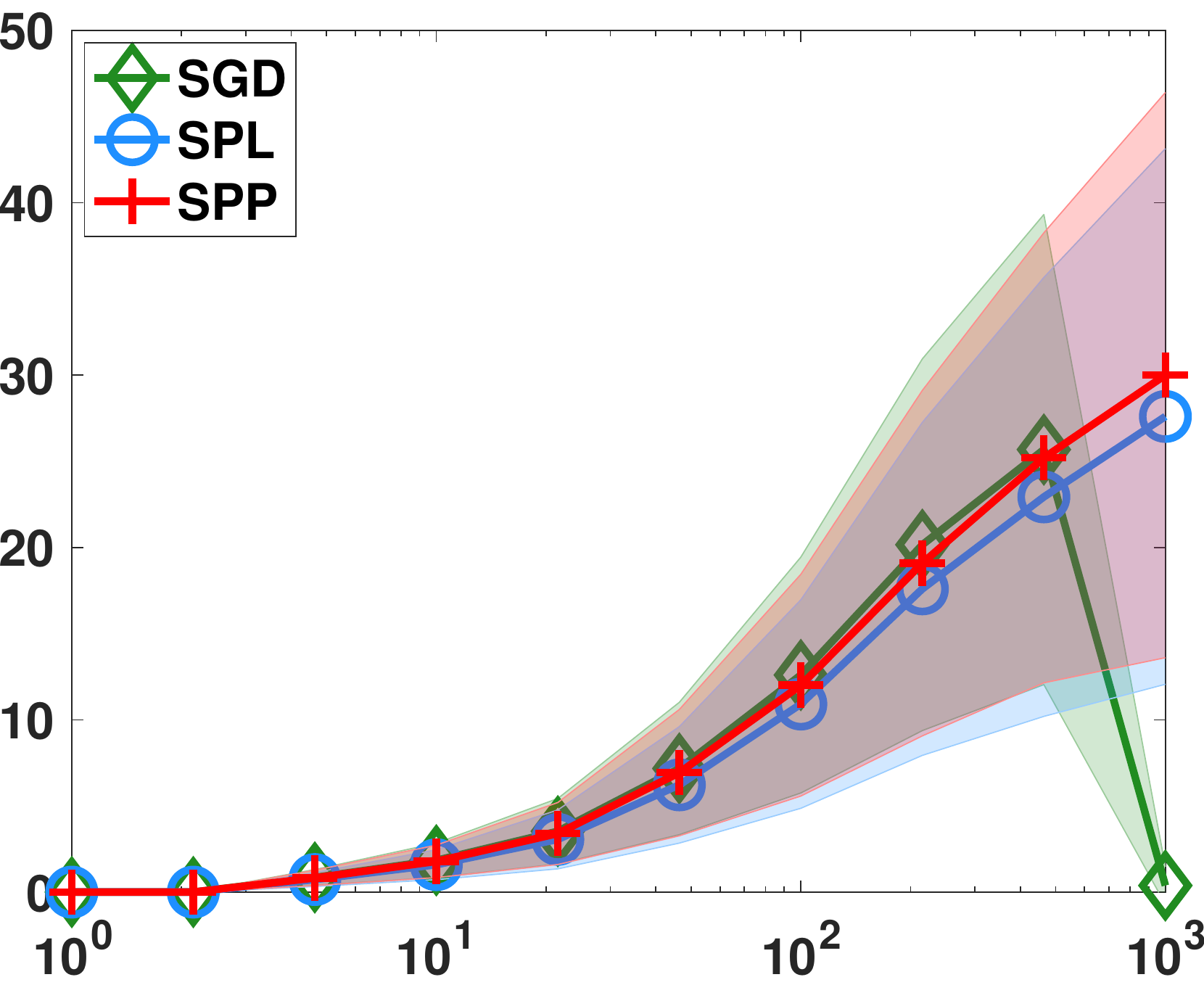}
\includegraphics[scale=0.20]{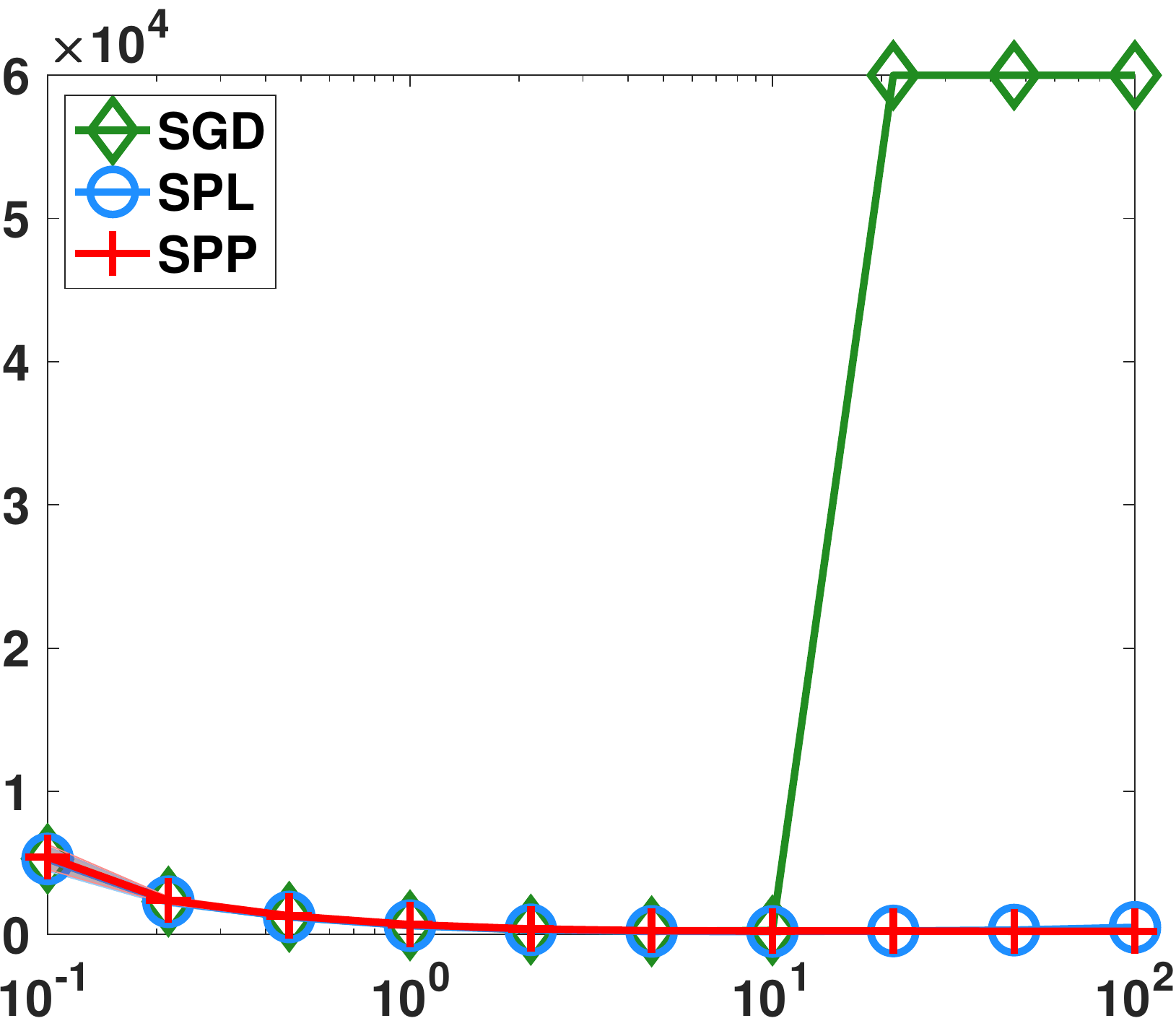} \includegraphics[scale=0.20]{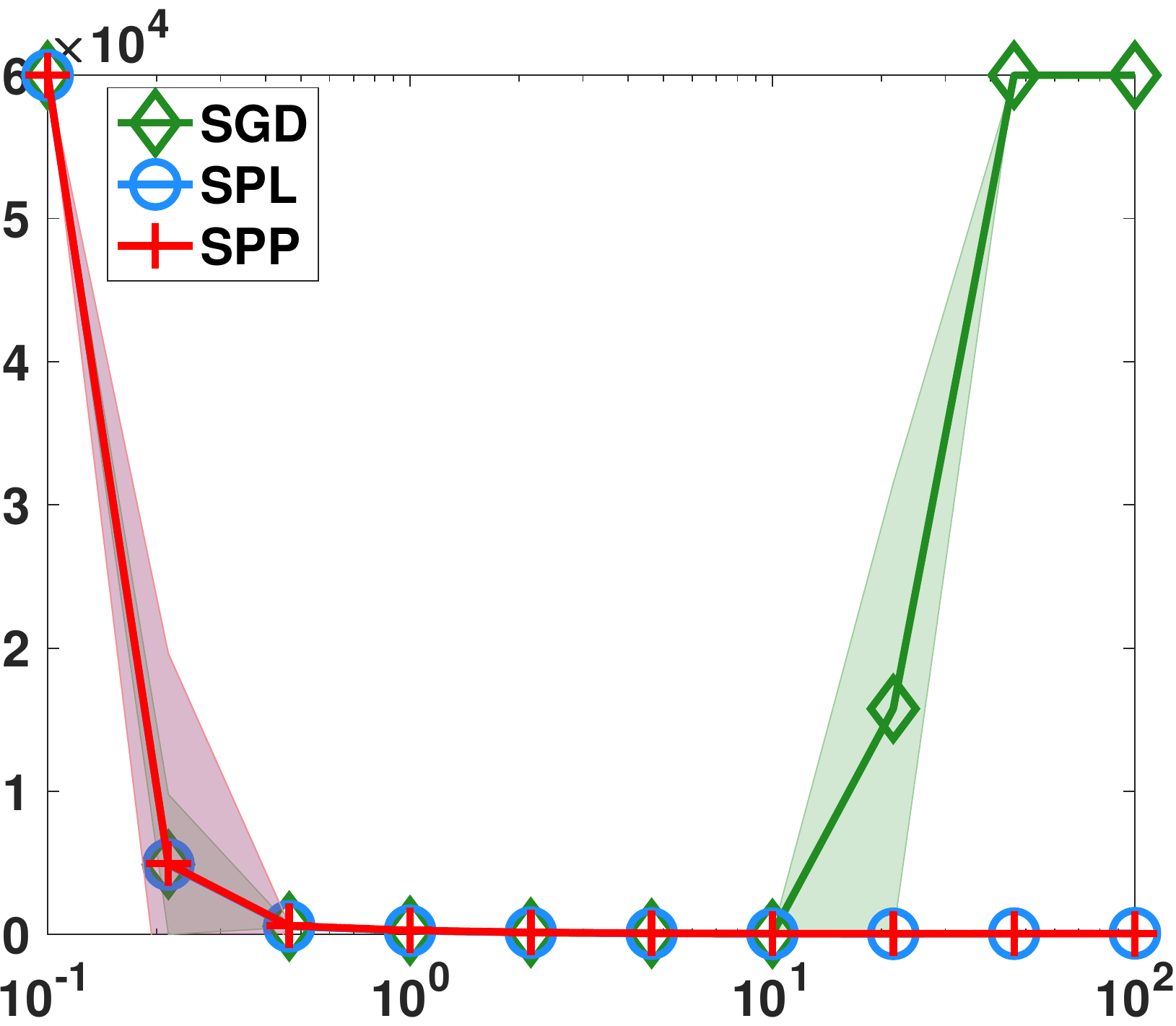}\includegraphics[scale=0.20]{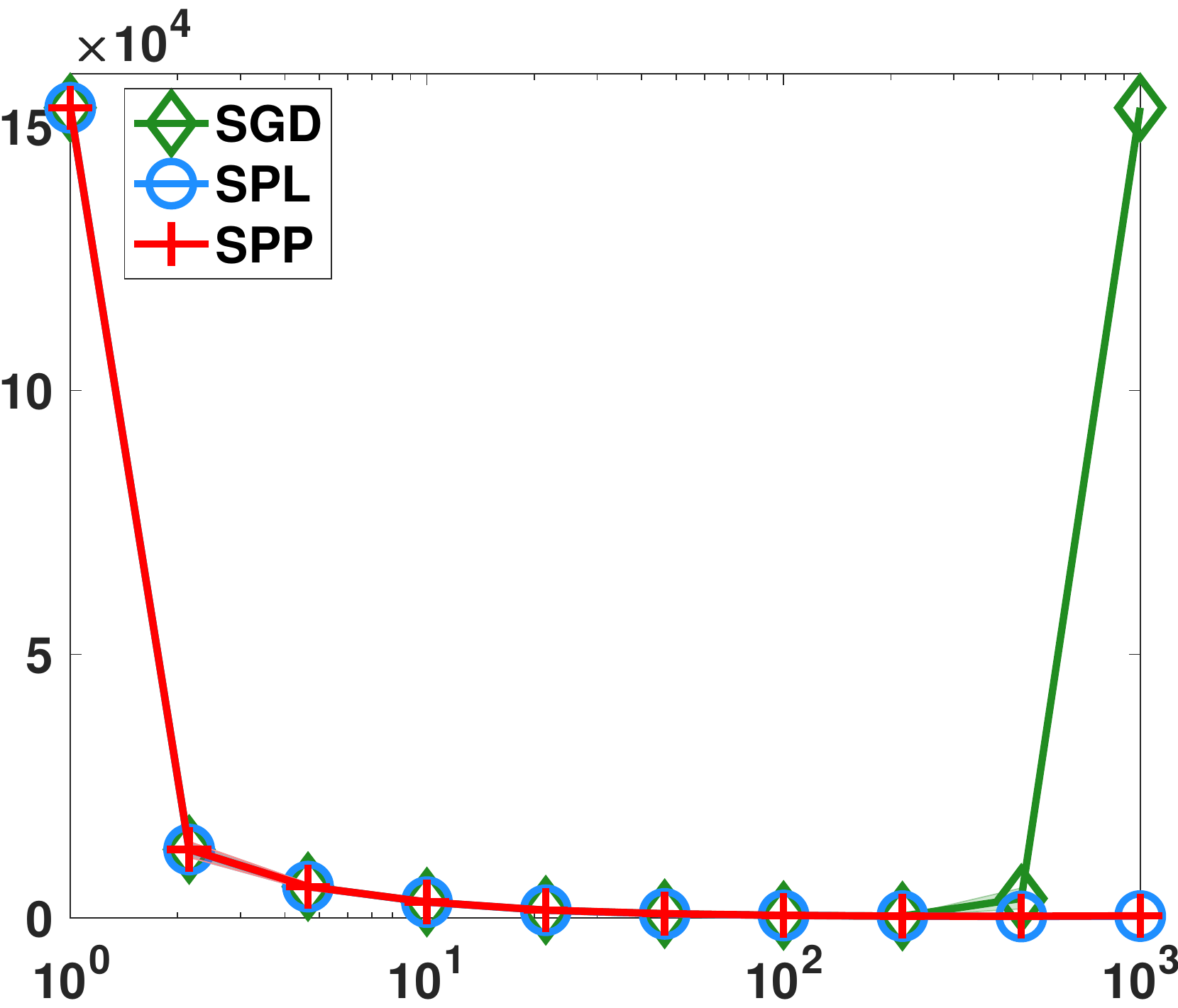}\includegraphics[scale=0.20]{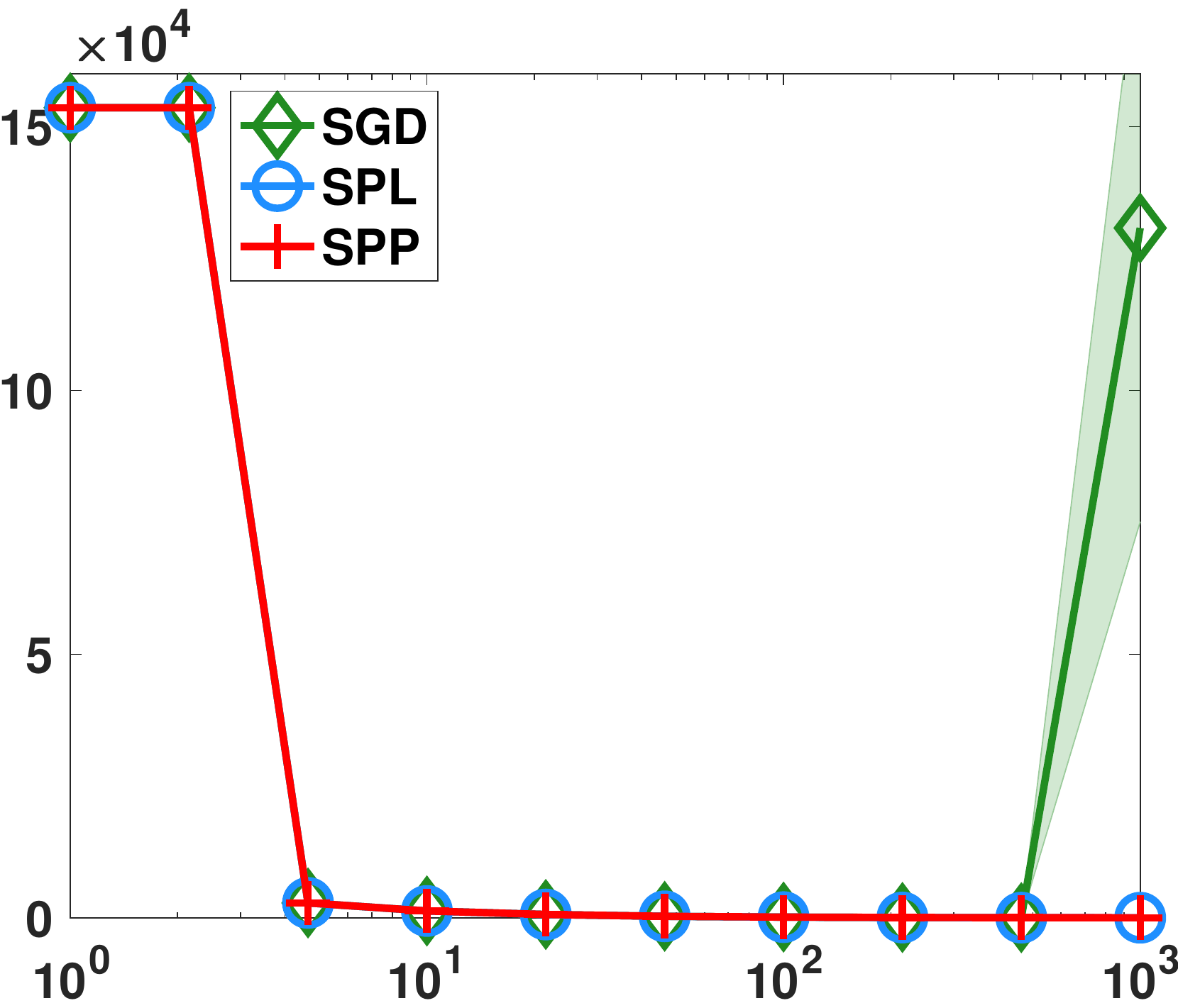}
	\caption{From left to right: synthetic datasets with $m \in \{8, 32\}$ and \texttt{zipcode} image (id=24) with $m \in \{8, 32\}$.
	x-axis: initial stepsize $\alpha_0$. y-axis (first row): speedup over the sequential version: $T_1^*/T_m^*(\alpha_0)$ where $T_m^*(\alpha_0)$ stands for the number of iterations when using batch size $m$ and initial stepsize $\alpha_0$.
	y-axis (second row): Total number of iterations.
	 \label{exp1:speedup-lr}}
\end{figure*}

  Figure~\ref{exp1:mb-speedup-1}  plots the speedup of each algorithm over different values of batch size according to the average of 20 independent runs.
 It can be seen that {\spl} exhibits a linear acceleration over the batch size, which confirms our theoretical analysis.
Moreover, we  find  {\sgd} admits considerable acceleration using minibatches, and sometimes the speedup performance matches that of {\spl} and {\spp}.  This observation seems to suggest the  effectiveness of minibatch {\sgd} in practice, despite the lack of theoretical support.
 
 Next, we investigate the sensitivity of minibatch acceleration to the choice of initial stepsizes.  We  plot the algorithm speedup over the initial stepsize $\alpha_0$ in Figure~\ref{exp1:speedup-lr} (1st row).
 It can be readily seen that {\sgd}, {\spl} and {\spp} all  achieve  considerable minibatch acceleration when choosing the  initial stepsize properly. However, {\spl} and {\spp} enjoy a much wider range of initial stepsizes for good speedup performance, and hence, lays more robust performance than {\sgd}.
 To further illustrate the robustness of {\spl} and {\spp}, we compare the efficiency of both algorithms in the minibatch setting. In contrast to the previous comparison on the relative scale, we directly compare the iteration complexity of the two algorithms. We plot the total iteration number over the choice of initial stepsizes in Figure~\ref{exp1:speedup-lr} (2nd row) for batch size $m=8$ and $ 32$. We observe that minibatch {\spl}({\spp})s exhibits promising performance for a wide range of stepsize policies, while minibatch {\sgd} quickly diverges for large stepsizes. Overall, our experiment complements the recent work \cite{davis2019stochasticweakly}, which shows that {\spl} ({\spp}) is more robust than {\sgd} in the sequential setting.

Our second experiment investigates the  performance of the proposed momentum methods.
We compare three  model-based methods ({\sgd}, {\spl}, {\spp}) and extrapolated model-based methods  ({\segd}, {\sepl}, {\sepp}). We generate four testing cases: the synthetic datasets with $(\kappa, p_\text{fail})=(10, 0.2)$ and $(10, 0.3)$; \texttt{zipcode} with digit images of id 2 and $p_\text{fail} \in \{0.2, 0.3\}$. We set $\alpha_0 \in [10^{-2}, 10^0], \beta = 0.6$ for synthetic data, and  set $\alpha_0 \in [10^0, 10^1], \beta = 0.9$ for \texttt{zipcode} dataset.
The rest of settings are the same as in minibatch with $m=1$.

Figure~\ref{exp4:epoch-lr} plots the number of epochs to $\varepsilon$-accuracy over initial stepsize $a_0$. It can be seen that with properly selected momentum parameters ({\segd}, {\sepl}, {\sepp}) all suggest improved convergence when stepsize is relatively small.

\begin{figure*}[!htb]
\centering
\includegraphics[scale=0.20]{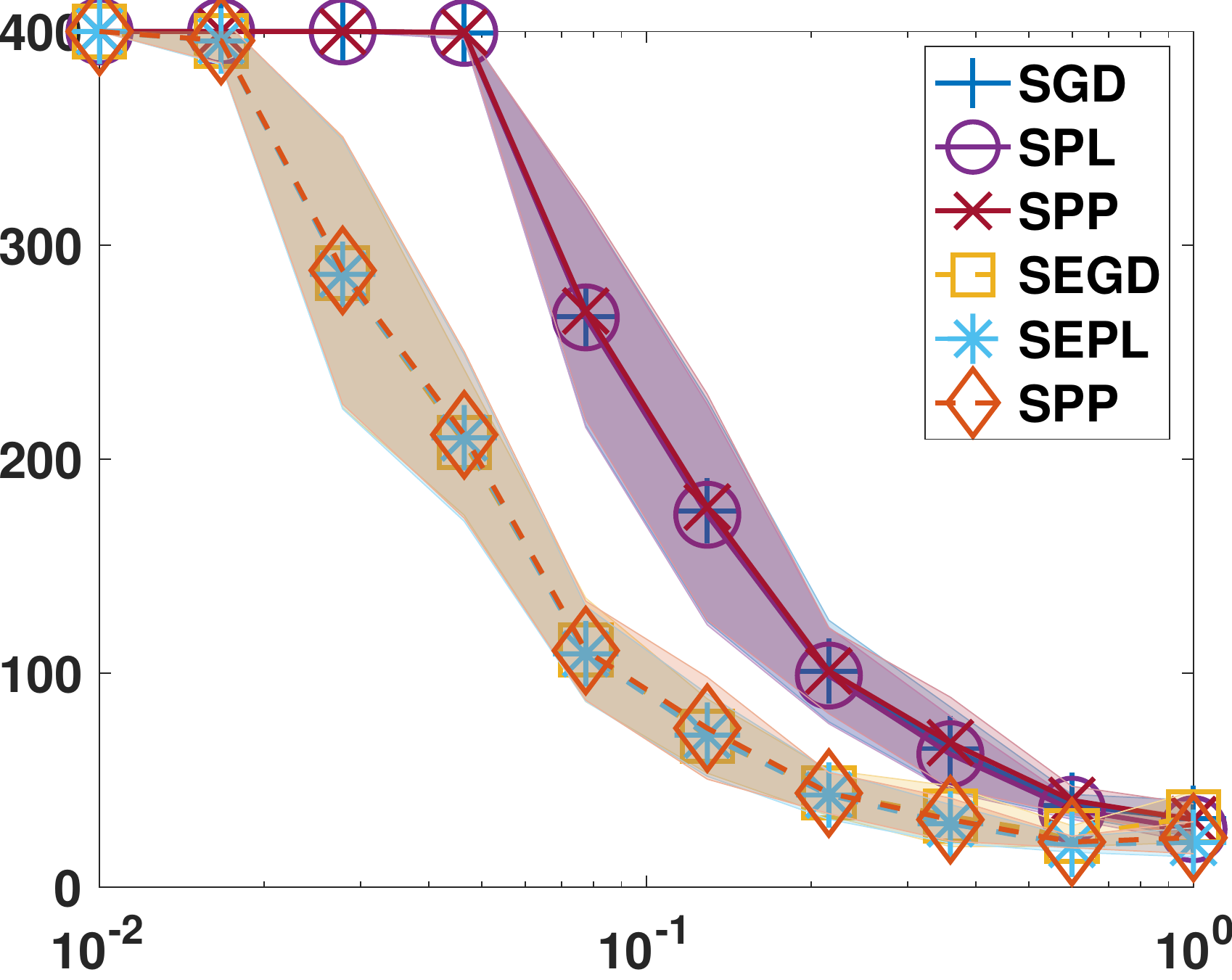} \includegraphics[scale=0.20]{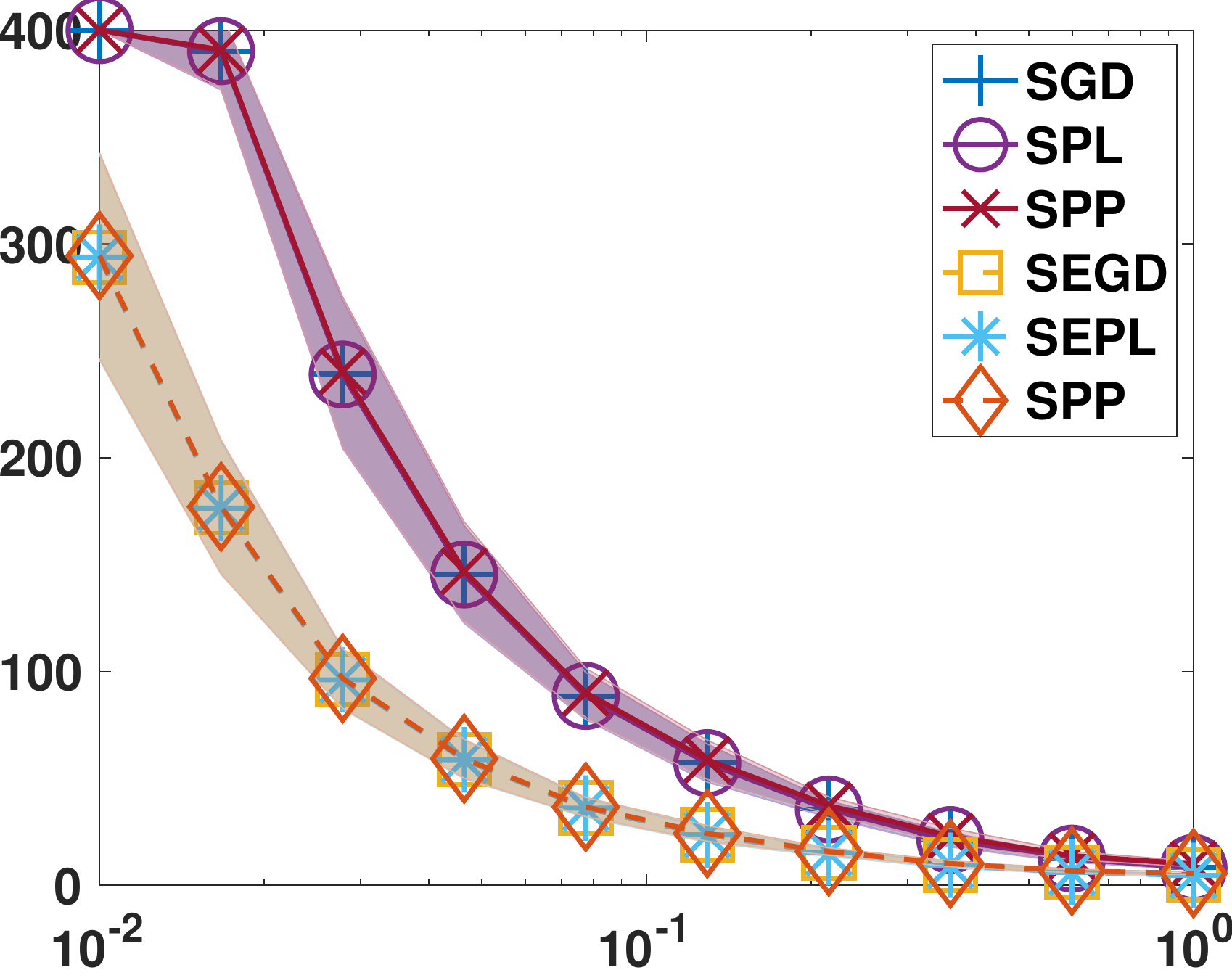}\includegraphics[scale=0.20]{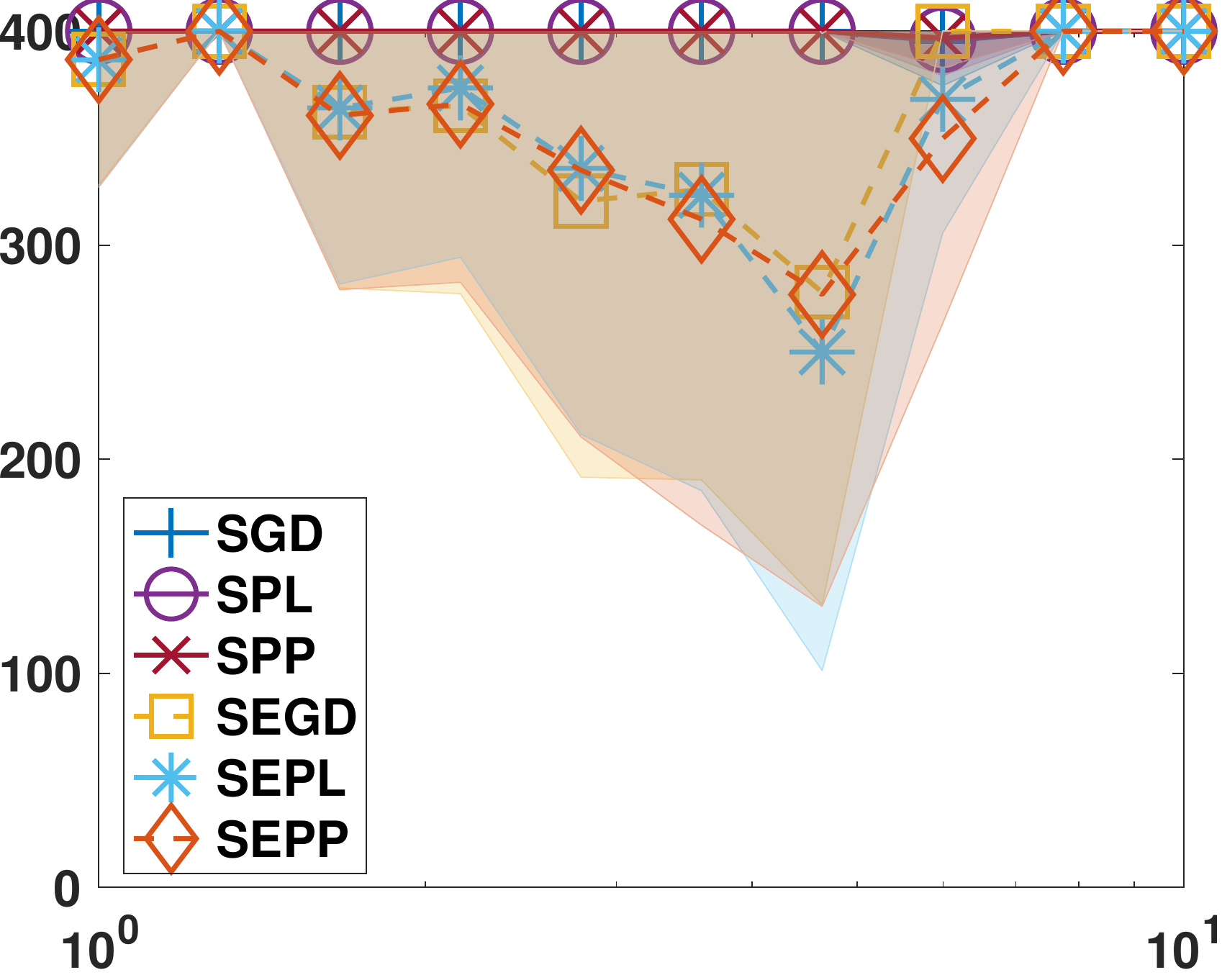}\includegraphics[scale=0.20]{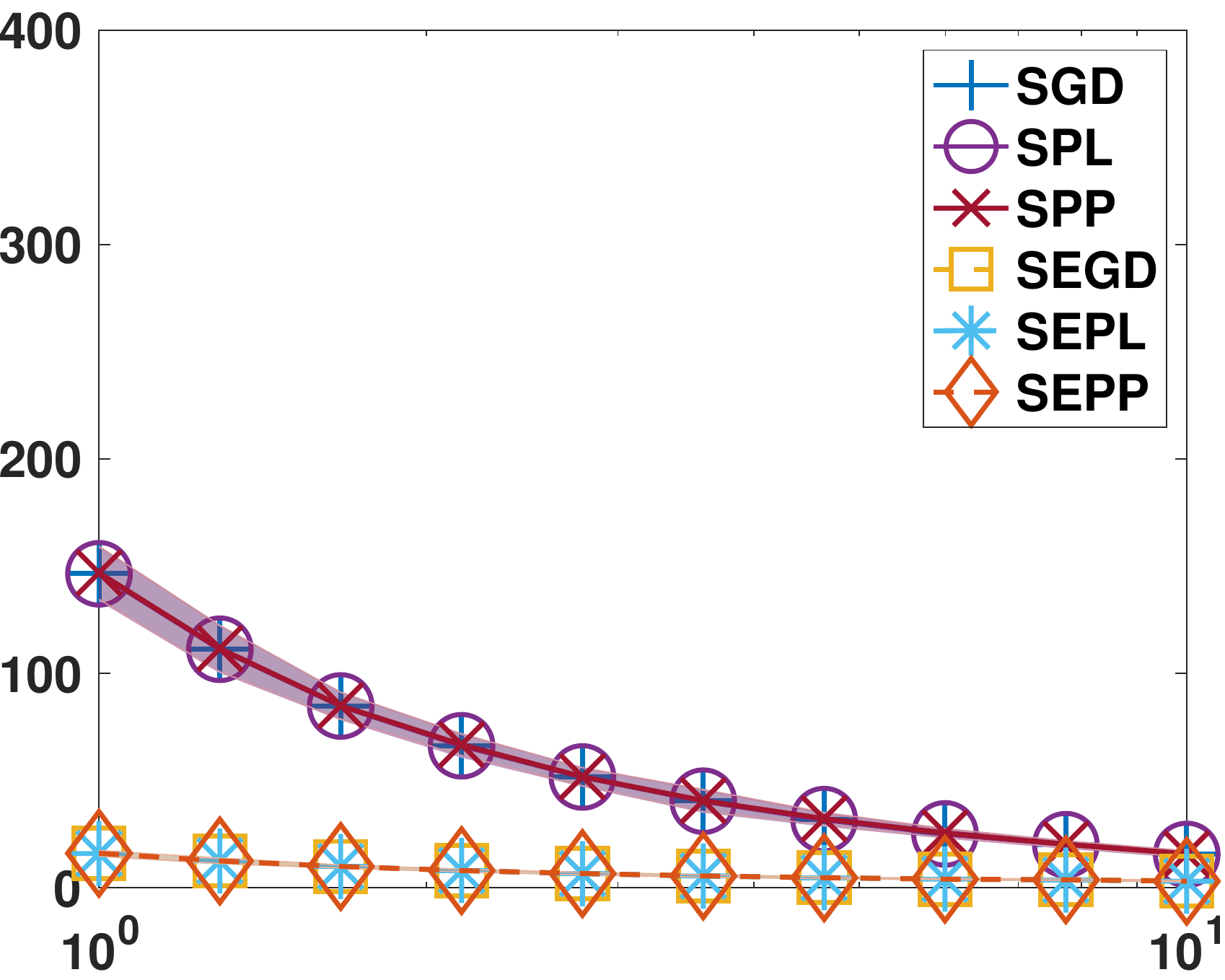}
	\caption{From left to right: synthetic datasets with $\kappa = 10, p_\text{fail}\in\{0.2, 0.3\}, \beta = 0.6$ and \texttt{zipcode} image (id=2) with $p_\text{fail} \in \{0.2, 0.3\}, \beta = 0.9$.
	x-axis: initial stepsize $\alpha_0$. y-axis: number of epochs on reaching desired accuracy 
	 \label{exp4:epoch-lr}}
\end{figure*}

\begin{figure*}[!htb]
\centering
\includegraphics[scale=0.20]{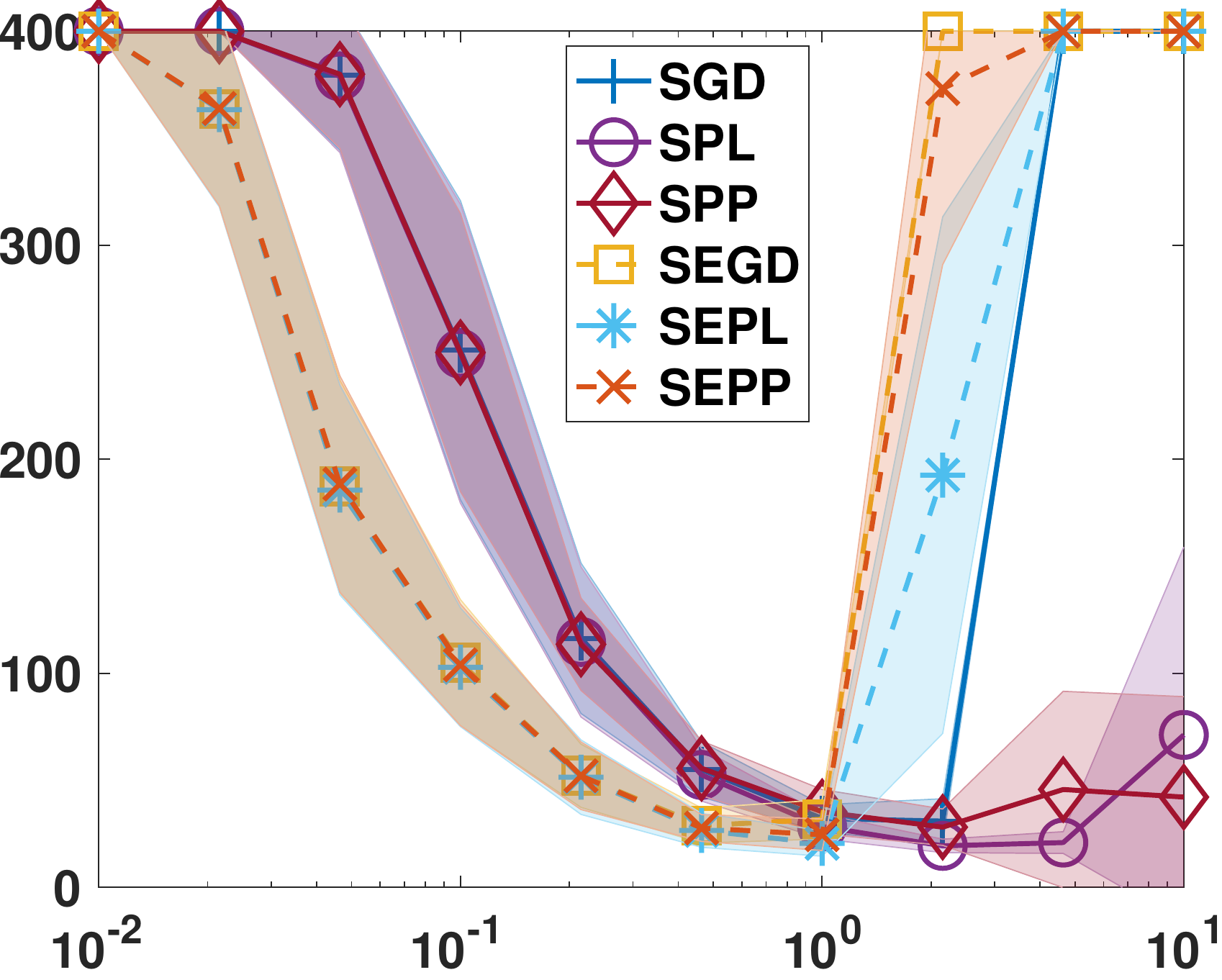} \includegraphics[scale=0.20]{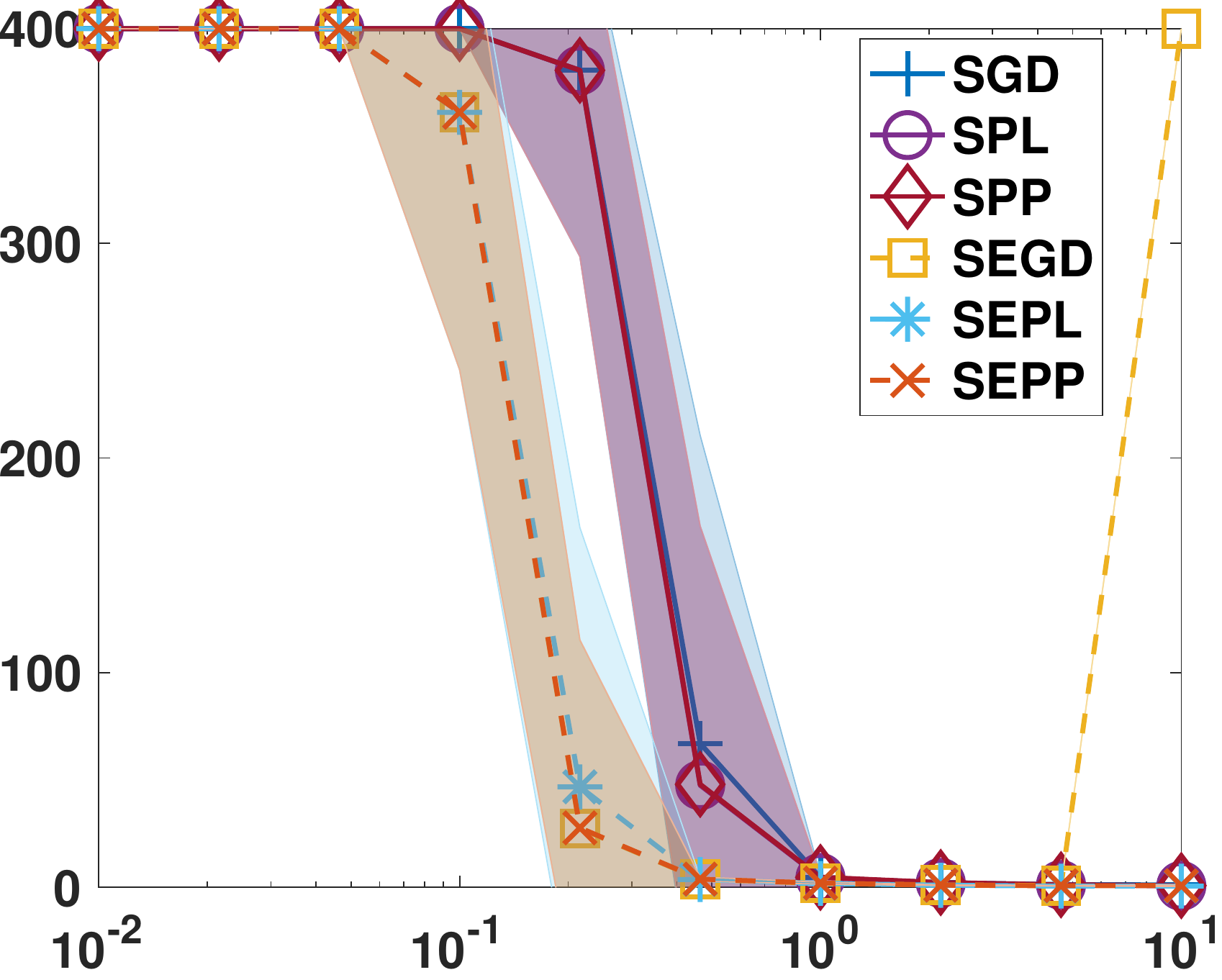}\includegraphics[scale=0.20]{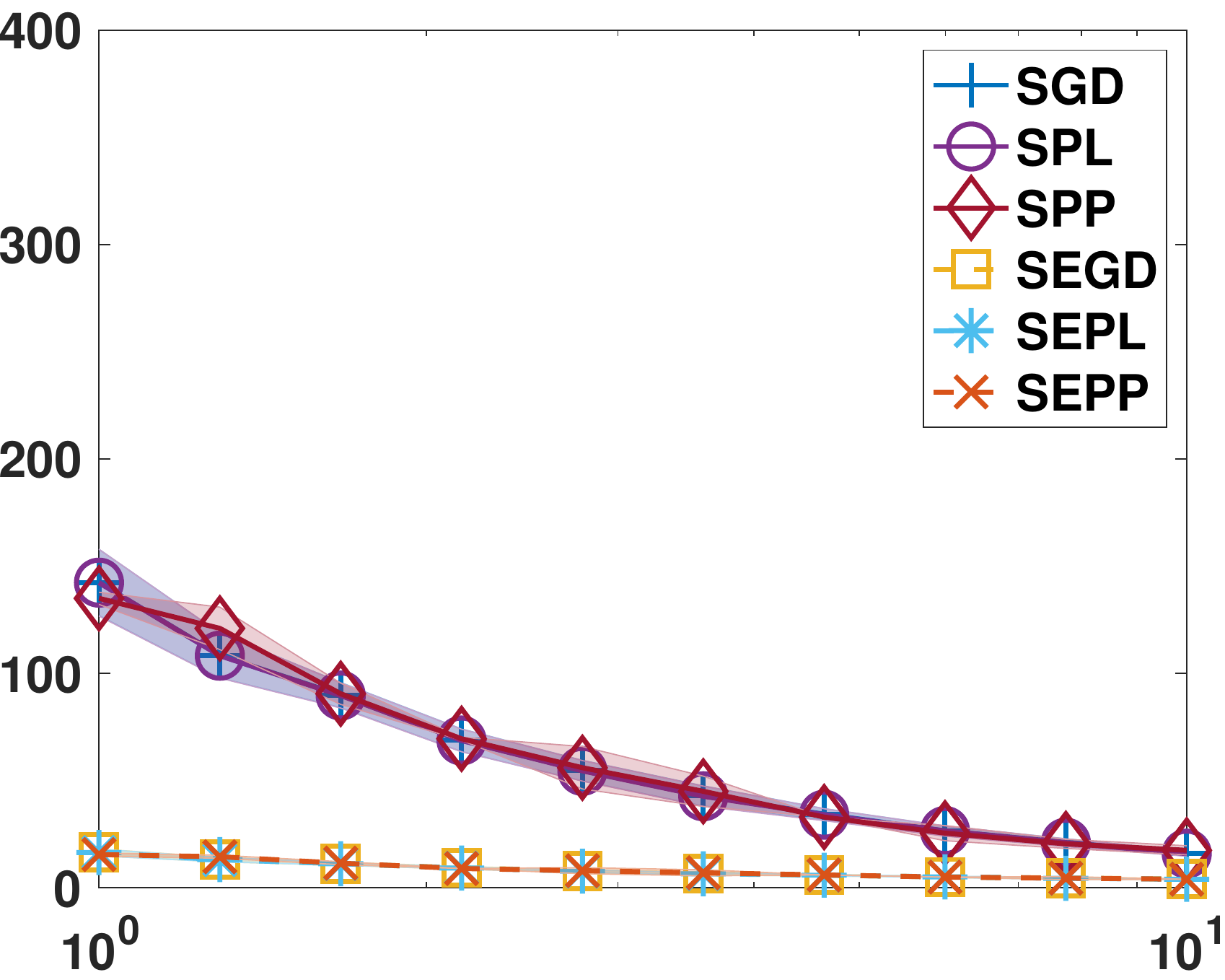}\includegraphics[scale=0.20]{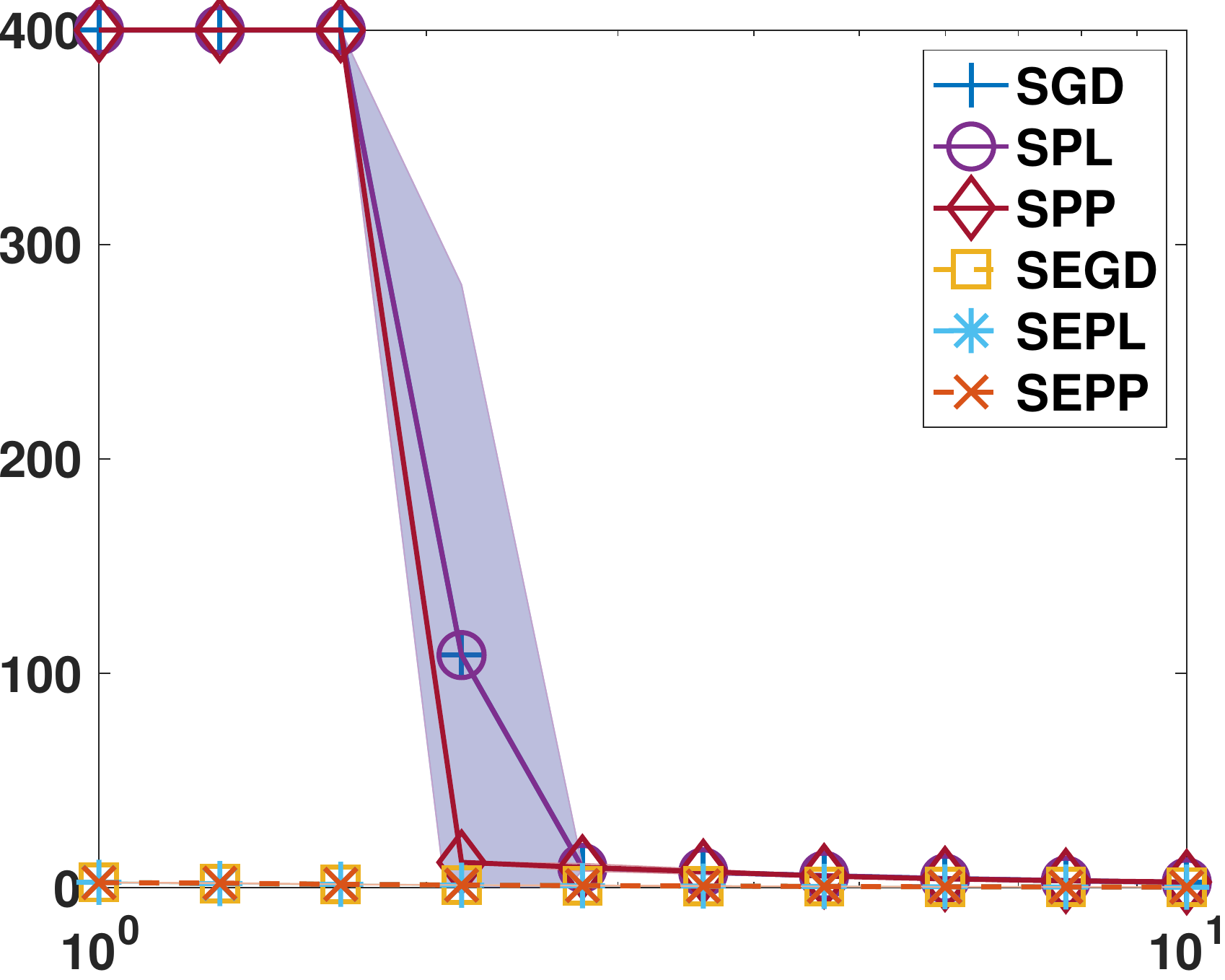}
	\caption{From left to right: synthetic datasets with $\kappa = 10, p_\text{fail} = 0.2, \beta = 0.6, m \in \{1, 32\}$ and \texttt{zipcode} image (id=24) with $p_\text{fail} = 0.3, \beta = 0.9, m \in \{1, 32\}$.
	x-axis: initial stepsize $\alpha_0$. y-axis: number of epochs for reaching desired accuracy 
	 \label{exp4:bmepoch-lr}}
\end{figure*}

In the last experiment,  we attempt to exploit the performance of the compared algorithms when minibatching and momentum are applied simultaneously. The  parameter setting is the same as that of the second experiment, except that we choose $m \in \{1, 32\}$. Results are plotted in Figure \ref{exp4:bmepoch-lr} and it can be seen that minibatch {\smod}, when combined with momentum, exhibits even better convergence performance and robustness.

\section{Discussion\label{sec:Conclu}}
On a broad class of non-smooth non-convex (particularly, weakly convex) problems,  we make stochastic model-based methods more efficient by leveraging minibatching and momentum---two techniques that are well-known only for {\sgd}.
Applying algorithm stability for optimization analysis is a key step to achieving improved convergence rate over the batch size. 
This perspective appears to be interesting for stochastic optimization in a much broader context. 
Although some progress is made, we are unable to show whether minibatches can accelerate {\sgd} when the objective does not have a smooth component. Note that the complexity of {\sgd} already matches the best bound of full subgradient method. It would be interesting to know whether this bound for {\sgd} is tight or improvable. It would also be interesting to study the lower bound of {\sgd} (and other stochastic algorithms) in the non-smooth setting.  Some interesting recent results can be referred from \citep{kornowski2021oracle, zhang2020complexity}.

\newpage
\bibliographystyle{abbrvnat}
\bibliography{citations}
\doparttoc
\faketableofcontents %
\part{} %

\newpage

\appendix

\addcontentsline{toc}{section}{Appendix}
\part{Appendix} %
\parttoc %

\vspace{10pt}
In the appendix, we present additional convergence analysis of the proposed  algorithms. Appendix~\ref{app:sec3} proves the convergence results for minibatching {\smod}. Appendix~\ref{app:sec4} proves the convergence results of momentum {\smod}.  Convergence results of {\smod} with both minibatching and momentum is formally presented in Appendix~\ref{app:mm}. 
Besides the missing proof for the main article, we present some new  convergence results of {\smod} for convex stochastic optimization in Appendix~\ref{app:cvx}, and show how to achieve and possibly improve state-of-the-art complexity rates.  {\smod} with Nesterov acceleration, which achieves the best complexity rate, is developed in Appendix~\ref{subsec:nes}. We provide details on how to solve the subproblems in the experiments in Section~\ref{app:subproblem}. Additional experiments on blind deconvolution are given in Appendix~\ref{app:addexp}.

\newpage
\section{Proof of results in Section~\ref{sec:minibatch}} \label{app:sec3}

Our paper will make use of the following elementary result, we refer
to \citeapp{davis2019stochasticweaklydup} for proof details.
\begin{lem}
A function $f(x)$ is $\lambda$-weakly convex if and only if for
any $x$, $y$ and $f'(x)\in\partial f(x)$, we have
$f(y)\ge f(x)+\langle f'(x),y-x\rangle-\frac{\lambda}{2}\|y-x\|^{2}$.
\end{lem}

We  state an important result which generalizes the well-known
three-point lemma to handle nonconvex function. 
\begin{lem}
\label{lem:three-point}Let $g(x)$ be a $\eta$-weakly convex function,
and $\kappa>\eta$. If 
\[
z^{+}=\argmin_{x\in\Xcal} \left\{g(x)+\frac{\kappa}{2}\|x-z\|^{2}\right\},
\]
then for any $x\in\Xcal$, we have 
\begin{equation}
g(z^{+})+\frac{\kappa}{2}\|z^{+}-z\|^{2}\le g(x)+\frac{\kappa}{2}\|x-z\|^{2}-\frac{\kappa-\eta}{2}\|x-z^{+}\|^{2}. 
\end{equation}
\end{lem}
\begin{proof}
Since $g(x)$ is $\eta$-weakly convex, $g(x)+\frac{\kappa}{2}\|x-z\|^{2}=\big[g(x)+\frac{\eta}{2}\|x-z\|^{2}\big]+\frac{\kappa-\eta}{2}\|x-z\|^{2}$
is strongly convex with parameter $\kappa-\eta$. Using the optimality
condition $0\in\partial\big[g(z^{+})+\frac{\kappa}{2}\|z^{+}-z\|^{2}\big]$
and strong convexity of $g(\cdot)+\frac{\kappa}{2}\|\cdot-z\|^{2}$,
we immediately obtain
\[
g(x)+\frac{\kappa}{2}\|x-z\|^{2}\ge g(z^{+})+\frac{\kappa}{2}\|z^{+}-z\|^{2}+\langle0,x-z^{+}\rangle+\frac{\kappa-\eta}{2}\|x-z^+\|^{2}.
\]
\end{proof}

Before getting down to the proof, first recall that in Section \ref{sec:minibatch}, we let $B=\{\xi_{1},\xi_{2},\ldots,\xi_{m}\}$
be the i.i.d. samples and $B_{(i)}=\{\xi_{1},\ldots,\xi_{i-1},\xi_{i}^{\prime},\xi_{i+1}\ldots,\xi_{m}\}$
by replacing $\xi_{i}$ with an i.i.d. copy $\xi_{i}^{\prime}$. We denote $B^{\prime}=\{\xi_{1}^{\prime},\xi_{2}^{\prime},\ldots,\xi_{m-1}^\prime, \xi_{m}^{\prime}\}$.

\subsection{Proof of Lemma~\ref{lem:model-uni-stable}}
For brevity, for $i=1,2,\ldots, m$, we denote
\[
\begin{aligned}
\hat{y} & =\arg\min_{x\in\Xcal}\left\{ f_{z}(x,B)+\frac{\gamma}{2}\|x-y\|^{2}\right\},\\
\hat{y}_{i} & =\arg\min_{x\in\Xcal}\left\{ f_{z}(x,B_{(i)})+\frac{\gamma}{2}\|x-y\|^{2}\right\}. 
\end{aligned}
\]
Using triangle inequality and Jensen's inequality, we deduce
\begin{align}
 & \big\vert\Ebb_{B,B^{\prime},i}\big[f_{z}(\hat{y}_{i},\xi_{i}^{\prime})-f_{z}(\hat{y},\xi_{i}^{\prime})\big]\big\vert\nonumber \\
={} & \Big\vert\frac{1}{m}\sum_{i=1}^{m}\Ebb_{B,\xi_{i}^{\prime}}\big[f_{z}(\hat{y}_{i},\xi_{i}^{\prime})-f_{z}(\hat{y},\xi_{i}^{\prime})\big]\Big\vert\nonumber \\
\le{} & \frac{1}{m}\sum_{i=1}^{m}\Ebb_{B,\xi_{i}^{\prime}}\big\vert f_{z}(\hat{y}_{i},\xi_{i}^{\prime})-f_{z}(\hat{y},\xi_{i}^{\prime})\big\vert\nonumber \\
\le{} & \frac{L}{m}\sum_{i=1}^{m}\Ebb_{B,\xi_{i}^{\prime}}\|\hat{y}_{i}-\hat{y}\|,\label{eq:bound-mid-1}
\end{align}
where the last inequality follows from \ref{ass:model-lip}.

Next we bound $\|\hat{y}-\hat{y}_{i}\|$. Due to $\lambda$-weak convexity
of $f_{z}(x,B)$ and by Lemma\,\ref{lem:three-point}, for any $i\in\{1,2,\ldots,m\}$,
we obtain
\begin{align*}
f_{z}(\hat{y},B)+\frac{\gamma}{2}\|\hat{y}-y\|^{2} & \leq f_{z}(\hat{y}_{i},B)+\frac{\gamma}{2}\|\hat{y}_{i}-y\|^{2}-\frac{\gamma-\lambda}{2}\|\hat{y}_{i}-\hat{y}\|^{2},\\
f_{z}(\hat{y}_{i},B_{(i)})+\frac{\gamma}{2}\|\hat{y}_{i}-y\|^{2} & \leq f_{z}(\hat{y},B_{(i)})+\frac{\gamma}{2}\|\hat{y}-y\|^{2}-\frac{\gamma-\lambda}{2}\|\hat{y}_{i}-\hat{y}\|^{2}.
\end{align*}
Summing up the above two relations, we deduce that
\begin{align}
 & (\gamma-\lambda)\|\hat{y}_{i}-\hat{y}\|^{2}\nonumber \\
\le{} & f_{z}(\hat{y},B_{(i)})-f_{z}(\hat{y},B)+f_{z}(\hat{y}_{i},B)-f_{z}(\hat{y}_{i},B_{(i)})\nonumber \\
={} & \frac{1}{m}\big[f_{z}(\hat{y},\xi_{i}^{\prime})-f_{z}(\hat{y}_{i},\xi_{i}^{\prime})+f_{z}(\hat{y}_{i},\xi_{i})-f_{z}(\hat{y},\xi_{i})\big]\label{eq:bound-mid-2}
\end{align}
Next, we use Assumption~\ref{ass:model-lip} and (\ref{eq:bound-mid-2}) 
to obtain
$
(\gamma-\lambda)\|\hat{y}_{i}-\hat{y}\|^{2}\le\frac{2L}{m}\|\hat{y}_{i}-\hat{y}\|,
$ which implies that  
\begin{equation}
\|\hat{y}_{i}-\hat{y}\|\leq\frac{2L}{m(\gamma-\lambda)}.\label{eq:bound-diff-1}
\end{equation}

In view of (\ref{eq:bound-mid-1}) and (\ref{eq:bound-diff-1}),
we have 
\begin{equation*}
	\big\vert\Ebb_{B,B^{\prime},i}\big[f_{z}(\hat{y}_{i},\xi_{i}^{\prime})-f_{z}(\hat{y},\xi_{i}^{\prime})\big]\big\vert
	\le \frac{2L^{2}}{m(\gamma-\lambda)}\\
= \vep.
\end{equation*}

\subsection{Proof of Theorem \ref{thm:stable-mod-func}}
Theorem \ref{thm:stable-mod-func} is an immediate consequence of Lemma\,\ref{lem:model-uni-stable} and the  following theorem which indicates that stability bounds the error of
approximating the full model function on expectation. %
\begin{thm}
\label{thm:stability-gen} Assume that $\prox_{\rho h}(\cdot,\cdot)$
is $\varepsilon$-stable and denote $x_B^+ = \prox_{\rho h}(x,B)$. Then, we have
\[
\big\vert\Ebb_{B}\big\{ h\big(x_B^+,B\big)-\Ebb_{\xi}\big[h\big(x_B^+,\xi\big)\big]\big\}\big|\leq\varepsilon.
\]
\end{thm}

\paragraph{Proof of Theorem~\ref{thm:stability-gen}}
The
proof resembles the argument of Lemma\,11 \citeapp{shalev2010learnabilitydup}.
For brevity we denote $\hat{x}=\prox_{\rho h}(x,B)$ and $\hat{x}_{i}=\prox_{\rho h}(x,B_{(i)})$.
Since $\xi_{i}^{\prime}$ is independent of $B$, we have $\Expe_{\xi}\big[h\big(\hat{x},\xi\big)\big]=\Expe_{\xi_{i}^{\prime}}\big[h\big(\hat{x},\xi_{i}^{\prime}\big)\big]$
for any $i\in\{1,\ldots,m\}$. Therefore, we have 
\begin{equation}
\Expe_{\xi}\big[h\big(\hat{x},\xi\big)\big]=\frac{1}{m}\sum_{j=1}^{m}\Expe_{\xi_{j}^{\prime}}\big[h\big(\hat{x},\xi_{j}^{\prime}\big)\big].\label{eq:stable-mid-01}
\end{equation}
Similarly, due to the independence assumption, we have
\begin{equation}
\Expe_{B}\big[h\big(\hat{x},\xi_{i}\big)\big]=\Expe_{B_{(i)}}\big[h(\hat{x}_{i},\xi_{i}^{\prime})\big],\label{eq:stable-mid-02}
\end{equation}
which implies that
\begin{equation}
\Expe_{B}\big[h\big(\hat{x},B\big)\big]=\frac{1}{m}\sum_{i=1}^{m}\Expe_{B}\big[h(\hat{x},\xi_{i})\big]=\frac{1}{m}\sum_{i=1}^{m}\Expe_{B_{(i)}}\big[h(\hat{x}_{i},\xi_{i}^{\prime})\big]\label{eq:stable-mid-03}
\end{equation}
In view of (\ref{eq:stable-mid-01}) and (\ref{eq:stable-mid-03}),
we deduce
\begin{align*}
 & \Ebb_{B}\Big\{ h\big(\hat{x},B\big)-\Ebb_{\xi}\big[h\big(\hat{x},\xi\big)\big]\Big\}\\
={} & \frac{1}{m}\sum_{i=1}^{m}\Expe_{B_{(i)}}\big[h(\hat{x}_{i},\xi_{i}^{\prime})\big]-\frac{1}{m}\sum_{i=1}^{m}\Ebb_{B,\xi_{i}^{\prime}}\big[h\big(\hat{x},\xi_{i}^{\prime}\big)\big]\\
={} & \frac{1}{m}\sum_{i=1}^{m}\Expe_{B,\xi_{i}^{\prime}}\big[h(\hat{x}_{i},B_{(i)})-h\big(\hat{x},\xi_{i}^{\prime}\big)\big]\\
={} & \Expe_{B,B^{\prime},i}\big[h(\hat{x}_{i},B_{(i)})-h\big(\hat{x},\xi_{i}^{\prime}\big)\big].
\end{align*}
Appealing to the stability assumption, we complete the proof.

\subsection{Proof of Theorem~\ref{thm:main-mbsmod}}
First, due to the weak convexity of $f_{x^{k}}(\cdot,B_{k})$ and
Lemma\,\ref{lem:three-point}, we have 
\begin{equation}
f_{x^{k}}(x^{k+1},B_{k})+\frac{\gamma_{k}}{2}\|x^{k+1}-x^{k}\|^{2}\le f_{x^{k}}(x,B_{k})+\frac{\gamma_{k}}{2}\|x-x^{k}\|^{2}-\frac{\gamma_{k}-\lambda}{2}\|x^{k+1}-x\|^{2},\quad\forall x\in\Xcal.\label{eq:mb-mid-10}
\end{equation}
For simplicity, we denote $\hat{x}^{k}=\prox_{f/\rho}(x^{k})=\argmin_{x\in\Xcal}\big\{ f(x)+\frac{\rho}{2}\|x-x^{k}\|^{2}\big\}$.
Then substituting $x=\hat{x}^{k}$ in~(\ref{eq:mb-mid-10}), we have
\begin{equation}
f_{x^{k}}(x^{k+1},B_{k})+\frac{\gamma_{k}}{2}\|x^{k+1}-x^{k}\|^{2}\leq f_{x^{k}}(\hat{x}^{k},B_{k})+\frac{\gamma_{k}}{2}\|\hat{x}^{k}-x^{k}\|^{2}-\frac{\gamma_{k}-\lambda}{2}\|x^{k+1}-\hat{x}^{k}\|^{2}.\label{eq15}
\end{equation}

Analogously, since $f(x)$ is $(\lambda+\tau)$-weakly convex, applying
Lemma\,\ref{lem:three-point} with $g(x)=f(x)$, $\eta=\lambda+\tau$
and $\kappa=\rho$, we have
\begin{equation}
f(\hat{x}^{k})+\frac{\rho}{2}\|\hat{x}^{k}-x^{k}\|^{2}\leq f(x^{k+1})+\frac{\rho}{2}\|x^{k+1}-x^{k}\|^{2}-\frac{\rho-\lambda-\tau}{2}\|\hat{x}^{k}-x^{k+1}\|^{2}.\label{eq16}
\end{equation}
Summing up \eqref{eq15} and \eqref{eq16} gives
\begin{align}
 & \frac{\gamma_{k}-\rho}{2}\|x^{k+1}-x^{k}\|^{2}+\frac{\gamma_{k}+\rho-2\lambda-\tau}{2}\|\hat{x}^{k}-x^{k+1}\|^{2}-\frac{\gamma_{k}-\rho}{2}\mathbb{E}_{k}\|\hat{x}^{k}-x^{k}\|^{2}\nonumber \\
\le{} & f(x^{k+1})-f_{x^{k}}(x^{k+1},B_{k})+f_{x^{k}}(\hat{x}^{k},B_{k})-f(\hat{x}^{k})\nonumber \\
={} & \big\{f(x^{k+1})-\Ebb_{\xi}\big[f_{x^{k}}(x^{k+1},\xi)\big]\big\}+\big\{\Ebb_{\xi}\big[f_{x^{k}}(x^{k+1},\xi)\big]-f_{x^{k}}(x^{k+1},B_{k})\big\}\nonumber\\
 & +\big[f_{x^{k}}(\hat{x}^{k},B_{k})-f(\hat{x}^{k})\big]\nonumber \\
\le{} & \frac{\tau}{2}\|x^{k}-x^{k+1}\|^{2}+\frac{\tau}{2}\|x^{k}-\hat{x}^{k}\|^{2}+\Ebb_{\xi}\big[f_{x^{k}}(x^{k+1},\xi)\big]-f_{x^{k}}(x^{k+1},B_{k}),\label{eq:mb-mid-01}
\end{align}

where the last inequality uses the Assumption~\ref{ass:two-sides-quad}.
Moreover, note that Theorem~\ref{thm:stable-mod-func} implies 
\begin{equation}
\Ebb_{k}\big\{\Ebb_{\xi}\big[f_{x^{k}}(x^{k+1},\xi)\big]-f_{x^{k}}(x^{k+1},B_{k})\big\}\le\vep_{k}.\label{eq:mb-mid-02}
\end{equation}

Taking expectation over $B_{k}$ in (\ref{eq:mb-mid-01}) and combining
the result with (\ref{eq:mb-mid-02}), we obtain
\begin{align*}
 & \frac{\gamma_{k}-\rho}{2}\thinspace\Ebb_{k}\big[\|x^{k+1}-x^{k}\|^{2}\big]+\frac{\gamma_{k}+\rho-2\lambda-\tau}{2}\thinspace\Ebb_{k}\big[\|\hat{x}^{k}-x^{k+1}\|^{2}\big]-\frac{\gamma_{k}-\rho}{2}\|\hat{x}^{k}-x^{k}\|^{2}\\
\le{} & \frac{\tau}{2}\thinspace\Ebb_{k}\big[\|x^{k}-x^{k+1}\|^{2}\big]+\frac{\tau}{2}\|\hat{x}^{k}-x^{k}\|^{2}+\vep_{k},
\end{align*}
which, by rearranging terms, implies
\begin{align}
 & \mathbb{E}_{k}\big[\|x^{k+1}-\hat{x}^{k}\|^{2}\big]\nonumber \\
\le{} & \frac{\gamma_{k}-\rho+\tau}{\gamma_{k}+\rho-2\lambda-\tau}\|\hat{x}^{k}-x^{k}\|^{2}-\frac{\gamma_{k}-\rho-\tau}{\gamma_{k}+\rho-2\lambda-\tau}\mathbb{E}_{k}\big[\|x^{k}-x^{k+1}\|^{2}\big]+\frac{2\vep_{k}}{\gamma_{k}+\rho-2\lambda-\tau}\nonumber \\
\le{} & \|\hat{x}^{k}-x^{k}\|^{2}-\frac{2(\rho-\lambda-\tau)}{\gamma_{k}+\rho-2\lambda-\tau}\|\hat{x}^{k}-x^{k}\|^{2}+\frac{2\vep_{k}}{\gamma_{k}+\rho-2\lambda-\tau},\label{eq:mb-mid-09}
\end{align}
Above, the last inequality in (\ref{eq:mb-mid-09}) uses the assumption $\gamma_{k}-\rho-\tau\ge0$.

Moreover, following the result (\ref{eq:mb-mid-09}) and the
definition of Moreau envelope, we have 
\begin{align*}
& \mathbb{E}_{k}\big[f_{1/\rho}(x^{k+1})\big] \\
={} & \mathbb{E}_{k}\Big[f(\hat{x}^{k+1})+\frac{\rho}{2}\|\hat{x}^{k+1}-x^{k+1}\|^2\Big]\\
\le{} & f(\hat{x}^{k})+\mathbb{E}_{k}\Big[\frac{\rho}{2}\|\hat{x}^{k}-x^{k+1}\|^2\Big]\\
\le{} &  f(\hat{x}^{k})+\frac{\rho}{2}\|\hat{x}^{k}-x^{k}\|^{2}-\frac{\rho(\rho-\lambda-\tau)}{\gamma_{k}+\rho-2\lambda-\tau}\|\hat{x}^{k}-x^{k}\|^{2}+\frac{\rho\vep_{k}}{\gamma_{k}+\rho-2\lambda-\tau}\\
={} & f_{1/\rho}(x^{k})-\frac{\rho(\rho-\lambda-\tau)}{\gamma_{k}+\rho-2\lambda-\tau}\|\hat{x}^{k}-x^{k}\|^{2}+\frac{\rho\vep_{k}}{\gamma_{k}+\rho-2\lambda-\tau}.
\end{align*}
Finally, applying the identity $\|\hat{x}^{k}-x^{k}\|^{2}=\rho^{-2}\|\nabla f_{1/\rho}(x^{k})\|^{2}$
and rearranging the terms, we get\,(\ref{eq:main-bound-mini}).

\subsection{Proof of Theorem \ref{thm:rate-mb-smod}}
First, summing up (\ref{eq:main-bound-mini}) over $k=1,2,\ldots,K$,
and taking expectation over all randomness, we have 
\begin{align*}
& \sum_{k=1}^{K}\frac{\rho-\lambda-\tau}{\rho(\gamma_{k}+\rho-2\lambda-\tau)}\Ebb[\|\nabla f_{1/\rho}(x^{k})\|^{2}]\\
\le{} &  f_{1/\rho}(x^{1})-\mathbb{E}\big[f_{1/\rho}(x^{K+1})\big]+\sum_{k=1}^{K}\frac{\rho\vep_{k}}{\gamma_{k}+\rho-2\lambda-\tau}\\
\le{} & \Delta+\sum_{k=1}^{K}\frac{\rho\vep_{k}}{\gamma_{k}+\rho-2\lambda-\tau},
\end{align*}
where the second inequality uses $-f_{1/\rho}(x^{K+1})\le-\min_{x}f(x)$.
Plugging in $\gamma_{k}=\gamma$ and $m_{k}=m$ in above and appealing
to the definition of $x^{{k^*}}$, we have
\begin{align}
 \frac{\rho-\lambda-\tau}{\rho}\Ebb\big[\|\nabla f_{1/\rho}(x^{{k^*}})\|^{2}\big]
 & = \frac{\rho-\lambda-\tau}{\rho K}\sum_{k=1}^{K}\Ebb\big[\|\nabla f_{1/\rho}(x^{k})\|^{2}\big]\nonumber \\
 & \le \frac{(\gamma+\rho-2\lambda-\tau)\Delta}{K}+\frac{\rho}{K}\sum_{k=1}^{K}\vep_{k}\nonumber \\
 & \le \frac{(2\rho-\lambda)\Delta}{K}+\frac{\eta\Delta}{K}+ \frac{2\rho L^{2}}{m(\gamma-\lambda)}\nonumber \\
 & \le \frac{(2\rho-\lambda)\Delta}{K}+ \frac{\eta\Delta}{K}+\frac{2\rho L^{2}}{m\eta}, \label{eq:mb-mid-05}
\end{align}
where the second inequality uses $\gamma\le\rho+\tau+\lambda+\eta$,
the third inequality uses $\gamma-\lambda\ge\eta$. 
Dividing both sides of (\ref{eq:mb-mid-05}) by $\frac{\rho-\lambda-\tau}{\rho}$
gives (\ref{eq:mb-mid-04}).

\section{Proof of results in Section~\ref{sec:momentum}}\label{app:sec4}
\subsection{Proof of Lemma~\ref{lem:extra-2}}
Denote $\bar{x}=\beta x^{k}+(1-\beta)x$ for $x\in\Xcal$. Then $\bar{x}$
is also feasible due to the convexity of $\Xcal.$ Noting that $\theta=1-\beta$,
we have the following identities:
\begin{align}
\bar{x}-x^{k} & =\theta(x-x^{k}),\label{eq:ident-mid-1}\\
\bar{x}-y^{k} & =\theta(x-z^{k}),\label{eq:ident-mid-2}\\
\bar{x}-x^{k+1} & =\theta(x-z^{k+1}).\label{eq:ident-mid-3}
\end{align}
Applying Lemma\,\ref{lem:three-point} and using the optimality of
$x^{k+1}$, we have
\begin{align}
& f_{x^{k}}(x^{k+1},\xi^{k})+\frac{\gamma}{2}\|x^{k+1}-y^{k}\|^{2} \nonumber \\
\le{} &  f_{x^{k}}(\bar{x},\xi^{k})+\frac{\gamma}{2}\|\bar{x}-y^{k}\|^{2}-\frac{\gamma-\lambda}{2}\|x^{k+1}-\bar{x}\|^{2}\nonumber \\
={} & f_{x^{k}}(\bar{x},\xi^{k})+\frac{\gamma\theta^{2}}{2}\|x-z^{k}\|^{2}-\frac{(\gamma-\lambda)\theta^{2}}{2}\|x-z^{k+1}\|^{2}\label{eq:extra-mid-14}
\end{align}
Since $f_{x^{k}}(\cdot,\xi^{k})+\frac{\lambda}{2}\|\cdot-x^{k}\|^{2}$
is convex, we have
\begin{align}
 f_{x^{k}}(\bar{x},\xi^{k})  & \le (1-\theta)\big[f_{x^{k}}(x^{k},\xi^{k})\big]+\theta\big[f_{x^{k}}(x,\xi^{k})+\frac{\lambda}{2}\|x-x^{k}\|^{2}\big]-\frac{\lambda}{2}\|\bar{x}-x^{k}\|^{2}\nonumber \\
& \le (1-\theta)f(x^{k},\xi^{k})+\theta\big[f(x,\xi^{k})+\frac{\lambda+\tau}{2}\|x-x^{k}\|^{2}\big]-\frac{\lambda\theta^{2}}{2}\|x-x^{k}\|^{2}\label{eq:extra-mid-15}
\end{align}
where the second inequality uses Assumptions~\ref{ass:unbiased}, \ref{ass:one-side-quad}
and (\ref{eq:ident-mid-1}). 
Summing up (\ref{eq:extra-mid-14}) and (\ref{eq:extra-mid-15}), we get
\begin{align}
& f_{x^{k}}(x^{k+1},\xi^{k})+\frac{\gamma}{2}\|x^{k+1}-y^{k}\|^{2} \nonumber \\
\le {} & (1-\theta)f(x^{k},\xi^{k})+\theta\big[f(x,\xi^{k})+\frac{\lambda+\tau}{2}\|x-x^{k}\|^{2}\big]-\frac{\lambda\theta^{2}}{2}\|x-x^{k}\|^{2}\nonumber \\
& + \frac{\gamma\theta^{2}}{2}\|x-z^{k}\|^{2}-\frac{(\gamma-\lambda)\theta^{2}}{2}\|x-z^{k+1}\|^{2} 
\label{eq:extra-mid-17}
\end{align}
Moreover, appealing to Assumption~\ref{ass:unbiased}
and \ref{ass:model-lip}, we have 
\begin{equation}
f(x^{k},\xi^{k})-L\|x^{k+1}-x^{k}\|=f_{x^{k}}(x^{k},\xi^{k})-L\|x^{k+1}-x^{k}\|\le f_{x^{k}}(x^{k+1},\xi^{k}).\label{eq:extra-mid-16}
\end{equation}
Next, Putting (\ref{eq:extra-mid-17}) and (\ref{eq:extra-mid-16}) together, we have
\begin{align}
& -L\|x^{k+1}-x^{k}\| +\frac{\gamma}{2}\|x^{k+1}-y^{k}\|^{2} \nonumber\\
\le {} & -\theta f(x^{k},\xi^{k})+\theta\big[f(x,\xi^{k})+\frac{\lambda+\tau}{2}\|x-x^{k}\|^{2}\big]-\frac{\lambda\theta^{2}}{2}\|x-x^{k}\|^{2}\nonumber \\
& + \frac{\gamma\theta^{2}}{2}\|x-z^{k}\|^{2}-\frac{(\gamma-\lambda)\theta^{2}}{2}\|x-z^{k+1}\|^{2}
\end{align}

Denote $\hat{z}^{k}=\prox_{f/\rho}(z^{k})$. 
Note that $z^{k}$ may be infeasible, but the feasibility of $\hat{z}^{k}$ is always guaranteed.
Substituting $x=\hat{z}^{k}$ in the above result and then taking expectation
over $\xi^{k}$, we have 
\begin{align}
 & -L\Ebb_{k}[\|x^{k+1}-x^{k}\|]+\theta f(x^{k})\nonumber \\
\le{} & \theta f(\hat{z}^{k})+\frac{\theta(\lambda+\tau)}{2}\|\hat{z}^{k}-x^{k}\|^{2}-\frac{\lambda\theta^{2}}{2}\|\hat{z}^{k}-x^{k}\|^{2}\nonumber \\
 & +\frac{\gamma\theta^{2}}{2}\|\hat{z}^{k}-z^{k}\|^{2}-\frac{(\gamma-\lambda)\theta^{2}}{2}\Ebb_{k}[\|\hat{z}^{k}-z^{k+1}\|^{2}]-\frac{\gamma}{2}\Ebb_{k}[\|x^{k+1}-y^{k}\|^{2}]\label{eq:extra-mid-7}
\end{align}
Next we apply Lemma\,\ref{lem:three-point} and use the optimality
condition for $\hat{z}^{k}$, noting that $f(x)$ is $(\tau+\lambda)$-weakly
convex, we get
\begin{equation}
f(\hat{z}^{k})+\frac{\rho}{2}\|\hat{z}^{k}-z^{k}\|^{2}\le f(x^{k})+\frac{\rho}{2}\|x^{k}-z^{k}\|^{2}-\frac{\rho-\tau-\lambda}{2}\|x^{k}-\hat{z}^{k}\|^{2}.\label{eq:extra-mid-8}
\end{equation}
Multiplying (\ref{eq:extra-mid-8}) by $\theta$ and then adding the result to (\ref{eq:extra-mid-7}), we
deduce 
\begin{align}
 & -L\Ebb_{k}[\|x^{k+1}-x^{k}\|]\nonumber \\
\le{} & \frac{\rho\theta}{2}\|x^{k}-z^{k}\|^{2}-\frac{\theta(\rho-\tau-\lambda)}{2}\|x^{k}-\hat{z}^{k}\|^{2}-\frac{\rho\theta}{2}\|\hat{z}^{k}-z^{k}\|^{2}\nonumber \\
 & +\frac{\theta(\lambda+\tau)}{2}\|\hat{z}^{k}-x^{k}\|^{2}-\frac{\lambda\theta^{2}}{2}\|\hat{z}^{k}-x^{k}\|^{2}\nonumber \\
 & +\frac{\gamma\theta^{2}}{2}\|\hat{z}^{k}-z^{k}\|^{2}-\frac{(\gamma-\lambda)\theta^{2}}{2}\Ebb_{k}[\|\hat{z}^{k}-z^{k+1}\|^{2}]-\frac{\gamma}{2}\Ebb_{k}[\|x^{k+1}-y^{k}\|^{2}]\nonumber \\
={} & \frac{\gamma\theta^{2}-\lambda\theta^{2}}{2}\big(\|\hat{z}^{k}-z^{k}\|^{2}-\Ebb_{k}[\|\hat{z}^{k}-z^{k+1}\|^{2}]\big)-\frac{\rho\theta-\lambda\theta^{2}}{2}\Ebb_{k}[\|\hat{z}^{k}-z^{k}\|^{2}]\nonumber \\
 & -\frac{\theta((\rho-2(\lambda+\tau))+\lambda\theta)}{2}\|\hat{z}^{k}-x^{k}\|^{2}\nonumber \\
 & -\frac{\gamma}{2}\Ebb_{k}[\|x^{k+1}-y^{k}\|^{2}]+\frac{\rho\beta^{2}\theta^{-1}}{2}\|x^{k}-x^{k-1}\|^{2}.\label{eq:extra-mid-9}
\end{align}
where the last equality uses the identity $z^{k}-x^{k}=\beta\theta^{-1}(x^{k}-x^{k-1})$.

Moreover, we can bound the term $\Ebb_{k}[\|x^{k+1}-y^{k}\|^{2}]$ using the following relation 
\begin{align}
& \|x^{k+1}-y^{k}\|^{2} \nonumber \\
={} & \|x^{k+1}-x^{k}\|^{2}+\beta^{2}\|x^{k}-x^{k-1}\|^{2}-2\beta\langle x^{k+1}-x^{k},x^{k}-x^{k-1}\rangle\nonumber \\
\ge{} & \|x^{k+1}-x^{k}\|^{2}+\beta^{2}\|x^{k}-x^{k-1}\|^{2}-\beta\|x^{k+1}-x^{k}\|^{2}-\beta\|x^{k}-x^{k-1}\|^{2}\nonumber \\
={} & \theta^{2}\|x^{k+1}-x^{k}\|^{2}+\beta\theta\big(\|x^{k+1}-x^{k}\|^{2}-\|x^{k}-x^{k-1}\|^{2}\big).\label{eq:extra-mid-10}
\end{align}
Next, adding $L\Ebb_{k}[\|x^{k+1}-x^{k}\|]$ to both sides of (\ref{eq:extra-mid-9}),
using the non-negativity of $\rho-2(\lambda+\tau)$ and the bound
(\ref{eq:extra-mid-10}), we deduce
\begin{align*}
0{} & \le\frac{\gamma\theta^{2}-\lambda\theta^{2}}{2}\big(\|\hat{z}^{k}-z^{k}\|^{2}-\Ebb_{k}[\|\hat{z}^{k}-z^{k+1}\|^{2}]\big)-\frac{\rho\theta-\lambda\theta^{2}}{2}\|\hat{z}^{k}-z^{k}\|^{2}\\
 & \quad-\frac{\gamma\beta\theta+\rho\beta^{2}\theta^{-1}}{2}\Ebb_{k}[\|x^{k+1}-x^{k}\|^{2}]+\frac{\gamma\beta\theta+\rho\beta^{2}\theta^{-1}}{2}\|x^{k}-x^{k-1}\|^{2}\\
 & \quad+\Ebb_{k}\Big[L\|x^{k+1}-x^{k}\|-\frac{\gamma\theta^{2}-\rho\beta^{2}\theta^{-1}}{2}\|x^{k+1}-x^{k}\|^{2}\Big]\\
 & \le\frac{\gamma\theta^{2}-\lambda\theta^{2}}{2}\big(\|\hat{z}^{k}-z^{k}\|^{2}-\Ebb_{k}[\|\hat{z}^{k}-z^{k+1}\|^{2}]\big)-\frac{\rho\theta-\lambda\theta^{2}}{2}\|\hat{z}^{k}-z^{k}\|^{2}\\
 & \quad-\frac{\gamma\beta\theta+\rho\beta^{2}\theta^{-1}}{2}\Ebb_{k}[\|x^{k+1}-x^{k}\|^{2}]+\frac{\gamma\beta\theta+\rho\beta^{2}\theta^{-1}}{2}\|x^{k}-x^{k-1}\|^{2}\\
 & \quad+\frac{L^{2}}{(\gamma\theta^{2}-\rho\beta^{2}\theta^{-1})} 
 -\frac{\gamma\theta^{2}-\rho\beta^{2}\theta^{-1}}{4}\Ebb_{k}[\|x^{k+1}-x^{k}\|^{2}]
\end{align*}
where the last inequality identifies the fact that $bx-\frac{a}{4}x^{2}\le\frac{b^{2}}{a}$
for $a,b>0$, $\forall x\in\Rbb$. It then follows that
\begin{align}
& \Ebb_{k}[\|\hat{z}^{k}-z^{k+1}\|^{2}] \nonumber \\
\le{} & \|\hat{z}^{k}-z^{k}\|^{2}-\frac{\rho-\lambda\theta}{\gamma\theta-\lambda\theta}\|\hat{z}^{k}-z^{k}\|^{2}
+\frac{2L^{2}}{(\gamma\theta^{2}-\rho\beta^{2}\theta^{-1})(\gamma\theta^{2}-\lambda\theta^{2})}\nonumber \\
 & -\frac{\gamma\beta+\rho\beta^{2}\theta^{-2}}{\gamma\theta-\lambda\theta}\big(\Ebb_{k}[\|x^{k+1}-x^{k}\|^{2}]-\|x^{k}-x^{k-1}\|^{2}\big)\nonumber \\
 & -  \frac{\gamma-\rho\beta^{2}\theta^{-3}}{2(\gamma-\lambda)} \Ebb_{k}[\|x^{k+1}-x^{k}\|^{2}]
 \label{eq:extra-mid-11}
\end{align}
In view of (\ref{eq:extra-mid-11}) and the definition of Moreau envelope,
we have 
\begin{align}
& \Ebb_{k} \left[ f_{1/\rho}(z^{k+1}) \right] \nonumber \\
={} & \Ebb_{k}\big[f(\hat{z}^{k+1})+\frac{\rho}{2}\|z^{k+1}-\hat{z}^{k+1}\|^{2}\big]\nonumber \\
\le{} & \Ebb_{k}\big[f(\hat{z}^{k})+\frac{\rho}{2}\|z^{k+1}-\hat{z}^{k}\|^{2}\big]\nonumber \\
\le{} &  f_{1/\rho}(z^{k})-\frac{\rho(\rho-\lambda\theta)}{2(\gamma\theta-\lambda\theta)}\|z^{k}-\hat{z}^{k}\|^{2}
+\frac{\rho L^{2}}{(\gamma\theta^{2}-\rho\beta^{2}\theta^{-1})(\gamma\theta^{2}-\lambda\theta^{2})}\nonumber \\
 & +\frac{\rho(\gamma\beta+\rho\beta^{2}\theta^{-2})}{2(\gamma\theta-\lambda\theta)}\big\{\|x^{k}-x^{k-1}\|^{2}-\Ebb_{k}\big[\|x^{k+1}-x^{k}\|^{2}\big]\big\}.\nonumber\\
 & - \frac{\rho(\gamma-\rho\beta^{2}\theta^{-3})}{4(\gamma-\lambda)} \Ebb_{k}[\|x^{k+1}-x^{k}\|^{2}]
 \label{eq:extra-mid-6-1}
\end{align}

In view of the above result and the relation $\|z^{k}-\hat{z}^{k}\|^2 = \rho^{-2} \|\nabla_{1/\rho}f(z^{k})\|^2$,
we obtain\,(\ref{eq:extra-decent-prop}).

\subsection{Proof of Theorem~\ref{thm:extra-2-1}}
Unfolding the relation\,(\ref{eq:extra-decent-prop}) and then taking
expectation over all the randomness, we have 
\begin{align}
 & \frac{\rho-\lambda\theta}{2\rho(\gamma\theta-\lambda\theta)}\sum_{k=1}^{K}\Ebb[\|\nabla f_{1/\rho}(z^{k})\|^{2}]\nonumber \\
\le{} & f_{1/\rho}(z^{1})-\Ebb\big[f_{1/\rho}(z^{K+1})\big] +\frac{\rho(\gamma\beta+\rho\beta^{2}\theta^{-2})}{2(\gamma\theta-\lambda\theta)}\|x^{1}-x^{0}\|^{2}\nonumber \\
& +\frac{\rho L^{2}K}{(\gamma\theta^{2}-\rho\beta^{2}\theta^{-1})(\gamma\theta^{2}-\lambda\theta^{2})} \nonumber \\ 
\le{} & \Delta+\frac{\rho L^{2}K}{(\gamma\theta^{2}-\rho\beta^{2}\theta^{-1})(\gamma\theta^{2}-\lambda\theta^{2})},\label{eq:extra-mid-12}
\end{align}
where the last inequality uses $x^{1}=x^{0}=z^{1}$ and that $f_{1/\rho}(z^{1})-f_{1/\rho}(z^{K+1})\le f(z^{1})-\min_{x}f(x)=\Delta$.
Appealing to the definition of $k^{*}$ and relation~(\ref{eq:extra-mid-12}),
we have
\begin{align*}
& \Ebb\thinspace[\|\nabla f_{1/\rho}(z^{k^{*}})\|^{2}] \\
={} & \frac{1}{K}\sum_{k=1}^{K}\Ebb\thinspace[\|\nabla f_{1/\rho}(z^{k})\|^{2}]\\
 \leq{} &\frac{2\rho(\gamma\theta-\lambda\theta)\Delta}{(\rho-\lambda\theta)K}+\frac{2\rho^{2}L^{2}}{\theta(\rho-\lambda\theta)(\gamma\theta-\rho\beta^{2}\theta^{-2})}\\
 \le{} & \frac{2\rho}{\rho-\lambda}\bigg[\frac{(\gamma\theta-\lambda\theta)\Delta}{K}+\frac{\rho L^{2}}{\theta(\gamma\theta-\rho\beta^{2}\theta^{-2})}\bigg]\\
 ={} & \frac{2\rho}{\rho-\lambda}\bigg[\frac{(\rho\beta^{2}\theta^{-2}+\gamma_{0}\sqrt{K})\Delta}{K}+\frac{\rho L^{2}}{\theta(\gamma_{0}\sqrt{K}+\lambda\theta)}\bigg]\\
 \le & \frac{2\rho}{\rho-\lambda}\bigg[\frac{\rho\beta^{2}\theta^{-2}\Delta}{K}+\Big(\gamma_{0}\Delta+\frac{\rho L^{2}}{\theta\gamma_{0}}\Big)\frac{1}{\sqrt{K}}\bigg].
\end{align*}
where the first inequality uses the fact that $(\rho-\lambda\theta)^{-1}\le (\rho-\lambda)^{-1}$ for $\theta\in(0,1]$ and that  $\gamma=\gamma_{0}\theta^{-1}\sqrt{K}+\lambda+\rho\beta^{2}\theta^{-3}$.
Therefore, (\ref{eq:extra-main-3}) immediately follows.

\subsection{{\smod} with momentum and minibatching}\label{app:mm}
We present a new model-based method by  combining the momentum and minibatching techniques in a single framework.
\begin{algorithm}[h]
   \caption{Stochastic Extrapolated Model-Based Method with Minibatching}\label{alg:semod-mb}
\begin{algorithmic}
   \STATE {\bfseries Input:} $x^{0}$, $x^{1}$, $\beta$, $\gamma$
   \FOR{$k=1$ {\bfseries to} $K$}
   \STATE Sample a minibatch  $B_k=\{\xi_{k,1},\ldots, \xi_{k,m}\}$ and update:
   \begin{align}
y^{k} & =x^{k}+\beta(x^{k}-x^{k-1}) \label{eq:update-yk-2}\\
x^{k+1} & =\argmin_{x\in\Xcal} \left\{\ f_{x^{k}}(x,B_k)+\frac{\gamma}{2}\|x-y^{k}\|^{2} \right\} \label{eq:update-xk-2}
\end{align}
   \ENDFOR
\end{algorithmic}
\end{algorithm}

The convergence analysis of Algorithm~\ref{alg:semod-mb} is more complicated than that of the sequential extrapolated {\smod}. 
We require a different design of potential function:
\[
f_{1/\rho}(z^k)+ \alpha f(x^k) +\beta \|x^k-x^{k-1}\|^2
\]
where $\alpha$ and $\beta$ are some constants and $z_k$ is defined as in Section \ref{sec:momentum}.
 We summarize the approximate descent property in the following function.
\begin{lem}
\label{lem:extra-mb-2}In Algorithm~\ref{alg:semod-mb},  Assume that \ref{ass:two-sides-quad} holds and $\rho > 3(\tau+\lambda)$, 
then we have 
\begin{align}
 & \frac{\rho-\lambda\theta}{2\theta\rho(\gamma-\lambda)}\|\nabla f_{1/\rho}(z^{k})\|^{2}\nonumber \\
\le{} & f_{1/\rho}(z^{k})-\Ebb_{k}\big[f_{1/\rho}(z^{k+1})\big]+\frac{\rho\beta}{2\theta^{2}(\gamma-\lambda)}\big[f(x^{k})-\Ebb_{k}[f(x^{k+1})]\big]\nonumber \\
& -\frac{\rho(\gamma\theta^{2}-\zeta)}{4\theta^{2}(\gamma-\lambda)}\|x^{k+1}-x^{k}\|^{2}+\frac{\rho\vep}{2\theta^{2}(\gamma-\lambda)}\nonumber\\
 & +\frac{\rho(\gamma\beta+2\rho\beta^{2}\theta^{-2})}{2\theta(\gamma-\lambda)}\big\{\|x^{k}-x^{k-1}\|^{2}-\Ebb_{k}[\|x^{k+1}-x^{k}\|^{2}]\big\}.\label{eq:extra-mb-decent-prop}
\end{align}
where $\zeta=2\theta(\rho+\lambda\beta+\tau)+\tau+2\rho\beta^{2}\theta^{-1}$ and  $\vep=\frac{2L^2}{m(\gamma-\lambda)}$.
\end{lem}

\begin{proof}
Analogous to the relation (\ref{eq:extra-mid-17}), we have 
\begin{align}
 & f_{x^{k}}(x^{k+1},B_{k})+\frac{\gamma}{2}\|x^{k+1}-y^{k}\|^{2}\nonumber \\
\le{} & (1-\theta)f(x^{k},B_{k})+\theta\big[f(x,B_{k})+\frac{\lambda+\tau}{2}\|x-x^{k}\|^{2}\big]-\frac{\lambda\theta^{2}}{2}\|x-x^{k}\|^{2}\nonumber \\
 & +\frac{\gamma\theta^{2}}{2}\|x-z^{k}\|^{2}-\frac{(\gamma-\lambda)\theta^{2}}{2}\|x-z^{k+1}\|^{2}\label{eq:extra-mb-mid-17}
\end{align}
Placing the value $x=\hat{z}^{k}$, we arrive at
\begin{align}
 & f_{x^{k}}(x^{k+1},B_{k})+\frac{\gamma}{2}\|x^{k+1}-y^{k}\|^{2}\nonumber \\
\le{} & (1-\theta)f(x^{k},B_{k})+\theta f(\hat{z}^{k},B_{k})+\frac{(\lambda+\tau)\theta-\lambda\theta^{2}}{2}\|\hat{z}^{k}-x^{k}\|^{2}\nonumber \\
 & +\frac{\gamma\theta^{2}}{2}\|\hat{z}^{k}-z^{k}\|^{2}-\frac{(\gamma-\lambda)\theta^{2}}{2}\|\hat{z}^{k}-z^{k+1}\|^{2}\nonumber \\
\le{} & (1-\theta)f(x^{k},B_{k})+\theta f(\hat{z}^{k},B_{k})+\theta(\lambda\beta+\tau)\big[\|\hat{z}^{k}-x^{k+1}\|^{2}+\|x^{k}-x^{k+1}\|^{2}\big]\nonumber \\
 & +\frac{\gamma\theta^{2}}{2}\|\hat{z}^{k}-z^{k}\|^{2}-\frac{(\gamma-\lambda)\theta^{2}}{2}\|\hat{z}^{k}-z^{k+1}\|^{2}\label{eq:extra-mb-mid-17-1}
\end{align}
where the last inequality uses the fact $(\lambda+\tau)\theta-\lambda\theta^{2}=\theta(\lambda\beta+\tau)$
and applies $\|a+b\|^{2}\le2\|a\|^{2}+2\|b\|^{2}$ with $a=\hat{z}^{k}-x^{k+1}$
and $b=x^{k+1}-x^k$. 

Recall that $\hat{z}^{k}=\prox_{f/\rho}(z^{k})$. In view of Lemma~\ref{lem:three-point} and the $(\tau+\lambda)$-weak convexity of $f(\cdot)$, we have
\begin{align}
\theta f(\hat{z}^{k})+\frac{\rho\theta}{2}\|\hat{z}^{k}-z^{k}\|^{2} & \le \theta f(x^{k+1})+\frac{\rho\theta}{2}\|x^{k+1}-z^{k}\|^{2}-\frac{\theta(\rho-\tau-\lambda)}{2}\|x^{k+1}-\hat{z}^{k}\|^{2}.\label{eq:extra-mb-mid-8}
\end{align}
Summing up (\ref{eq:extra-mb-mid-17-1}) and (\ref{eq:extra-mb-mid-8}) and rearranging the terms, we arrive at
\begin{align}
 & \frac{\gamma}{2}\|x^{k+1}-y^{k}\|^{2}\nonumber \\
\le{} & (1-\theta)\big[f(x^{k},B_{k})-f(x^{k+1})\big]+\theta\big[f(\hat{z}^{k},B_{k})-f(\hat{z}^{k})\big]+f(x^{k+1})-f_{x^{k}}(x^{k+1},B_{k})\nonumber \\
 & +\theta(\lambda\beta+\tau)\|x^{k}-x^{k+1}\|^{2}\nonumber \\
 & +\frac{\gamma\theta^{2}-\rho\theta}{2}\|\hat{z}^{k}-z^{k}\|^{2}-\frac{(\gamma-\lambda)\theta^{2}}{2}\|\hat{z}^{k}-z^{k+1}\|^{2}\nonumber \\
 & +\frac{\rho\theta}{2}\|x^{k+1}-z^{k}\|^{2}-\frac{\theta(\rho-3(\tau+\lambda)+2\lambda\theta)}{2}\|x^{k+1}-\hat{z}^{k}\|^{2}\label{eq:mbm-mid-04}
\end{align}
On both sides of the above inequality, we take expectation over $B_{k}$ conditioned on all the randomness that generates $B_1, B_2,\ldots, B_{k-1}$ .
Noting that $\Ebb_{k}\big[f(x^{k},B_{k})\big]=f(x^{k})$ and $\Ebb_{k}\big[f(\hat{z}^{k},B_{k})\big]=f(\hat{z}^{k})$,  it follows that
\begin{align}
 & \frac{\gamma}{2}\Ebb_{k}\big[\|x^{k+1}-y^{k}\|^{2}\big]\nonumber \\
\le{} & (1-\theta)\big[f(x^{k})-\Ebb_{k}[f(x^{k+1})]\big]+\Ebb_{k}\big[f(x^{k+1})-f_{x^{k}}(x^{k+1},B_{k})\big]\nonumber \\
 & +\theta(\lambda\beta+\tau)\Ebb_{k}\|x^{k}-x^{k+1}\|^{2}+\frac{\gamma \theta^{2}-\rho\theta}{2}\|\hat{z}^{k}-z^{k}\|^{2}-\frac{(\gamma -\lambda)\theta^{2}}{2}\Ebb_{k}\|\hat{z}^{k}-z^{k+1}\|^{2}\nonumber \\
 & +\frac{\rho\theta}{2}\Ebb_{k}\|x^{k+1}-z^{k}\|^{2}-\frac{\theta(\rho-3(\tau+\lambda)+2\lambda\theta)}{2}\Ebb_{k}\|x^{k+1}-\hat{z}^{k}\|^{2}\label{eq:mbm-mid-05}
\end{align}
Moreover, similar to the analysis for minibatch {\smod}, we apply Theorem~\ref{thm:stability-gen} and Lemma~\ref{lem:model-uni-stable} to show that 
\[\Ebb_{k}\big\{\Ebb_{\xi}\big[f_{x^{k}}(x^{k+1},\xi)\big]-f_{x^{k}}(x^{k+1},B_{k})\big\} \le\vep.\]
In view of this result and  Assumption~\ref{ass:two-sides-quad}, we arrive at
\begin{align}
 & \Ebb_{k}\big[f(x^{k+1})-f_{x^{k}}(x^{k+1},B_{k})\big]\nonumber \\
={} & \Ebb_{k}\big[f(x^{k+1})-\Ebb_{\xi}[f_{x^{k}}(x^{k+1},\xi)]\big]+\Ebb_{k}\big\{\Ebb_{\xi}\big[f_{x^{k}}(x^{k+1},\xi)\big]-f_{x^{k}}(x^{k+1},B_{k})\big\}\nonumber \\
\le{} & \frac{\tau}{2}\Ebb_{k}[\|x^{k}-x^{k+1}\|^{2}]+\vep. \label{eq:mbm-mid-02}
\end{align}
Putting (\ref{eq:mbm-mid-05}) and (\ref{eq:mbm-mid-02}) together and using the assumption $\rho>3(\tau+\lambda)$, we  have
\begin{align}
 & \frac{\gamma}{2}\Ebb_{k}\big[\|x^{k+1}-y^{k}\|^{2}\big]\nonumber \\
\le{} & (1-\theta)\big[f(x^{k})-\Ebb_{k}[f(x^{k+1})]\big]+\frac{2\theta(\lambda\beta+\tau)+\tau}{2}\Ebb_{k}[\|x^{k}-x^{k+1}\|^{2}]+\vep \nonumber \\
 & +\frac{\gamma\theta^{2}-\rho\theta}{2}\|\hat{z}^{k}-z^{k}\|^{2}-\frac{(\gamma-\lambda)\theta^{2}}{2}\Ebb_{k}[\|\hat{z}^{k}-z^{k+1}\|^{2}]\nonumber \\
 & +\frac{\rho\theta}{2}\Ebb_{k}[\|x^{k+1}-z^{k}\|^{2}]\label{eq:mbm-mid-06}
\end{align}

Moreover, we can bound the term $\Ebb_{k}[\|x^{k+1}-y^{k}\|^{2}]$
\begin{align}
 & \|x^{k+1}-y^{k}\|^{2}\nonumber \\
={} & \|x^{k+1}-x^{k}\|^{2}+\beta^{2}\|x^{k}-x^{k-1}\|^{2}-2\beta\langle x^{k+1}-x^{k},x^{k}-x^{k-1}\rangle\nonumber \\
\ge{} & \|x^{k+1}-x^{k}\|^{2}+\beta^{2}\|x^{k}-x^{k-1}\|^{2}-\beta\|x^{k+1}-x^{k}\|^{2}-\beta\|x^{k}-x^{k-1}\|^{2}\nonumber \\
={} & \theta\|x^{k+1}-x^{k}\|^{2}-\beta\theta\|x^{k}-x^{k-1}\|^{2},\label{eq:mbm-mid-08}
\end{align}
and 
\begin{align}
\frac{\rho\theta}{2}\|x^{k+1}-z^{k}\|^{2} & =\frac{\rho\theta}{2}\|x^{k+1}-x^{k}-\beta\theta^{-1}(x^{k}-x^{k-1})\|^{2}\nonumber\\
 & \le\rho\theta\|x^{k+1}-x^{k}\|^{2}+\rho\beta^{2}\theta^{-1}\|x^{k}-x^{k-1}\|^{2} \label{eq:mbm-mid-07}
\end{align}
where the inequality comes from the fact that $\|a+b\|^2\le2\|a\|^2+2\|b\|^2$.

Putting (\ref{eq:mbm-mid-06}), (\ref{eq:mbm-mid-08}) and (\ref{eq:mbm-mid-07}) together, we have 
\begin{align*}
 & \frac{\gamma\theta^{2}-2\theta(\rho+\lambda\beta+\tau)-\tau-2\rho\beta^{2}\theta^{-1}}{2}\Ebb_{k}[\|x^{k+1}-x^{k}\|^{2}]\\
\le{} & (1-\theta)\big[f(x^{k})-\Ebb_{k}[f(x^{k+1})]\big]+\vep\\
 & +\frac{\gamma\theta^{2}-\rho\theta}{2}\|\hat{z}^{k}-z^{k}\|^{2}-\frac{(\gamma-\lambda)\theta^{2}}{2}\Ebb_{k}[\|\hat{z}^{k}-z^{k+1}\|^{2}]\\
 & +\frac{\gamma\beta\theta+2\rho\beta^{2}\theta^{-1}}{2}\Ebb_{k}\big[\|x^{k}-x^{k-1}\|^{2}-\|x^{k+1}-x^{k}\|^{2}\big]
\end{align*}
It then follows that 
\begin{align}
 & \Ebb_{k}[\|\hat{z}^{k}-z^{k+1}\|^{2}]\nonumber \\
\le{} & \|\hat{z}^{k}-z^{k}\|^{2}-\frac{\rho-\lambda\theta}{\gamma\theta-\lambda\theta}\|\hat{z}^{k}-z^{k}\|^{2}+\frac{\vep}{(\gamma-\lambda)\theta^{2}}\nonumber \\
 & +\frac{\beta}{(\gamma-\lambda)\theta^{2}}\big[f(x^{k})-\Ebb_{k}[f(x^{k+1})]\big]\nonumber \\
 & -\frac{\gamma\theta^{2}-2\theta(\rho+\lambda\beta+\tau)-\tau-2\rho\beta^{2}\theta^{-1}}{2(\gamma-\lambda)\theta^{2}}\Ebb_k[\|x^{k+1}-x^{k}\|^{2}]\nonumber \\
 & -\frac{\gamma\beta+2\rho\beta^{2}\theta^{-2}}{\gamma\theta-\lambda\theta}\big(\Ebb_{k}\big[\|x^{k+1}-x^{k}\|^{2}-\|x^{k}-x^{k-1}\|^{2}\big]\big)\label{eq:extra-mb-mid-11}
\end{align}
In view of (\ref{eq:extra-mb-mid-11}) and the definition of Moreau
envelope, we have 
\begin{align}
 & \Ebb_{k}\left[f_{1/\rho}(z^{k+1})\right]\nonumber \\
={} & \Ebb_{k}\big[f(\hat{z}^{k+1})+\frac{\rho}{2}\|z^{k+1}-\hat{z}^{k+1}\|^{2}\big]\nonumber \\
\le{} & \Ebb_{k}\big[f(\hat{z}^{k})+\frac{\rho}{2}\|z^{k+1}-\hat{z}^{k}\|^{2}\big]\nonumber \\
\le{} & f_{1/\rho}(z^{k})-\frac{\rho(\rho-\lambda\theta)}{2(\gamma\theta-\lambda\theta)}\|z^{k}-\hat{z}^{k}\|^{2}+\frac{\rho\vep}{2(\gamma\theta^{2}-\lambda\theta^{2})}+\frac{\rho\beta}{2(\gamma-\lambda)\theta^{2}}\big[f(x^{k})-\Ebb_{k}[f(x^{k+1})]\big]\nonumber \\
 & -\frac{\rho(\gamma\theta^{2}-2\theta(\rho+\lambda\beta+\tau)-\tau-2\rho\beta^{2}\theta^{-1})}{4(\gamma-\lambda)\theta^{2}}\|x^{k+1}-x^{k}\|^{2}\nonumber \\
 & +\frac{\rho(\gamma\beta+2\rho\beta^{2}\theta^{-2})}{2(\gamma\theta-\lambda\theta)}\big\{\|x^{k}-x^{k-1}\|^{2}-\Ebb_{k}\big[\|x^{k+1}-x^{k}\|^{2}\big]\big\}.\label{eq:extra-mb-mid-6}
\end{align}

In view of the above result and the relation $\|z^{k}-\hat{z}^{k}\|^{2}=\rho^{-2}\|\nabla_{1/\rho}f(z^{k})\|^{2}$,
we obtain\,(\ref{eq:extra-mb-decent-prop}).
\end{proof}
\begin{thm}
\label{thm:extra-mb-momentum}Suppose we choose 
$
\gamma=\gamma_0\sqrt{\frac{K}{m}}+\theta^{-2}\zeta+\lambda
$, where $\zeta$ is defined in \ref{lem:extra-mb-2}.
Then we have
\[
\Ebb\thinspace[\|\nabla f_{1/\rho}(z^{k^{*}})\|^{2}]\le\frac{\rho}{\rho-\theta\lambda}\bigg[\frac{\theta^{-1}(\rho\beta+2\zeta)\Delta}{K}+ \Big(\theta\gamma_0\Delta+\frac{\rho L^2}{\theta\gamma_0}\Big)\frac{2}{\sqrt{mK}}\bigg].
\]
\end{thm}

\begin{proof}
Unfolding the relation\,(\ref{eq:extra-mb-decent-prop}) and then
taking expectation over all the randomness, we have 
\begin{align}
 & \frac{\rho-\lambda\theta}{2\theta\rho(\gamma-\lambda)}\sum_{k=1}^{K}\Ebb[\|\nabla f_{1/\rho}(z^{k})\|^{2}]\nonumber \\
\le{} & f_{1/\rho}(z^{1})-\Ebb\big[f_{1/\rho}(z^{K+1})\big]+\frac{\rho\beta}{2\theta^{2}(\gamma-\lambda)}\big[f(x^{1})-\Ebb_{k}[f(x^{K+1})]\big] \nonumber\\
 & +\frac{\rho\vep K}{2\theta^{2}(\gamma-\lambda)} +\frac{\rho(\gamma\beta+2\rho\beta^{2}\theta^{-2})}{2\theta(\gamma-\lambda)}\|x^{1}-x^{0}\|^{2}.\nonumber\\
 \le & \Big(1+ \frac{\rho\beta}{2\theta^{2}(\gamma-\lambda)}\Big)\Delta + \frac{L^2\rho K}{\theta^2m(\gamma-\lambda)^2},
 \label{eq:extra-mid-12}
\end{align}

where we use the assumption $x^{1}=x^{0}=z^{1}$ and that
\[
\max\left\{f_{1/\rho}(z^{1})-f_{1/\rho}(z^{K+1}),\, f_{1/\rho}(x^{1})-f_{1/\rho}(x^{K+1}) \right\} \le\Delta.\\
\]

Appealing to the definition of $k^{*}$, $\gamma$  and then using relation~(\ref{eq:extra-mid-12}),
we arrive at
\begin{align*}
 & \Ebb\thinspace[\|\nabla f_{1/\rho}(z^{k^{*}})\|^{2}]\\
 \le {} & \frac{\rho}{\rho-\theta\lambda}\bigg[\frac{\rho \beta\theta^{-1}\Delta}{K}+\frac{2\theta(\gamma-\lambda)\Delta}{K}+\frac{2\rho L^{2}}{\theta m(\gamma-\lambda)}\bigg]\\
\le{} & \frac{\rho}{\rho-\theta\lambda}\bigg[\frac{\theta^{-1}(\rho\beta+2\zeta)\Delta}{K}+ \Big(\theta\gamma_0\Delta+\frac{\rho L^2}{\theta\gamma_0}\Big)\frac{2}{\sqrt{mK}}\bigg].
\end{align*}
\end{proof}
\begin{rem}
While the  convergence result in Theorem~\ref{thm:extra-mb-momentum} is established for all $\gamma_{0}>0$,
we can see that the optimal $\gamma_{0}$ would be $\gamma_{0}=\theta^{-1}\sqrt{\frac{\rho}{\Delta}}L$,
which gives the bound 
$
\Ebb[\thinspace\|\nabla f_{1/\rho}(z^{k^{*}})\|^{2}]=\Ocal\big(\frac{\Delta}{K}+L\sqrt{\frac{\rho\Delta}{mK}}\big).
$
In practice we can set $\gamma_{0}$ to a suboptimal value and obtain a possibly loose upper-bound. 
\end{rem}

\newpage
\section{{\smod} for convex optimization}\label{app:cvx}
In this section, we develop new complexity results of model-based methods for stochastic convex optimization. 
To provide the sharpest convergence rate possible, we replace Assumption~\ref{ass:two-sides-quad} with the following assumption
\begin{enumerate}[label=\textbf{A\arabic*:},ref=A\arabic*,start=6]
\item For any $x\in\Xcal$, $f_{x}(\cdot,\xi)$ is a convex function,\label{cvx:two-sided} and  \
\begin{equation}\label{eq:cvx-two-sided}
-\frac{\tau}{2}\|x-y\|^{2}\le f_{x}(y,\xi)-f(y,\xi)\le0,\quad\xi\in\Xi, y\in\Xcal.
\end{equation}
\end{enumerate}

It is easy to see that Assumption~\ref{cvx:two-sided} ensures the convexity of  $f(y,\xi)$. More specifically, let $\bar{x}=(1-\alpha)x + \alpha y$ where $x,y\in\Xcal$ and $\alpha\in[0,1]$,  we have
\begin{align*}
	f(\bar{x},\xi) & =	f_{\bar{x}}(\bar{x},\xi) \\
	& \le (1-\alpha) f_{\bar{x}}({x},\xi) + \alpha f_{\bar{x}}(y,\xi) \\
	& \le (1-\alpha) f({x},\xi) + \alpha f(y,\xi) 
\end{align*}
where the equality comes from Assumption~\ref{ass:unbiased}, the first inequality follows from convexity of $f_{\bar{x}}(\cdot,\xi)$ and the second inequality uses (\ref{eq:cvx-two-sided}).

\textbf{Outline of this section.} Since convergence to global optimality can be guaranteed in convex optimization, it is  favorable to describe convergence rates with respect to the optimality gap. To this end, we conduct new convergence analysis of  {\smod} with minibatching and momentum for stochastic convex  optimization. In subsection~\ref{subsec:semod}, we show that  under the additional Assumption~\ref{cvx:two-sided}, after $K$ iterations of the extrapolated minibatch method (Algorithm~\ref{alg:semod-mb}),  the expected optimality gap converges at rate
\[
\Ocal\Big(\frac{1}{K}+\frac{1}{\sqrt{mK}}\Big).
\]
In view of the above result, the deterministic part of our rate is consistent with the best  $\Ocal{(\frac{1}{K}})$ rate for heavy-ball method. For example, see \citepapp{diakonikolas2021generalizeddup, ghadimi2015globaldup}.
Moreover, the stochastic part of the rate is improved from the result $\Ocal(\frac{1}{\sqrt{K}})$ of Theorem~4.4~\citepapp{davis2019stochasticweaklydup} by a factor of $\sqrt{m}$. 

As is mentioned in the main article, one major advantage of {\smod} methods is the robustness to stepsize selection (see \citep{asi2019stochasticdup}). In other words, compared to {\sgd}, {\spl} and {\spp} tend to admit a wider range of stepsizes. In subsection~\ref{subsec:robust}, we show that the extrapolated model-based method inherits the merits of robustness from the model-based method.

An important question arises naturally: Can we further improve the  convergence rate of model-based methods?  Due to the widely known limitation of heavy-ball type momentum, it would be interesting to consider Nesterov's acceleration.  In subsection~\ref{subsec:nes}, we present a model-based method with Nesterov type momentum. Thanks to the stability argument, we obtain the following improved rate of convergence: 
\[
\Ocal\Big(\frac{1}{K^2}+\frac{1}{\sqrt{mK}}\Big).
\]
We note that a similar convergence rate for minibatch model-based methods is obtained in a recent paper~\citepapp{chadha2021accelerateddup}. However, their result requires the assumption that the stochastic function is Lipschitz smooth while our assumption is much weaker.
The full complexity results are presented in Table~\ref{tab:cvx-bound}.

\begin{table}[h]
\caption{Complexity of stochastic algorithms to reach $\protect\vep$-accuracy:
$\protect\Ebb[f(x)-f(x^{*})]\le\protect\vep$. (M: minibatching; E:
Extrapolation (Polyak type); N: Nesterov acceleration\label{tab:cvx-bound}}
\centering{}%
\begin{tabular}{cccc}
\hline 
Algorithms & Problems & Current Best & Ours\tabularnewline
\hline 
M + {\smod} & $f$: smooth composite & $\Ocal(1/\vep^{2})$ \citeapp{davis2019stochasticweaklydup} & $\Ocal(1/\vep+1/(m\vep^{2}))$\tabularnewline
M + E + {\smod} & $f$: non-smooth & $\Ocal(1/\vep^{2})$ \citeapp{davis2019stochasticweaklydup}  & $\Ocal(1/\vep+1/(m\vep^{2}))$\tabularnewline
\hline 
M + N + {\smod} & $f$: smooth composite & $\Ocal(1/\vep^{1/2}+1/(m\vep^{2}))$ \citepapp{chadha2021accelerateddup} & $\Ocal(1/\vep^{1/2}+1/(m\vep^{2}))$\tabularnewline
M + N + {\smod} & $f$: non-smooth & --- & $\Ocal(1/\vep^{1/2}+1/(m\vep^{2}))$\tabularnewline
\hline 
\end{tabular}
\end{table}

\subsection{Convergence of extrapolated {\smod}}\label{subsec:semod}

The following Lemma summarizes some important convergence property of Extrapolated {\smod} for convex stochastic optimization.
\begin{lem}
\label{lem:smod-extra-recurs}Under Assumption~\ref{cvx:two-sided}, let $\theta=1-\beta$  in Algorithm~\ref{alg:semod-mb}. Then for any $\hat{x} \in \Xcal$ and $k=1,2,3,\ldots$, we have
\begin{equation}
\begin{aligned} & \Ebb_{k}\big[f(x^{k+1})-f(\hat{x})\big]-(1-\theta)\big[f(x^k)-f(\hat{x})\big]\\
\le{} & \frac{2L^2}{m\gamma}+\frac{\gamma\theta^{2}}{2}\|\hat{x}-z^{k}\|^{2}-\frac{\gamma\theta^{2}}{2}\Ebb_{k}[\|\hat{x}-z^{k+1}\|^{2}]\\
 & +\frac{\gamma\beta(1-\beta)}{2}\|x^{k}-x^{k-1}\|^{2}-\frac{\gamma(1-\beta)-\tau}{2}\Ebb_{k}[\|x^{k+1}-x^{k}\|^{2}]
\end{aligned}
\label{eq:smod-extra-recursion}
\end{equation}
\end{lem}

\begin{proof}
Applying three point lemma, for any $x\in \Xcal$, we have
\begin{align}
\begin{aligned}f_{x^{k}}(x^{k+1},B_{k})-f_{x^{k}}(x,B_{k}) & \le\frac{\gamma}{2}\|x-y^{k}\|^{2}-\frac{\gamma}{2}\|x-x^{k+1}\|^{2}-\frac{\gamma}{2}\|y^{k}-x^{k+1}\|^{2}.\end{aligned}
\label{eq:smod-extra-01}
\end{align}

Based on  Assumption~\ref{cvx:two-sided}, we have
\begin{align}
 & f(x^{k+1})-f_{x^{k}}(x^{k+1},B_{k})\nonumber \\
={} & \Ebb_{\xi}\big[f(x^{k+1},\xi)\big]-f_{x^{k}}(x^{k+1},B_{k})\nonumber \\
={} & \Ebb_{\xi}\big[f(x^{k+1},\xi)-f_{x^{k}}(x^{k+1},\xi)\big]+\Ebb_{\xi}\big[f_{x^{k}}(x^{k+1},\xi)-f_{x^{k}}(x^{k+1},B_{k})\big]\nonumber \\
\le{} & \frac{\tau}{2}\|x^{k}-x^{k+1}\|^{2}+\Ebb_{\xi}\big[f_{x^{k}}(x^{k+1},\xi)-f_{x^{k}}(x^{k+1},B_{k})\big].\label{eq:smod-extra-02}
\end{align}
Plugging the above into (\ref{eq:smod-extra-01}), we have that 
\[
\begin{aligned}f(x^{k+1})-f_{x^{k}}(x,B_{k}) & \le\frac{\gamma}{2}\|x-y^{k}\|^{2}-\frac{\gamma}{2}\|x-x^{k+1}\|^{2}-\frac{\gamma}{2}\|y^{k}-x^{k+1}\|^{2}\\
 & \quad+\frac{\tau}{2}\|x^{k}-x^{k+1}\|^{2}+\Ebb_{\xi}\big[f_{x^{k}}(x^{k+1},\xi)-f_{x^{k}}(x^{k+1},B_{k})\big].
\end{aligned}
\]
Let $x=(1-\theta)x^{k}+\theta\hat{x}$ and $z^{k}=x^{k}+\theta^{-1}\beta(x^{k}-x^{k-1})$. Then we have
\[
\begin{aligned}x-y^{k} & =\theta(\hat{x}-z^{k}),\\
x-x^{k+1} & =\theta(\hat{x}-z^{k+1}),
\end{aligned}
\]
and by convexity, we obtain that
\begin{equation}
\begin{aligned} & f(x^{k+1})-f(\hat{x},B_{k})-(1-\theta)\big[f(x^{k},B_{k})-f(\hat{x},B_{k})\big]\\
\le{} & f(x^{k+1})-f_{x^{k}}(x,B_{k})\\
\le{} & \frac{\gamma\theta^{2}}{2}\|\hat{x}-z^{k}\|^{2}-\frac{\gamma\theta^{2}}{2}\|\hat{x}-z^{k+1}\|^{2}-\frac{\gamma}{2}\|y^{k}-x^{k+1}\|^{2}\\
 & +\frac{\tau}{2}\|x^{k}-x^{k+1}\|^{2}+\Ebb_{\xi}\big[f_{x^{k}}(x^{k+1},\xi)-f_{x^{k}}(x^{k+1},B_{k})\big].
\end{aligned}
\label{eq:smod-extra-04}
\end{equation}
Then we have 
\begin{align} 
& -\frac{\gamma}{2}\|y^{k}-x^{k+1}\|^{2}+\frac{\tau}{2}\|x^{k}-x^{k+1}\|^{2}\nonumber\\
={} & -\frac{\gamma}{2}\|x^{k+1}-x^{k}\|^{2}+\gamma\beta\big\langle x^{k+1}-x^{k},x^{k}-x^{k-1}\big\rangle-\frac{\gamma\beta^{2}}{2}\|x^{k}-x^{k-1}\|^{2}+\frac{\tau}{2}\|x^{k}-x^{k+1}\|^{2}\nonumber\\
\le{} & \frac{\gamma\beta(1-\beta)}{2}\|x^{k}-x^{k-1}\|^{2}-\frac{\gamma(1-\beta)-\tau}{2}\|x^{k+1}-x^{k}\|^{2},\label{eq:smod-extra-07}
\end{align}
where the last inequality is by Cauchy-Schwarz and we deduce that
\[
\begin{aligned} & f(x^{k+1})-f(\hat{x},B_{k})-(1-\theta)\big[f(x^{k},B_{k})-f(\hat{x},B_{k})\big]\\
\le{} & \frac{\gamma\theta^{2}}{2}\|\hat{x}-z^{k}\|^{2}-\frac{\gamma\theta^{2}}{2}\|\hat{x}-z^{k+1}\|^{2}\\
 & +\frac{\gamma\beta(1-\beta)}{2}\|x^{k}-x^{k-1}\|^{2}-\frac{\gamma(1-\beta)-\tau}{2}\|x^{k+1}-x^{k}\|^{2}\\
 & +\Ebb_{\xi}\big[f_{x^{k}}(x^{k+1},\xi)-f_{x^{k}}(x^{k+1},B_{k})\big].
\end{aligned}
\]
Next, we take expectation over $B_{k}$ conditioned on $B_{1},B_{2},\ldots,B_{k-1}$. Note that $\Ebb_{k}[f(\hat{x},B_{k})]=f(\hat{x})$, $\Ebb_{k}[f(x^{k},B_{k})]=f(x^{k})$ and 
\begin{equation}
\begin{aligned} & \Ebb_{k}\big[f(x^{k+1})-f(\hat{x})\big]-(1-\theta)\big[f(x^{k})-f(\hat{x})\big]\\
\le{} & \frac{\gamma\theta^{2}}{2}\|\hat{x}-z^{k}\|^{2}-\frac{\gamma\theta^{2}}{2}\Ebb_{k}[\|\hat{x}-z^{k+1}\|^{2}]\\
 & +\frac{\gamma\beta(1-\beta)}{2}\|x^{k}-x^{k-1}\|^{2}-\frac{\gamma(1-\beta)-\tau}{2}\Ebb_{k}[\|x^{k+1}-x^{k}\|^{2}]\\
 & +\Ebb_{k}\big\{\Ebb_{\xi}\big[f_{x^{k}}(x^{k+1},\xi)-f_{x^{k}}(x^{k+1},B_{k})\big]\big\}.
\end{aligned}
\label{eq:smod-extra-05}
\end{equation}
Moreover, based on the stability of the proximal mapping, we have 
\begin{equation}
\Ebb_{k}\big\{\Ebb_{\xi}\big[f_{x^{k}}(x^{k+1},\xi)-f_{x^{k}}(x^{k+1},B_{k})\big]\big\}\le\vep_{k},\ \textrm{where }\vep_{k}=\frac{2L^{2}}{m\gamma}.\label{eq:smod-extra-06}
\end{equation}
Combining (\ref{eq:smod-extra-05}) and (\ref{eq:smod-extra-06})
gives the desired result (\ref{eq:smod-extra-recursion}).
\end{proof}

By specifying a constant stepsize and batch size, we develop the convergence rate of {\extra} in the following Theorem.
\begin{thm}
\label{thm:smod-extra-cvx-rate}Let $x^{1}=x^{0}$, $x^{*}$ 
be an optimal solution and $\gamma=\gamma_{0}\sqrt{\frac{K}{m}}+\theta^{-2}\tau$, where $\gamma_0=\frac{2\theta^{-1}L}{\til{D}}$ and $\til{D}\ge \|x^0-x^*\|$, then we have 
\begin{equation}
\Ebb\big[f(x^{k^{*}})-f(x^{*})\big]
\le \frac{f(x^0)-f(x^*)}{K} + \frac{\theta^{-1}\tau {\til{D}}^2}{2K}+\frac{2\til DL}{\sqrt{mK}}. \label{eq:app-bound-1}
\end{equation}
where $k^{*}$ is an index chosen in $\{1,2,\ldots, K\}$ uniformly at random.
\end{thm}

\begin{proof} Let us denote $\Delta_k = \Ebb[f(x^k)-f(x^*)]$ for the sake of simplicity. 
Following Lemma~\ref{lem:smod-extra-recurs} (with $\hat{x}=x^*$), we sum up (\ref{eq:smod-extra-recursion})
over $k=1,2,\ldots,K$ and then take expectation over all the randomness, then
we have
\[
\Delta_{K+1} + \theta\sum_{k=1}^{K}\Delta_{k}\le\Delta_{1}+\frac{\gamma\theta^{2}}{2}\|\hat{x}-z^{1}\|^{2}+\frac{\gamma\beta(1-\beta)}{2}\|x^{1}-x^{0}\|^{2}+\frac{2L^{2}K}{m\gamma},
\]
where the inequality holds since $\gamma \geq \theta^{-2} \tau$.
Using  $x^{1}=z^1=x^{0}$, we have 
\begin{align*}
\Ebb\big[f(x^{k^{*}})-f(x^{*})\big] & =\frac{1}{K}\sum_{k=1}^{K} \Delta_{k} \\
 & \le \frac{\Delta_1}{K} + \frac{\gamma\theta}{2K}\|x^{*}-x^{0}\|^{2}+\frac{2L^{2}}{m\theta\gamma}\\
 & \le \frac{\Delta_1}{K} + \frac{\gamma\theta{\til{D}}^2}{2K}+\frac{2L^{2}}{m\theta\gamma}\\
  & \le \frac{\Delta_1}{K} + \frac{\theta^{-1}\tau {\til{D}}^2}{2K} +  \frac{\theta\gamma_0{\til{D}}^2}{2\sqrt{mK}}+\frac{2L^{2}}{\sqrt{mK}\theta\gamma_0}\\
 & =   \frac{\Delta_1}{K} + \frac{\theta^{-1}\tau {\til{D}}^2}{2K}+\frac{2\til DL}{\sqrt{mK}}.
\end{align*}
Therefore, we complete the proof.
\end{proof}

\subsection{Robustness of the extrapolated {\smod}}\label{subsec:robust}
As is mentioned in the main article, one major advantage of {\smod} methods is the robustness to stepsize selection (see \citep{asi2019stochasticdup}). In other words, compared to {\sgd}, {\spl} and {\spp} tend to admit a wider range of stepsizes. We show that the extrapolated model-based method inherit the merits of robustness. 
 For the sake of the asymptotic analysis, stepsize parameter $\gamma_k$ in {\extra} is now indexed by $k$. We present the main convergence property in the following theorem.
\begin{thm}\label{thm:expect-dist}
Suppose that Assumption~\ref{cvx:two-sided} holds, $x^1=x^0$ and the stepsize $\gamma_k$ satisfies $ \gamma_k \geq 2\tau\theta^{-2}$. Then we have
	\[ \mathbb{E} [\| x^{\ast} - x^{K + 1} \|^2] \leq \| {x}^* - x^1
   \|^2  + 2{\theta^{-2}}
   \sum_{k = 1}^K \gamma_k^{-2}\mathbb{E} [\| f' (x^{\ast}, B_k) \|^2].\]
\end{thm}
\begin{proof}
First, Assumption~\ref{cvx:two-sided} implies that
\begin{align*}
  f_{x^k} (x^{k + 1}, B_k) & \geq f (x^{k + 1}, B_k) - \frac{\tau}{2} \|
  x^{k + 1} - x^k \|^2,\\
  - f_{x^k} (x, B_k) & \geq - f (x, B_k).
\end{align*}
Summing up the above two relations gives
\[ f_{x^k} (x^{k + 1}, B_k) - f_{x^k} (x, B_k) \geq f (x^{k + 1}, B_k) - f (x,
   B_k) - \frac{\tau}{2} \| x^{k + 1} - x^k \|^2 . \]
In view of (\ref{eq:smod-extra-01}), we have
\begin{align*}
  f (x^{k + 1}, B_k) - f (x, B_k) & \leq f_{x^k} (x^{k + 1}, B_k) - f_{x^k}
  (x, B_k) + \frac{\tau}{2} \| x^{k + 1} - x^k \|^2\\
  & \leq \frac{\gamma_k}{2} \| x - y^k \|^2 - \frac{\gamma_k}{2} \| x - x^{k +
  1} \|^2 - \frac{\gamma_k}{2} \| y^k - x^{k + 1} \|^2 + \frac{\tau}{2} \| x^{k
  + 1} - x^k \|^2,
\end{align*}
which implies that
\begin{eqnarray*}
  \frac{\gamma_k}{2} \| x - x^{k + 1} \|^2 \leq \frac{\gamma_k}{2} \| x - y^k
  \|^2 + \frac{\tau}{2} \| x^{k + 1} - x^k \|^2 - \frac{\gamma_k}{2} \| y^k -
  x^{k + 1} \|^2 - [f (x^{k + 1}, B_k) - f (x, B_k)].
\end{eqnarray*}
By the convexity of $f (\cdummy, B_k)$, we have, for any $\eta > 0$ that
\begin{align*}
  & \quad f (x^{k + 1}, B_k) - f (x, B_k)\\
  & \geq \langle f' (x, B_k), x^{k + 1} - x \rangle\\
  & = \langle f' (x, B_k), x^k - x \rangle + \langle f' (x, B_k), x^{k + 1}
  - x^k \rangle\\
  & \geq \langle f' (x, B_k), x^k - x \rangle - \| f' (x, B_k) \| \| x^{k +
  1} - x^k \|\\
  & \geq \langle f' (x, B_k), x^k - x \rangle - \frac{1}{2 \eta \gamma_k} \|
  f' (x, B_k) \|^2 - \frac{\eta \gamma_k}{2} \| x^{k + 1} - x^k \|^2.
\end{align*}
Recalling the identities $ x - y^k  =\theta (\hat{x} - z^k)$ and $x - x^{k + 1} =\theta (\hat{x} - z^{k + 1})$, we have
\begin{align*}
  &  \quad \frac{\gamma_k \theta^2}{2} \| \hat{x} - z^{k + 1} \|^2\\
  & = \frac{\gamma_k}{2} \| x - x^{k + 1} \|^2\\
  & \leq \frac{\gamma_k}{2} \| x - y^k \|^2 + \frac{\tau}{2} \| x^{k + 1} -
  x^k \|^2 - \frac{\gamma_k}{2} \| y^k - x^{k + 1} \|^2 - [f (x^{k + 1}, B_k) -
  f (x, B_k)]\\
  & \leq \frac{\gamma_k}{2} \| x - y^k \|^2 + \frac{\tau}{2} \| x^{k + 1} -
  x^k \|^2 - \frac{\gamma_k}{2} \| y^k - x^{k + 1} \|^2 + \frac{\eta \gamma_k}{2}
  \| x^{k + 1} - x^k \|^2 \\
  &  \quad ~ - \langle f' (x, B_k), x^k - x \rangle + \frac{1}{2 \eta \gamma_k} \| f' (x, B_k) \|^2\\
  & = \frac{\gamma_k \theta^2}{2} \| \hat{x} - z^k \|^2 + \frac{\tau + \eta
  \gamma_k}{2} \| x^{k + 1} - x^k \|^2 - \frac{\gamma_k}{2} \| y^k - x^{k + 1}
  \|^2\\
  &  \quad ~ - \langle f' (x, B_k), x^k - x \rangle + \frac{1}{2 \eta \gamma_k} \| f' (x, B_k) \|^2.
\end{align*}
Moreover, using an argument of \eqref{eq:smod-extra-07}, we obtain
\begin{align*}
  & \quad \frac{\gamma_k}{2} \| y^k - x^{k + 1} \|^2 + \frac{\tau + \eta
  \gamma_k}{2} \| x^{k + 1} - x^k \|^2\nonumber\\
  & = - \frac{\gamma_k}{2} \| y^k - x^{k + 1} \|^2 + \frac{\tau + \eta
  \gamma_k}{2} \| x^{k + 1} - x^k \|^2\nonumber\\
  & \leq \frac{\gamma_k \beta (1 - \beta)}{2} \| x^k - x^{k - 1} \|^2 -
  \frac{\gamma_k (1 - \beta - \eta) - \tau}{2} \| x^{k + 1} - x^k \|^2.
\end{align*}
Combining the above two results and  taking expectation $\Ebb_k[\cdot]$, we have that
\begin{align*}
  & \quad \frac{\gamma_k \theta^2}{2} \mathbb{E}_k [\| \hat{x} - z^{k + 1} \|^2]\\
  & \leq \frac{\gamma_k \theta^2}{2} \| \hat{x} - z^k \|^2 -\mathbb{E}_k
  [\langle f' (x, B_k), x^k - x \rangle] + \frac{1}{2 \eta \gamma_k}
  \mathbb{E}_k [\| f' (x, B_k) \|^2]\\
  &  \quad ~+ \frac{\tau + \eta \gamma_k}{2} \mathbb{E}_k [\| x^{k + 1} - x^k \|^2]
  - \frac{\gamma_k}{2} \mathbb{E}_k [\| y^k - x^{k + 1} \|^2]\\
  & \leq \frac{\gamma_k \theta^2}{2} \| \hat{x} - z^k \|^2 -\mathbb{E}_k
  [\langle f' (x, B_k), x^k - x \rangle] + \frac{1}{2 \eta \gamma_k} \| f' (x,
  B_k) \|^2 + \frac{\gamma_k \beta (1 - \beta)}{2} \| x^k - x^{k - 1} \|^2\\
  &  \quad ~- \frac{\gamma_k (1 - \beta - \eta) - \tau}{2} \mathbb{E}_k [\| x^{k +
  1} - x^k \|^2].
\end{align*}
Dividing both sides of the above relation by ${\gamma_k \theta ^ 2}/2$ and taking take $x = x^{\ast}$ gives
\begin{align*}
  & \quad \mathbb{E}_k [\| \hat{x} - z^{k + 1} \|^2]\\
  & \leq \| \hat{x} - z^k \|^2 - \frac{2}{\gamma_k \theta^2} \mathbb{E}_k
  [\langle f' (x, B_k), x^k - x^* \rangle] + \frac{1}{\eta \gamma_k^2 \theta^2}
  \mathbb{E}_k [\| f' (x^*, B_k) \|^2] \\
  &  \quad ~ + \frac{\beta (1 - \beta)}{\theta^2} \|
  x^k - x^{k - 1} \|^2- \frac{(1 - \beta - \eta) - \tau / \gamma_k}{\theta^2} \mathbb{E}_k [\|
  x^{k + 1} - x^k \|^2] \\
    & \leq \| \hat{x} - z^k \|^2 + \frac{1}{\eta \gamma_k^2 \theta^2}
  \mathbb{E}_k [\| f' (x^{\ast}, B_k) \|^2] + \frac{\beta (1 -
  \beta)}{\theta^2} \| x^k - x^{k - 1} \|^2\\
  &  \quad ~- \frac{(1 - \beta - \eta) - \tau / \gamma_k}{\theta^2} \mathbb{E}_k [\|
  x^{k + 1} - x^k \|^2].
\end{align*}
where the last inequality uses the property $\mathbb{E}_k [\langle f' (x^{\ast}, B_k),
x^k - x^{\ast} \rangle] = \langle f' (x^{\ast}), x^k - x^{\ast} \rangle \leq
0$, which is derived from optimality condition.

Last we take $\eta=\theta^2/2$,  $\gamma_k \geq \frac{\tau}{(1 - \beta)^2 - \eta}$ such that $\beta (1 - \beta) \leq (1 - \beta - \eta) - \tau / \gamma_k $ and sum over $k = 1, \ldots, K$ to obtain
\[
  \mathbb{E} [\| \hat{x} - z^{K + 1} \|^2] \leq \| \hat{x} - z^1 \|^2 +
  \frac{\beta (1 - \beta)}{\theta^2} \| x^1 - x^0 \|^2 + \frac{2}{
   \theta^4} \sum_{k = 1}^K \gamma_k^{-2}\mathbb{E} [\| f' (x^{\ast}, B_k) \|^2].
   \]
Plugging  $\|\hat{x} - z^{K + 1} \|^2 = \frac{1}{\theta^2} \|
x^{\ast} - x^{K + 1} \|^2$ in the above inequality and then multiplying both sides by $\theta^2$, we obtain the desired result.
\end{proof}
\begin{rem}
Let $\mathcal{X}^*$ be the set of optimal solutions and assume that $\sup_{x\in \mathcal{X}^*} \Ebb_k [\|f'(x^*,B_k)\|^2]< \infty$. Using an argument of Cor~3.2~\cite{asi2019stochasticdup}, we can show that when $\sum_{k=1}^\infty\gamma_k^{-2}<\infty$, then $\sup_k \textrm{dist}(x^k,X^*)<\infty$ with probability one. This completes our proof of the boundedness of the iterates.

It is interesting to compare {\extra} and  {\sgd} in terms of the robustness to the stepsize policy. Consider that {\sgd} takes the form $x^{k+1}=\argmin_x \langle f'(x^k,B_k), x\rangle +\frac{\gamma_k}{2}\|x-x^k\|^2.$
Using the argument of \citep{nemirovski2009dup}, it is easy to show that {\sgd} exhibits the bound
	\[ \mathbb{E} [\| x^{\ast} - x^{K + 1} \|^2] \leq  \| {x}^\ast - x^1
   \|^2 + 
   \sum_{k = 1}^K \gamma_k^{-2}\mathbb{E} [\| f' (x^{k}, B_k) \|^2],\]
which explicitly depends on  the subgradients of iterates $\{x^k\}$. When $\|f'(x^k,B_k)\|$ is large, (e.g. $f$ is a high order polynomial or an exponential function)
we need sufficiently large $\{\gamma_k\}$ (i.e. small stepsize 
$1/\gamma_k$) to ensure the boundedness of iterates.
However, in contrast to {\sgd}, {\extra} has a bound only depending on the subgradient over the optimal solutions.
For many problems, (e.g. interpolation problems), $\|f'(x^*,\xi)\|$ can be substantially smaller than $\sup_{x}\|f'(x,\xi)\|$.

 We also note that the best bound for SMOD is when  $\theta=1$ (i.e. $\beta=0$). It appears that adding momentum encourages more exploration of the parameter space, however, at the cost of potentially departing from the original solution path.	
 \end{rem}
\subsection{Improved convergence using Nesterov acceleration\label{subsec:nes}}

It is known that the heavy-ball type stochastic gradient does not give an optimal rate of convergence. Next we show that our proposed stability analysis can be combined with Nesterov's acceleration \citep{lan2012optimaldup}, yielding  an accelerated {\smod} method which achieves the best complexity for convex stochastic optimization.

\begin{algorithm}[H]
\caption{Stochastic Model-based Method with Minibatching and Nesterov's Acceleration \label{alg:smod-cvx-nes}}

\begin{algorithmic}
 \STATE {\bfseries Input:} $x^{0}=z^{0}$;
\FOR{$k=0$ {\bfseries to} $K$}
  \STATE Sample a minibatch $B_{k}=\{\xi_{k,1},\ldots,\xi_{k,m_{k}}\}$  and update $y^{k}$, $z^{k+1}$, $x^{k+1}$
by 
\begin{align*}
y^{k} & =(1-\theta_{k})x^{k}+\theta_{k}z^{k},\\
z^{k+1} & =\argmin_{x\in\Xcal}\,\Big\{ f_{y^{k}}(x,B_{k})+\frac{\gamma_{k}}{2}\|x-z^{k}\|^{2}\Big\},\\
x^{k+1} & =(1-\theta_{k})x^{k}+\theta_{k}z^{k+1}.
\end{align*}
 \ENDFOR
\end{algorithmic}
\end{algorithm}

\begin{lem}
\label{lem:cmodcvx-main-recurs}Let $\Delta_{k}\triangleq f(x^{k})-f({x})$ for some ${x} \in \Xcal$.
For $k=0,1,2,\ldots$ we have 
\begin{equation}
\begin{aligned} & \Ebb_{k}\big[\Delta_{k+1}\big]-(1-\theta_{k})\Delta_{k}\\
\le{} & \frac{2L^{2}\theta_{k}}{m_{k}\gamma_{k}}+\frac{\gamma_{k}\theta_{k}}{2}\|x-z^{k}\|^{2}-\frac{\gamma_{k}\theta_{k}}{2}\Ebb_{k}[\|x-z^{k+1}\|^{2}]\\
 & -\frac{\gamma_{k}\theta_{k}-\tau\theta_{k}^{2}}{2}\Ebb_{k}[\|z^{k}-z^{k+1}\|^{2}].
\end{aligned}
\label{eq:smodcvx-nes-recur}
\end{equation}
\end{lem}

\begin{proof}
First, recall that $f_{y}(x)=\Ebb_{\xi}[f_{y}(x,\xi)]$. Assumption
\ref{cvx:two-sided} implies that for any $x,y\in\Xcal$, we have 
\[
f(x)=\Ebb_{\xi}[f(x,\xi)]\le\Ebb_{\xi}\big[f_{y}(x,\xi)+\frac{\tau}{2}\|x-y\|^{2}\big]=f_{y}(x)+\frac{\tau}{2}\|x-y\|^{2}.
\]
Therefore, we deduce that
\begin{align}
f(x^{k+1}) & \le f_{y^{k}}(x^{k+1})+\frac{\tau}{2}\|x^{k+1}-y^{k}\|^{2}\nonumber \\
 & =f_{y^{k}}\big((1-\theta_{k})x^{k}+\theta_{k}z^{k+1}\big)+\frac{\tau\theta_{k}^{2}}{2}\|z^{k+1}-z^{k}\|^{2}\nonumber \\
 & \le(1-\theta_{k})f_{y^{k}}(x^{k})+\theta_{k}f_{y^{k}}(z^{k+1})+\frac{\tau\theta_{k}^{2}}{2}\|z^{k+1}-z^{k}\|^{2}\nonumber \\
 & \le(1-\theta_{k})f(x^{k})+\theta_{k}f_{y^{k}}(z^{k+1})+\frac{\tau\theta_{k}^{2}}{2}\|z^{k+1}-z^{k}\|^{2}\nonumber \\
 & =(1-\theta_{k})f(x^{k})+\theta_{k}f_{y^{k}}(z^{k+1},B_{k})+\frac{\tau\theta_{k}^{2}}{2}\|z^{k+1}-z^{k}\|^{2}\nonumber \\
 & \quad+\theta_{k}\big[f_{y^{k}}(z^{k+1})-f_{y^{k}}(z^{k+1},B_{k})\big]\label{eq:smod-nes-01}
\end{align}
where the equality uses the fact $\theta_{k}(z^{k+1}-z^{k})=x^{k+1}-y^{k}$,
the third inequality uses Assumption \ref{cvx:two-sided} again. Moreover, due
to the optimality of $z^{k+1}$ for the subproblem, for any $x\in \Xcal$,
we have
\begin{align}
\begin{aligned}f_{y^{k}}(z^{k+1},B_{k}) & \le f_{y^{k}}(x,B_{k})+\frac{\gamma_{k}}{2}\|x-z^{k}\|^{2}-\frac{\gamma_{k}}{2}\|x-z^{k+1}\|^{2}-\frac{\gamma_{k}}{2}\|z^{k}-z^{k+1}\|^{2}\\
 & \le f(x,B_{k})+\frac{\gamma_{k}}{2}\|x-z^{k}\|^{2}-\frac{\gamma_{k}}{2}\|x-z^{k+1}\|^{2}-\frac{\gamma_{k}}{2}\|z^{k}-z^{k+1}\|^{2}
\end{aligned}
\label{eq:smodcvx-01}
\end{align}
where the second inequality uses Assumption {\ref{cvx:two-sided}}. Following
(\ref{eq:smodcvx-01}) and (\ref{eq:smod-nes-01}), we obtain 
\begin{align}
f(x^{k+1}) & \le(1-\theta_{k})f(x^{k})+\theta_{k}f(x,B_{k})+\theta_{k}\big[f_{y^{k}}(z^{k+1})-f_{y^{k}}(z^{k+1},B_{k})\big]\nonumber \\
 & \quad+\frac{\gamma_{k}\theta_{k}}{2}\|x-z^{k}\|^{2}-\frac{\gamma_{k}\theta_{k}}{2}\|x-z^{k+1}\|^{2}-\frac{\gamma_{k}\theta_{k}-\tau\theta_{k}^{2}}{2}\|z^{k}-z^{k+1}\|^{2}.\label{eq:smod-nes-06}
\end{align}

On both sides of (\ref{eq:smod-nes-06}), we take expectation over
$B_{k}$ conditioned on $B_{1},B_{2},\ldots,B_{k-1}$. Noting that $\Ebb_{k}[f(x,B_{k})]=f(x)$,
we have that
\begin{equation}
\begin{aligned} & \Ebb_{k}\big[f(x^{k+1})-f(x)\big]-(1-\theta_{k})\big[f(x^{k})-f(x)\big]\\
\le{} & \frac{\gamma_{k}\theta_{k}}{2}\|x-z^{k}\|^{2}-\frac{\gamma_{k}\theta_{k}}{2}\Ebb_{k}[\|x-z^{k+1}\|^{2}]-\frac{\gamma_{k}\theta_{k}-\tau\theta_{k}^{2}}{2}\Ebb_{k}[\|z^{k}-z^{k+1}\|^{2}]\\
 & +\theta_{k}\Ebb_{k}\big[f_{y^{k}}(z^{k+1})-f_{y^{k}}(z^{k+1},B_{k})\big].
\end{aligned}
\label{eq:smodcvx-05}
\end{equation}
Moreover, based on the stability of proximal mapping, we have that
\begin{align}
 & \Ebb_{k}\big[f_{y^{k}}(z^{k+1})-f_{y^{k}}(z^{k+1},B_{k})\big] 
= \Ebb_{k}\big\{\Ebb_{\xi}\big[f_{y^{k}}(z^{k+1},\xi)-f_{y^{k}}(z^{k+1},B_{k})\big]\big\}
\le  \frac{2L^{2}}{m_{k}\gamma_{k}}.\label{eq:smod-nes-02}
\end{align}
Combining the above two results together immediately gives us the
desired result (\ref{eq:smodcvx-nes-recur}). 
\end{proof}
\begin{thm}
In Algorithm~\ref{alg:smod-cvx-nes}, let the sequence $\{\Gamma_{k}\}$,
\begin{equation}
\Gamma_{k}=\begin{cases}
(1-\theta_{k})^{-1}\Gamma_{k-1} & \textrm{if }k>0\\
1 & \textrm{if }k=0
\end{cases}\label{eq:smodcvx-Gamma}
\end{equation}
and assume that $\Gamma_{k}$, $\gamma_{k}$, and $\theta_{k}$ satisfy
\begin{align}
\Gamma_{k}\gamma_{k}\theta_{k} & \ge\Gamma_{k+1}\gamma_{k+1}\theta_{k+1},\label{eq:smod-nes-param-1}\\
\gamma_{k} & \ge\tau\theta_{k},\label{eq:smod-nes-param-2}
\end{align}
then we have 
\begin{equation}
\Gamma_{K}\Ebb[\Delta_{K+1}]\le(1-\theta_{0})\Delta_{0}+\frac{\Gamma_{0}\gamma_{0}\theta_{0}^{2}}{2}\|x-z^{0}\|^{2}+\sum_{k=0}^{K}\frac{2L^{2}\Gamma_{k}\theta_{k}}{m_{k}\gamma_{k}}.\label{eq:smod-nes-bound}
\end{equation}
Moreover, if we take ${x}=x^{*}$ be an optimal solution, and
assume that $m_{k}=m$, $\theta_{k}=\frac{2}{k+2}$, $\gamma_{k}=\frac{\gamma}{k+1}$,
$
\gamma=2\tau+\eta$, $\eta=\frac{2L}{\sqrt{3m}\til D}(K+2)^{\frac{3}{2}}
$
 where  $\tilde{D}\ge\|x^{0}-x^{*}\|$, then
we have 
\begin{equation}\label{eq:smod-nes-rate}
\Ebb\big[f(x^{K+1})-f(x^{*})\big]\le\frac{2\tau\til D^{2}}{(K+1)(K+2)}+\frac{4\sqrt{2}L\til D}{\sqrt{3m(K+1)}}.
\end{equation}
\end{thm}

\begin{proof}
First of all, it can be easily checked that conditions (\ref{eq:smod-nes-param-1})
and (\ref{eq:smod-nes-param-2}) are satisfied by the proposed setting
of $\theta_{k}$ and $\gamma_{k}$. Next, multiplying both sides of
(\ref{eq:smodcvx-nes-recur}) by $\Gamma_{k}$, and then dropping
out the negative term $-\frac{\gamma_{k}\theta_{k}-\tau\theta_{k}^{2}}{2}\Gamma_{k}\Ebb_{k}[\|z^{k}-z^{k+1}\|^{2}]$
in the result, we have
\[
\begin{aligned} & \Gamma_{k}\Ebb_{k}\big[\Delta_{k+1}\big]-\Gamma_{k-1}\Delta_{k}\\
\le{} & \frac{2L^{2}\Gamma_{k}\theta_{k}}{m_{k}\gamma_{k}}+\frac{\Gamma_{k}\gamma_{k}\theta_{k}}{2}\|x-z^{k}\|^{2}-\frac{\Gamma_{k}\gamma_{k}\theta_{k}}{2}\Ebb_{k}[\|x-z^{k+1}\|^{2}]
\end{aligned}
\]
Summing up the above result over $k=0,1,2,...,K$ and taking expectation
over all the randomness, we obtain the desired result~(\ref{eq:smod-nes-bound}). 

Moreover, note that $\theta_{0}=1$, $\Gamma_{k}=\frac{(k+2)(k+1)}{2}$,
hence we have 
\begin{equation}
\sum_{k=0}^{K}\frac{2L^{2}\Gamma_{k}\theta_{k}}{m_{k}\gamma_{k}}=\sum_{k=0}^{K}\frac{2L^{2}(k+1)^{2}}{m\gamma}\le\frac{2L^{2}}{m\gamma}\int_{1}^{K+2}s^{2}\textrm{d}s\le\frac{2L^{2}}{3m\gamma}(K+2)^{3}.\label{eq:smod-nes-07}
\end{equation}
Placing ${x}=x^{*}$, then we have
\begin{align*}
\Ebb\big[f(x^{K+1})-f(x^{*})\big] & \le\Gamma_{K}^{-1}\Big\{(1-\theta_{0})\Delta_{0}+\frac{\Gamma_{0}\gamma_{0}\theta_{0}^{2}}{2}\|x-z^{0}\|^{2}+\sum_{k=0}^{K}\frac{2L^{2}\Gamma_{k}\theta_{k}}{m_{k}\gamma_{k}}\Big\}\\
 & \le\Gamma_{K}^{-1}\Big\{\frac{\gamma}{2}\tilde{D}^{2}+\frac{2L^{2}}{3m\gamma}(K+2)^{3}\Big\}\\
 & =\frac{1}{K+1}\Big\{\frac{\gamma\tilde{D}^{2}}{K+2}+\frac{4L^{2}(K+2)^{2}}{3m\gamma}\Big\}\\
 & \le\frac{2\tau\til D^{2}}{(K+1)(K+2)}+\frac{1}{K+1}\Big\{\frac{\eta\tilde{D}^{2}}{K+2}+\frac{4L^{2}(K+2)^{2}}{3m\eta}\Big\}\\
 & =\frac{2\tau\til D^{2}}{(K+1)(K+2)}+\frac{4L\til D}{K+1}\sqrt{\frac{K+2}{3m}}\\
 & \leq \frac{2\tau\til D^{2}}{(K+1)(K+2)}+\frac{4\sqrt{2}L\til D}{\sqrt{3m(K+1)}}.
\end{align*}
where the second inequality uses (\ref{eq:smod-nes-07}), and $\tilde{D}\ge\|x^{0}-x^{*}\|$, the third inequality uses the fact $\gamma=2\tau+\eta$ and $\frac{1}{\gamma}\le\frac{1}{\eta}$,
and the last inequality uses $K+2\le2(K+1)$ for $K\ge1$. This completes the proof.
\end{proof}

\section{Solving the subproblems\label{app:subproblem}}
In this section, we describe how to solve the subproblems arising from {\sgd}, {\spl} and {\spp}. For the sake of simplicity, we suppress all the iteration indices and express the {\smod} subproblems as follows
\begin{mini}
{x\in \Rbb^d}{\frac{1}{m} \sum_{i=1}^m \varphi_{z}(x,\xi_i)+\frac{\gamma}{2}\|x-y\|^2.}{}{}	
\end{mini}

\subsection{Phase retrieval}
We first state the expressions for the sequential updates (i.e. $m=1$). More technical derivations can be referred from \citeapp{davis2019stochasticweaklydup}. 
Given the current iterate $x$, we denote $x^+$ to be the output of {\smod} update. 

Let $\xi=(a,b)$ for $a\in\Rbb^d$ and $b\in\Rbb$, we have
\begin{flalign*}
x_{\text{sgd}}^{+} & =\argmin_x \ \left\{ \langle v, x - z \rangle + \frac{\gamma}{2} \| x - y \|^2 \right\},\\	
 x_{\text{spl}}^{+} & =\argmin_x\ \left\{ \left| \langle a, z \rangle^2 + 2
   \langle a, z \rangle \langle a,  x - z \rangle - b \right| + \frac{\gamma}{2} \| x -
   y \|^2 \right\},\\
   x_{\text{spp}}^{+}& =\argmin_x \  \left\{ | \langle a, x \rangle^2 - b|
   + \frac{\gamma}{2} \| x - y \|^2 \right\}.
\end{flalign*}
The above three subproblems admit closed-form solutions
\begin{flalign*}
x_{\text{sgd}}^{+} & = y - \gamma^{-1}v,\\
x_{\text{spl}}^{+} & = y + \tmop{Proj}_{[- 1, 1]} \left( -
   \tfrac{\delta}{\| \zeta \|^2} \right) \zeta,\\
x_{\text{spp}}^{+} & \in \left\{ y - \tfrac{2 \langle a, y \rangle a}{2 \| a \|^2 \pm \gamma}, y - 
   \tfrac{\langle a, y \rangle \pm \sqrt{b}}{\| a \|^2} a \right\},
\end{flalign*}
where 
\begin{align*}
	v &\in \partial_x (| \langle a, z \rangle^2 - b |)  \\
	  & =  2 \langle a, z \rangle a
\cdummy \left\{ \begin{array}{ll}
  \sign (\langle a, z \rangle^2 - b), & \text{if } \langle a, z
  \rangle^2 - b \neq 0\\
  {}[- 1, 1], & \text{o.w.}
\end{array} \right. \\	
	\delta & ={\gamma}^{-1}(\langle a, z \rangle^2 + 2 \langle a, z \rangle
\langle a,  x - z \rangle - b), \\
\zeta & = 2 \gamma^{-1}\langle a, z \rangle a,
\end{align*}
and $\tmop{Proj}_{[- 1, 1]} (\cdummy)$ denotes the orthogonal projection operator onto $[-1, 1]$.

When $m>1$, we let $y=z$ and
\begin{flalign*}
x_{\text{sgd}}^{+} = &\argmin_x \ \left\{ \frac{1}{m} \sum_{i = 1}^m
   \langle v_i, x - z \rangle +
   \frac{\gamma}{2} \| x - z \|^2 \right\}\\
x_{\text{spl}}^{+} = &\argmin_x \ \left\{ \frac{1}{m} \sum_{i = 1}^m
   \left| \langle a_i,  z \rangle^2 - b_i + 2 \langle a_i,  z \rangle \langle a_i, x - z \rangle \right| +
   \frac{\gamma}{2} \| x - z \|^2 \right\}\\
x_{\text{spp}}^{+} = & \argmin_x \ \left\{ \frac{1}{m} \sum_{i = 1}^m
   | \langle a_i,  x \rangle^2 - b_i | + \frac{\gamma}{2} \| x - z \|^2 \right\},
\end{flalign*}
where $v_i \in \partial_x (| \langle a_i, z \rangle^2 - b_i |)$. Minibatch subproblems can be reformulated as standard convex programs.
\begin{flalign}
 x_{\text{sgd}}^{+} & = z - \frac{1}{m \gamma}\sum_{i = 1}^m v_i \label{eq:batch-update-sgd}\\ \nonumber \\
  \big( x_{\text{spl}}^{+}, * \big)  & = \argmin_{(x, t)}\ \left\{
  \frac{1}{m} \sum_{i = 1}^m t_i + \frac{\gamma}{2} \| x - z \|^2 \right\}\nonumber \\
  \text{subject to} \ \ \ &  \langle a_i,  z \rangle^2 - b_i + 2 \langle a_i,  z \rangle \langle a_i, x - z \rangle \geq - t_i \\
  & \langle a_i,  z \rangle^2 - b_i + 2 \langle a_i,  z \rangle \langle a_i, x - z \rangle \leq t_i
  \  i=1,2,\ldots,m. \label{eq:batch-update-spl} \\ \nonumber\\
  \big(x_{\text{spp}}^{+}, * \big)  & = \argmin_{(x, t)}\ \left\{ 
  \frac{1}{m} \sum_{i = 1}^m t_i \right\}\nonumber\\
  \text{subject to}\ \ \  & x\trans \left( \frac{\gamma}{2} I - a_i a_i\trans
  \right) x - \gamma \langle z,  x \rangle + \frac{\gamma}{2} \| z \|^2 + b_i \leq
  t_i\nonumber \\
  & x\trans \left( \frac{\gamma}{2} I + a_i a_i\trans \right) x - \gamma
  \langle z,  x \rangle + \frac{\gamma}{2} \| z \|^2 - b_i \leq t_i, \ i=1,2,\ldots,m \label{eq:batch-update-spp}
\end{flalign}

We make a few comments. First, the update ({\sgd}) (\ref{eq:batch-update-sgd}) admits a simple closed-form solution by directly  averaging the subgradients over the minibatch samples. 
Second, the {\spl} subproblem (\ref{eq:batch-update-spl}) can be further transformed into an $O(m)$-dimensional dual quadratic program, which can be efficiently solved in parallel. (See \citeapp{asi2020minibatchdup}).
Third, the {\spp} subproblem (\ref{eq:batch-update-spp}) is solvable by interior point methods for quadratically constrained quadratic programming (QCQP).

However, despite the fast theoretical convergence, interior point methods are potentially unscalable to problems with a large number of constraints. In our experiments, the commercial solver Gurobi  fails to get an  accurate solution for solving (\ref{eq:batch-update-spp}) when $m>5$.  Therefore, we alternatively utilize the strong convexity of (\ref{eq:batch-update-spp}) and adopt deterministic prox-linear algorithm to obtain an accurate solution (up to \texttt{1e-08} accuracy) by solving several QPs as in (\ref{eq:batch-update-spl}). The theoretical linear convergence of this method is verified in Section~\ref{sec:subproblinconv}.
Finally, similar observations can be made for the experiments of blind deconvolution.

\subsection{Blind deconvolution}
Blind deconvolution aims to separate two unknown signals from their convolution, resulting in  the following non-smooth biconvex problem
\begin{mini}
{x, y \in\Rbb^d}{ \frac{1}{n} \sum_{i = 1}^n \big\vert\langle u_i, x\rangle
  \langle v_i, y\rangle - b_i \big\vert. }{}{}\label{pb:phase}
\end{mini}
For convenience we use $(x;y)$ to denote the vertical concatenation of two column vectors. Given current iterate $w = (w_x; w_y)$, the subproblems are given by
\begin{flalign*}
 w^+_\text{sgd} & = \argmin_{(x; y)}\  \left\{ \langle s, (x - z_x; y - z_y) \rangle + \frac{\gamma}{2} \| x - w_x \|^2 +
   \frac{\gamma}{2} \| y - w_y \|^2 \right\}\\
 w^+_\text{spl} & =  \argmin_{(\Delta_x; \Delta_y)}\ \Big\{ | \langle u, z_x \rangle \langle v,  z_y\rangle + \langle v, z_y \rangle \langle u, \Delta_x \rangle + \langle u, z_x \rangle \langle v, \Delta_y \rangle \\
  &\quad \quad \quad  + \langle v, z_y\rangle \langle u, w_x - z_x \rangle + \langle u, z_x \rangle \langle v, w_y - z_y \rangle - b | + \frac{\gamma}{2} [\| \Delta_x \|^2 + \| \Delta_y \|^2] \Big\} + w\\
 w^+_\text{spp} & = \argmin_{(x; y)}\  \left\{ | \langle u, x \rangle \langle v, y \rangle - b | + \frac{\gamma}{2} \| x - w_x \|^2 + \frac{\gamma}{2} \| y - w_y \|^2 \right\}
\end{flalign*}
and we have
\begin{flalign*}
   w^+_\text{sgd} & = w - \gamma^{-1}s\\
   w^+_\text{spl} & = w + \tmop{Proj}_{[- 1,
   1]} \left( - \tfrac{\delta}{\| \zeta \|^2} \right) \zeta
\end{flalign*}
where 
\begin{align*}
s &\in \partial_{(x; y)} (| \langle u, z_x \rangle \langle v, z_y
   \rangle - b |) \\
  & = (\langle v, z_y \rangle u; \langle u, z_x \rangle
   v) \cdummy \left\{ \begin{array}{ll}
     \tmop{sign} (\langle u, z_x \rangle \langle v, z_y \rangle - b), &
     \text{if } \langle u, z_x \rangle \langle v, z_y \rangle - b \neq
     0\\
     {}[- 1, 1], & \text{o.w.}
   \end{array} \right.\\
\delta & = {\gamma}^{-1} \left[\langle u, z_x \rangle \langle v, z_y \rangle + \langle v, z_y \rangle
\langle u, w_x - z_x \rangle + \langle u, z_x \rangle \langle v, w_y - z_y \rangle - b \right]\\
\zeta & = {\gamma}^{-1} \left(
  \langle v, z_y \rangle u; \langle u, z_x \rangle v \right).
\end{align*}

As for {\spp}, we consider the following two cases.\\

\textbf{Case 1.} If $\langle u, w_x \rangle \langle v, w_y\rangle - b \neq 0$, then
\[	w_x^{+} = w_x - \left\{ \tfrac{\pm \gamma \langle v, w_y \rangle - \| v \|^2
   \langle u, w_x \rangle}{\gamma^2 - \| u \|^2 \| v \|^2} \right\} u, \quad
   w_y^{+} = w_y - \left\{ \tfrac{\pm \gamma \langle u, w_x \rangle - \| u \|^2
   \langle v, w_y \rangle}{\gamma^2 -  \| u \|^2 \| v \|^2} \right\} v.
\]

\textbf{Case 2.} If $\langle u, w_x \rangle \langle v, w_y\rangle - b = 0$, then
\[
	w_x^{+} = w_x - \zeta \left( \tfrac{b}{\eta} \right) u,\quad
   w_y^{+}  = w_y - \zeta \eta v,
   \]
where $\zeta = \frac{\eta \langle u, w_x \rangle - \eta^2}{b \| u \|^2}$ and $\eta$ is determined by
\[ \eta^4 \| v \|^2 - \eta^3 \| v \|^2 \langle u, w_x \rangle + b \eta \| u
   \|^2 \langle v, w_y \rangle - b^2 \| u \|^2 = 0.\]

Moreover, for  the minibatch variants, we set $w=z$ and get the following subproblems
\begin{flalign*}
  w_{\tmop{sgd}}^+ & = \argmin_{(x; y)}\  \left\{ \frac{1}{m} \sum_{i = 1}^m
  \langle s_i,  (x - z_x; y - z_y) \rangle + \frac{\gamma}{2} \| x - z_x \|^2 +
  \frac{\gamma}{2} \| y - z_y \|^2 \right\}\nonumber\\
  w_{\tmop{spl}}^+ & = \argmin_{(\Delta_x; \Delta_y)}\ \Bigg\{
  \frac{1}{m} \sum_{i = 1}^m \left| \langle u_i, z_x \rangle \langle v_i, z_y
  \rangle + \langle v_i, z_y \rangle \langle u_i, \Delta_x \rangle + \langle
  w_i, z_x \rangle \langle v_i, \Delta_y \rangle - b_i \right| \\
  & \hspace{50pt} + \frac{\gamma}{2} \| \Delta_x \|^2 + \frac{\gamma}{2} \| \Delta_y \|^2
  \Bigg\} + z\nonumber\\
  w^+_{\tmop{spp}} & = \argmin_{(x; y)}\ \left\{ \frac{1}{m} \sum_{i =
  1}^m | \langle u_i, x \rangle \langle v_i, y \rangle - b_i | +
  \frac{\gamma}{2} \| x - z_x \|^2 + \frac{\gamma}{2} \| y - z_y \|^2 \right\},
  \nonumber
\end{flalign*}

where $s_i \in \partial_{(x; y)} (| \langle u_i, z_x \rangle \langle v_i, z_y
   \rangle - b_i |)$. Then we solve the subproblems by \\
\begin{flalign*}
 w^+_\text{sgd} & = z - \frac{1}{m \gamma}\sum_{i=1}^m  s_i,\\ \\
  \big(x_{\tmop{spl}}^+; y^+_{\tmop{spl}}, *\big) & = \argmin_{(x, y,
  t)}\ \left\{ \frac{1}{m} \sum_{i = 1}^m t_i + \frac{\gamma}{2} \| x - z_x \|^2 +
  \frac{\gamma}{2} \| y - z_y \|^2  \right\} \nonumber\\
  \text{subject to } \ \ \ &\langle u_i, z_x \rangle \langle v_i, z_y
  \rangle + \langle v_i, z_y \rangle \langle u_i, x - z_x \rangle + \langle
  u_i, z_x \rangle \langle v_i, y - z_y \rangle - b_i \leq t_i \\
  & \langle u_i, z_x \rangle \langle v_i, z_y
  \rangle + \langle v_i, z_y \rangle \langle u_i, x - z_x \rangle + \langle
  u_i, z_x \rangle \langle v_i, y - z_y \rangle - b_i \geq -t_i,
  \\& i=1,2,\ldots, m
  \nonumber\\
  &  \nonumber\\
  \big(x_{\tmop{spp}}^+; y^+_{\tmop{spp}}, \ast\big) & = \argmin_{(x, y,
  t)}\ \left\{ \frac{1}{m} \sum_{i = 1}^m t_i \right\} \nonumber\\
  \text{subject to } \ \ \ & \frac{\gamma}{2} [\| x - z_x \|^2 + \| y - z_y \|^2] +
  \langle u_i; x \rangle \langle v_i, y \rangle - b_i \leq t_i \nonumber\\
  & \frac{\gamma}{2} [\| x - z_x \|^2 + \| y - z_y \|^2] - \langle u_i, x
  \rangle \langle v_i,y \rangle + b_i \leq t_i,\quad i=1,2,\ldots, m, \nonumber
 \end{flalign*}
where the last two problems are solved by either QP (QCQP) optimizers or prox-linear iterations as in phase retrieval.\\

\subsection{Solving the {\spp} subproblem by Prox-linear algorithm} \label{sec:subproblinconv}
In this section we describe how to solve the subproblem of {\spp} when the objective admits a composition form $h(c(\cdot))$.  Specifically, we show that when applied to the {\spp} subproblem, the deterministic prox-linear algorithm obtains a linear convergence rate. Without loss of
generality, consider the SPP subproblem
\begin{mini}
{x\in\Xcal}{\varphi(x)+ \frac{\gamma}{2} \|
  x - \bar{x} \|^2}{}{} \label{ppsubprob}
\end{mini}
where $\varphi (x) =\frac{1}{m} \sum_{i = 1}^m  h (c (x, \xi_i))$. For clarity we denote $z^t$ to be the iterate of the subproblems and define $\varphi_{z^t} (z) \assign \frac{1}{m} \sum_{i = 1}^m h (c
(z^t, \xi_i) + \langle (\nabla c (z^t, \xi_i), z - z^t \rangle)$. In each prox-linear
iteration, we take $\eta \geq \tau$ and compute
\begin{eqnarray*}
  z^{t + 1} = \arg \min_z  \left\{ \varphi_{z^t}(z) +
  \frac{\gamma}{2} \| z - \bar{x} \|^2 + \frac{\eta}{2} \| z - z^t \|^2 \right\}.
\end{eqnarray*}
First, according to Assumption~\ref{ass:two-sides-quad}, we have
\begin{align*}
  \varphi_{z^t} (z) - \varphi (z) & \leq \frac{\tau}{2} \| z - z^t \|^2,\\
  \varphi (z^{t + 1}) - \varphi_{z^t} (z^{t + 1}) & \leq \frac{\tau}{2} \| z^{t + 1} -
  z^t \|^2,
\end{align*}
and by the strong convexity of the objective in~\eqref{ppsubprob}, we have
\begin{align*}
  & \quad \varphi_{z^t} (z^{t + 1}) + \frac{\gamma}{2} \| z^{t + 1} - \bar{x} \|^2 +
  \frac{\eta}{2} \| z^{t + 1} - z^t \|^2\\
  & \leq \varphi_{z^t} (z) + \frac{\gamma}{2} \| z - \bar{x} \|^2 + \frac{\eta}{2} \|
  z - z^t \|^2 - \frac{\gamma + \eta - \lambda}{2} \| z^{t + 1} - z \|^2.
\end{align*}
Combining the above three inequalities leads to
\begin{align*}
  & \quad \varphi_{z^t} (z^{t + 1}) + \frac{\gamma}{2} \| z^{t + 1} - \bar{x} \|^2 +
  \frac{\eta}{2} \| z^{t + 1} - z^t \|^2 + \varphi_{z^t} (z) - \varphi (z) + \varphi (z^{t + 1})
  - \varphi_{z^t} (z^{t + 1})\\
  & \leq \varphi_{z^t} (z) + \frac{\gamma}{2} \| z - \bar{x} \|^2 + \frac{\eta}{2} \|
  z - z^t \|^2 - \frac{\gamma + \eta - \lambda}{2} \| z^{t + 1} - z \|^2 +
  \frac{\tau}{2} \| z - z^t \|^2 + \frac{\tau}{2} \| z^{t + 1} - z^t \|^2 .
\end{align*}
Rearranging the terms accordingly, we have
\begin{align}
  & \quad \left[ \varphi (z^{t + 1}) + \frac{\gamma}{2} \| z^{t + 1} - \bar{x} \|^2
  \right] - \left[ \varphi (z) + \frac{\gamma}{2} \| z - \bar{x} \|^2 \right] \nonumber \\ 
  & \leq \frac{\eta + \tau}{2} \| z - z^t \|^2 - \frac{\gamma + \eta -
  \lambda}{2} \| z^{t + 1} - z \|^2 + \frac{\tau - \eta}{2} \| z^{t + 1} - z^t
  \|^2 \nonumber \\
  & \leq \frac{\eta + \tau}{2} \| z - z^t \|^2 - \frac{\gamma + \eta -
  \lambda}{2} \| z^{t + 1} - z \|^2, \label{linear-conv}
\end{align}
where the last inequality is by $\eta \geq \tau$. Define $\alpha = \frac{\eta + \tau}{\gamma + \eta - \lambda}$ and
divide both sides of the inequality by $\left( \frac{\eta + \tau}{2} \right)
\alpha^t$, we obtain
\begin{align*}
  & \quad \frac{2}{\alpha^t (\eta + \tau)} \left\{ \left[ \varphi (z^{t + 1}) +
  \frac{\gamma}{2} \| z^{t + 1} - \bar{x} \|^2 \right] - \left[ \varphi (z) +
  \frac{\gamma}{2} \| z - \bar{x} \|^2 \right] \right\}\\
  & \leq \frac{1}{\alpha^t} \| z^t - z \|^2 - \frac{1}{\alpha^t} \cdummy
  \frac{\gamma + \eta - \lambda}{\eta + \tau} \| z^{t + 1} - z \|^2\\
  & = \frac{1}{\alpha^t} \| z^t - z \|^2 - \frac{1}{\alpha^{t + 1}} \| z^{t
  + 1} - z \|^2 .
\end{align*}
We denote $\Delta_t \assign \left[ \varphi (z^{t + 1}) + \frac{\gamma}{2} \| z^{t + 1}
- \bar{x} \|^2 \right] - \left[ \varphi (z^{\ast}) + \frac{\gamma}{2} \| z^{\ast} - \bar{x}
\|^2 \right]$.  By placing $z=z^t$ in~\eqref{linear-conv}, we can easily verify that $\{\Delta_t \}$ is monotonically decreasing.
Taking $z = z^{\ast}$ and summing over $t = 0, \ldots, T$, we get
\begin{align*}
  \sum_{t = 0}^T \frac{2\Delta_T}{\alpha^t (\eta + \tau)}  & \leq \sum_{t = 0}^T \frac{2 \Delta_t}{\alpha^t (\eta + \tau)} \\
  & \leq \sum_{t = 0}^T \frac{1}{\alpha^t} \| z^t - z^{\ast} \|^2 -
  \frac{1}{\alpha^{t + 1}} \| z^{t + 1} - z^{\ast} \|^2\\
  & = \| z^0 - z^{\ast} \|^2 - \frac{1}{\alpha^{T + 1}} \| z^{T + 1} -
  z^{\ast} \|^2\\
  & \leq \| z^0 - z^{\ast} \|^2
\end{align*}
and we have
\[ \Delta_T \leq \frac{(\eta + \tau) \| z^0 - z^{\ast}
   \|^2}{2 \big( \sum_{t = 0}^T 1 / \alpha^t \big)} \leq \frac{(\eta +
   \tau) \| z^0 - z^{\ast} \|^2}{2} \left( \frac{\eta + \tau}{\gamma + \eta -
   \lambda} \right)^T. \]
Therefore, we have shown the linear convergence of prox-linear algorithm for solving {\spp} subproblems.

\section{Additional experiments}\label{app:addexp}
This section presents the experiments that were not displayed in the main article due to space limit. First, we complement the experiments in Section~\ref{sec:Experi} by visualizing the effectiveness of image recovery on \texttt{zipcode} datasets. Second, we provide new experiments on the problem of  blind deconvolution.

\subsection{Phase retrieval}
We conduct the experiments on  the test images of digit 6 and illustrate the results of {\spl} and {\sgd}  in Figure~\ref{fig:recover4-spl} and Figure~\ref{fig:recover4-sgd}, respectively.  We fix $\alpha_0 = 100$ and run each algorithm over 200 epochs (number of passes over the data). Then we report the results over the earliest 600 iterations and plot the recovered digits for different batch sizes $m \in \{1, 4, 8, 16, 32, 48, 64\}$.  It can be seen that with larger batch size, both methods exhibit improved performance and generate images with better quality, which suggests the practical advantage of using large batch size.
  Moreover,  {\spl} outperforms {\sgd} by giving a much better recovered image quality. This observation confirms the earlier study about the superior performance of prox-linear methods \citeapp{duchi2019solvingdup}. 

\begin{figure*}[!h]
	\centering
	\includegraphics[scale=0.15]{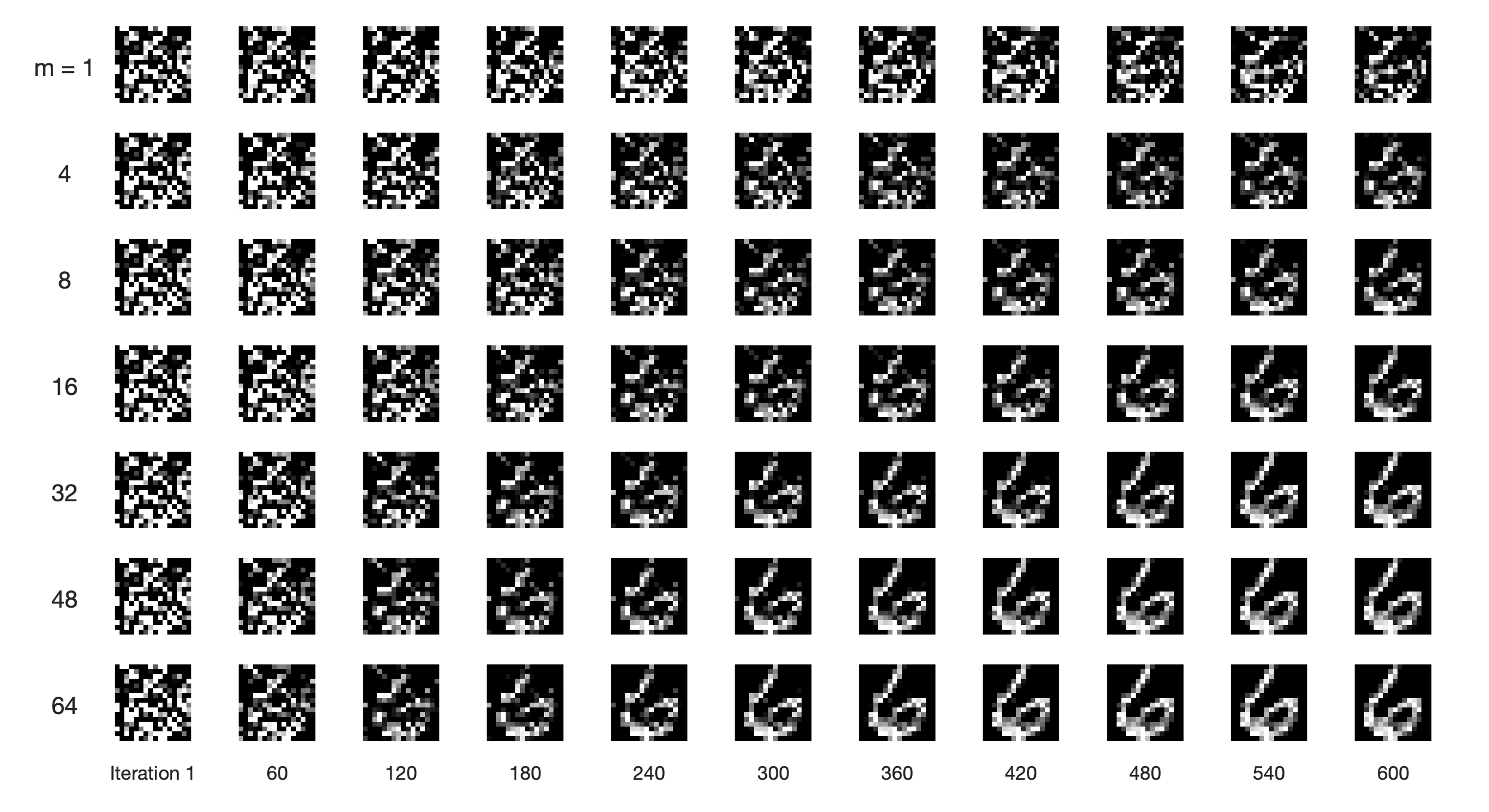}
	\caption{Reconstruction of image (digit 6) for {\spl}. Rows correspond to recovery results of different minibatch sizes $m \in \{1, 4, 8, 16, 32, 48, 64\}$. Columns correspond to recovery results after different number of iterations $T\in\{1, 60,120, 180, 240, 300, 360, 420, 480, 540, 600\}$. \label{fig:recover4-spl}}
\end{figure*}

\begin{figure*}[!h]
	\centering
	\includegraphics[scale=0.17]{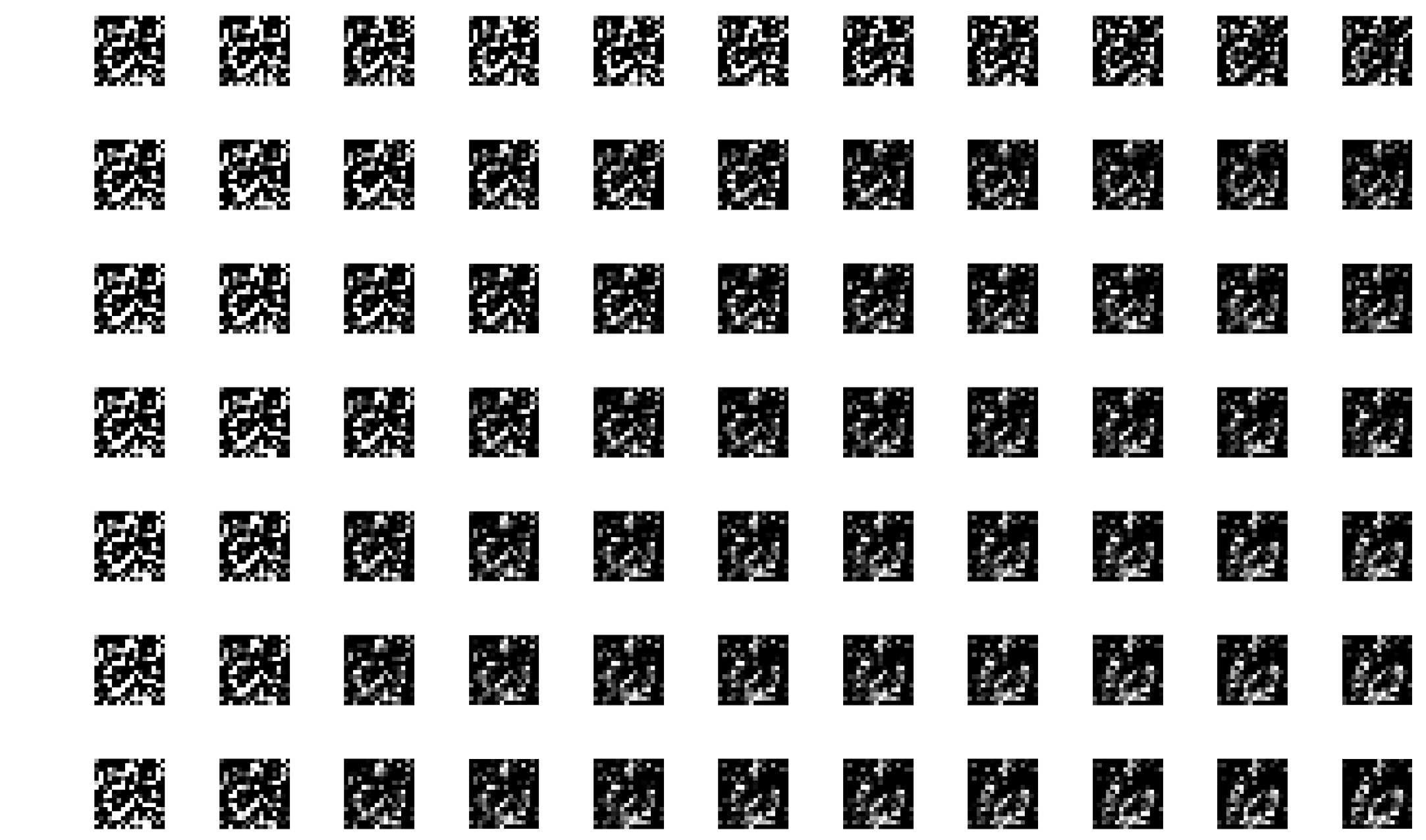}
	\caption{Reconstruction of real image (digit 6) for stochastic (sub)gradient descent. \label{fig:recover4-sgd}}
\end{figure*}

\subsection{Blind deconvolution}
\textbf{Data preparation.} We conduct the experiments over the synthetic dataset.

\textbf{1) Synthetic data.}
We choose $n$, $d$ and the signal $x^*$ in the same way as in phase retrieval. Namely we generate $U = Q_1 D_1, V = Q_2 D_2$
where $q_{i j} \sim \mathcal{N} (0, 1)$ and $D_1, D_2$ are diagonal matrices whose diagonal entries evenly distribute between 1 and $1/\kappa$; Measurements $\{b_i\}$ are generated by $b_i = \langle u_i, x^{\ast} \rangle \langle v_i, x^{\ast} \rangle + \delta_i \zeta_i$ with $\zeta_i \sim \mathcal{N} (0,
25)$ and $ \delta \sim \tmop{Bernoulli} (p_{\tmop{fail}})$.

The detailed experiment setup is given as follows:

\textbf{1) Dataset generation.}  We test $\kappa \in \{1, 10\}$ and $ p_{\tmop{fail}} \in \{0.2, 0.3\}$;

\textbf{2) Initial point.} For all the algorithms, we set the initial point $x^1(=x^0)$ and $y^1(=y^0)\sim\Ncal(0, I_d)$;

\textbf{3) Stepsize.} We set the parameter $\gamma=\alpha_0^{-1}\sqrt{K/m}$ where $m$ is the batch size;
we test 10 evenly spaced $\alpha_0$ values in range $ [10^{-1}, 10^{2}]$ for {\sgd}, {\spl} and in range $ [10^{-2}, 10^{1}]$ for {\sgd}, {\spl} and {\spp};

\textbf{4) Others.} The rest of the experiment setup are the same as in synthetic phase retrieval, which can be referred from Section \ref{sec:Experi}.

It should be noted that if $\alpha_0 \geq 10$, the resulting {\spp} subproblem is still nonconvex. Therefore, we present the results with $\alpha_0$ in two ranges for {\spp} and the other two {\smod} algorithms.

In Figure~\ref{fig:bd-speedup-batchsize}, we plot the the algorithm speedup over the size of minibatches for two different settings $p_\text{fail} \in \{0.2, 0.3\}$.
We find that both {\spl} and {\sgd} enjoy linear speedup over the size of minibatches.
Figure~\ref{fig:bd-speedup-stepsize} shows the algorithm speedup over different values of $\alpha_0$.  In comparison with {\sgd},  {\spl} has significant acceleration over a much wider range of stepsize values. Figure~\ref{fig:bd-itercount-stepsize} shows the total iteration number over different values of $\alpha_0$. The result suggests that momentum can further improve the performance of all the stochastic algorithms, particularly, when the algorithms are initiated with small stepsizes.

\begin{figure*}[!ht]
\centering
\includegraphics[scale=0.20]{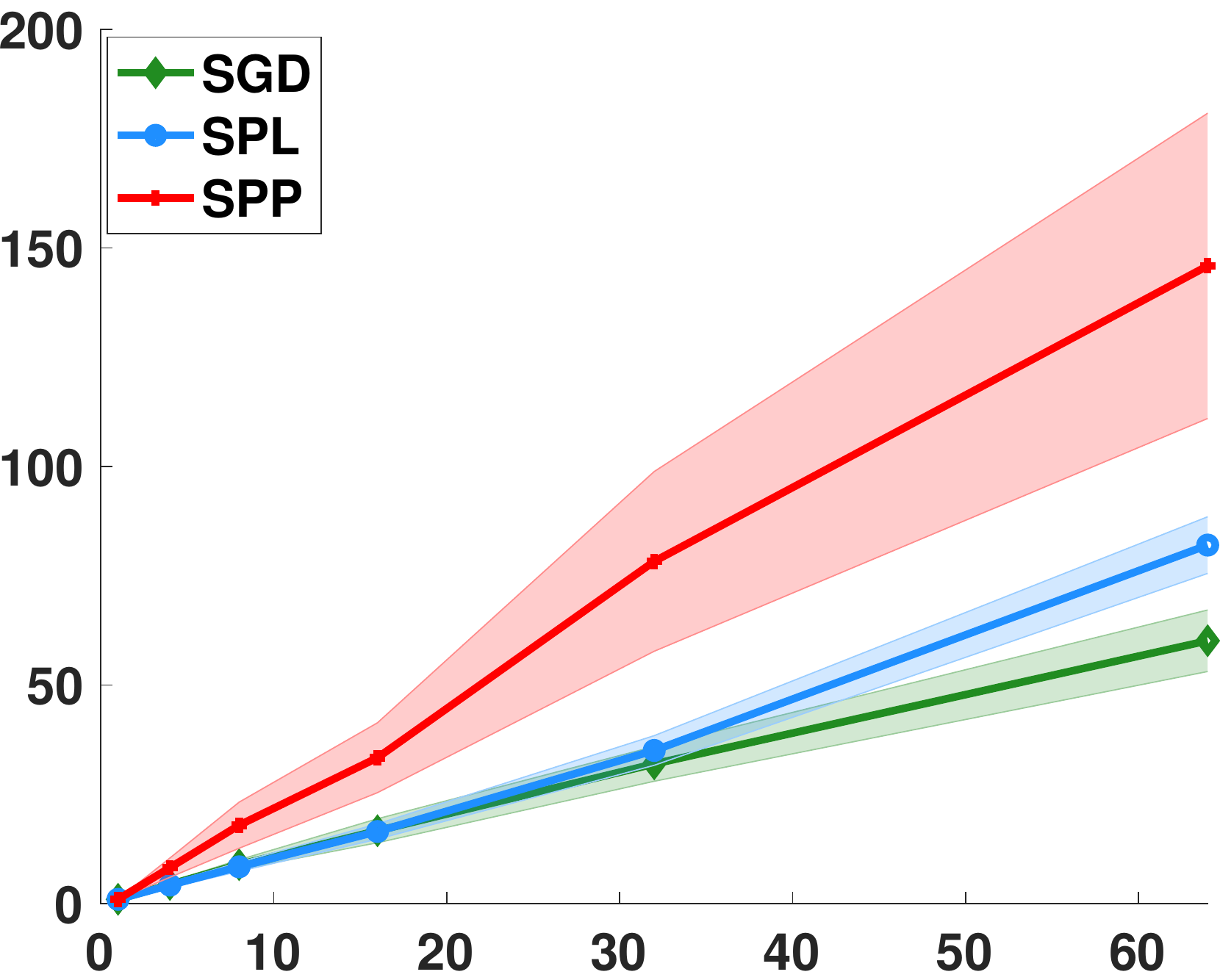}
\includegraphics[scale=0.20]{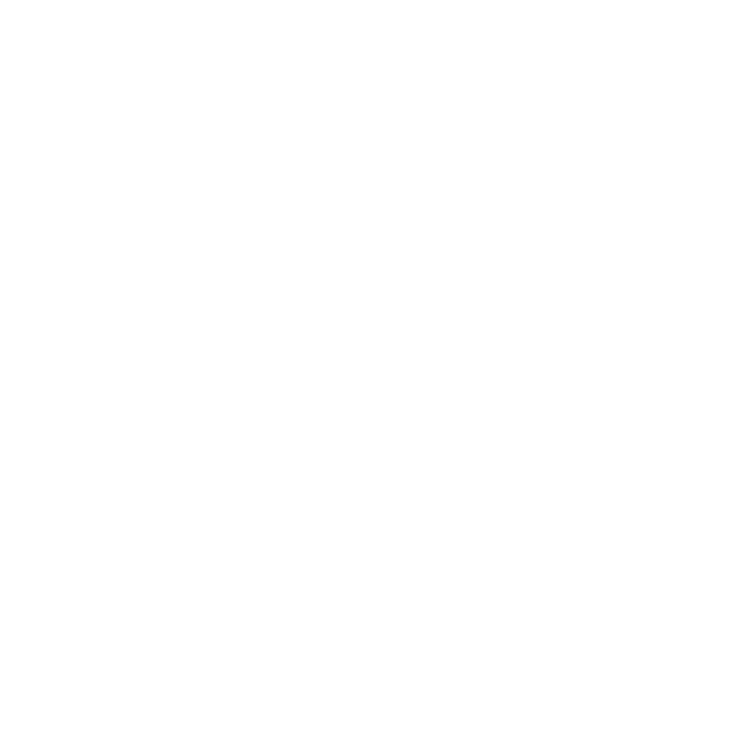}
\includegraphics[scale=0.20]{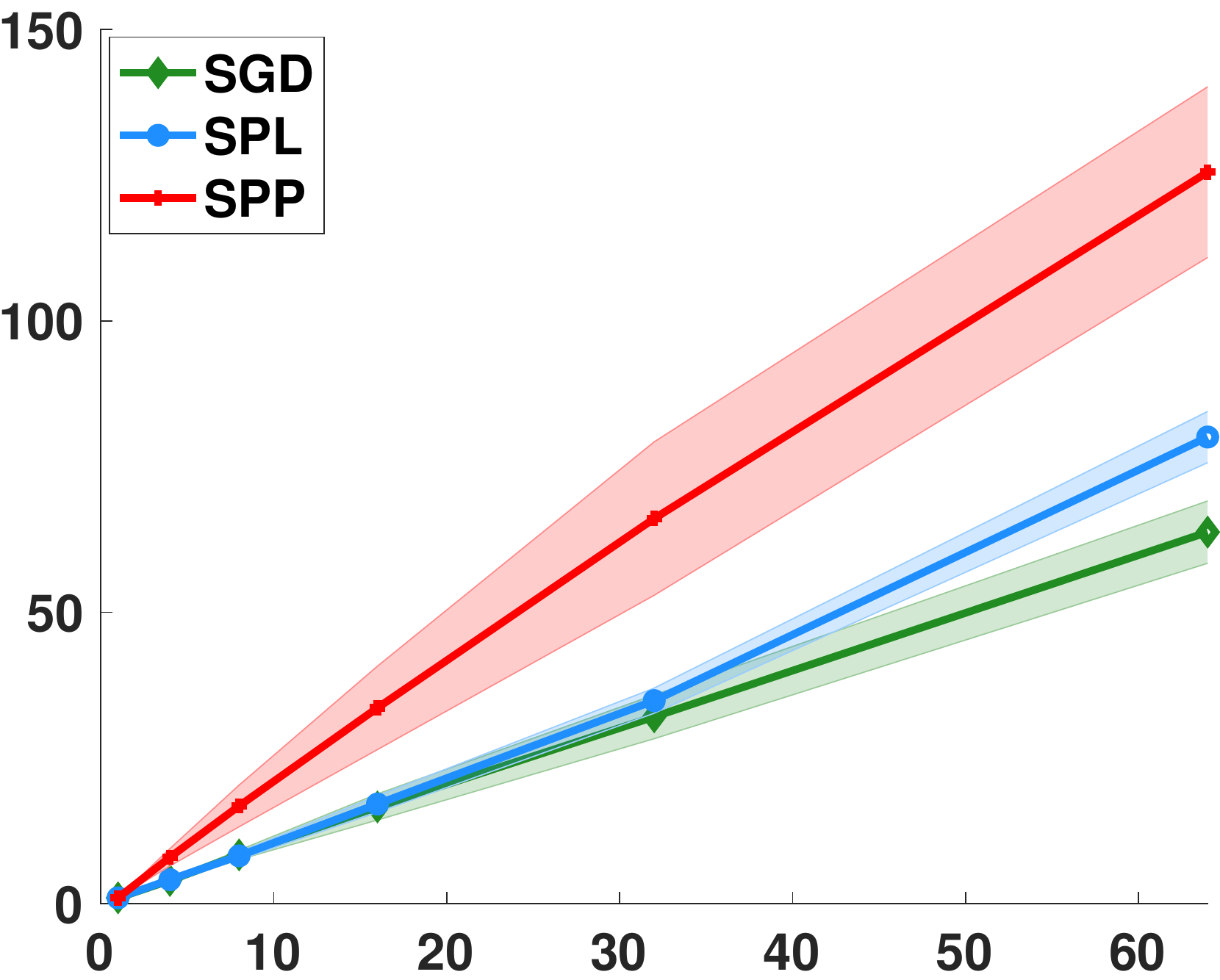}
	\caption{Speedup vs. batch size $m$.
	$\kappa=10$.
	From left to right: $(\alpha_0, p_\text{fail}) = ([10^{-2}, 10], 0.2), ([10^{-1}, 10^2], 0.2), ([10^{-2}, 10], 0.3), ([10^{-1}, 10^2], 0.3)$.\label{fig:bd-speedup-batchsize}}
\end{figure*}

\begin{figure*}[!htbp]
	\centering
	\includegraphics[scale=0.20]{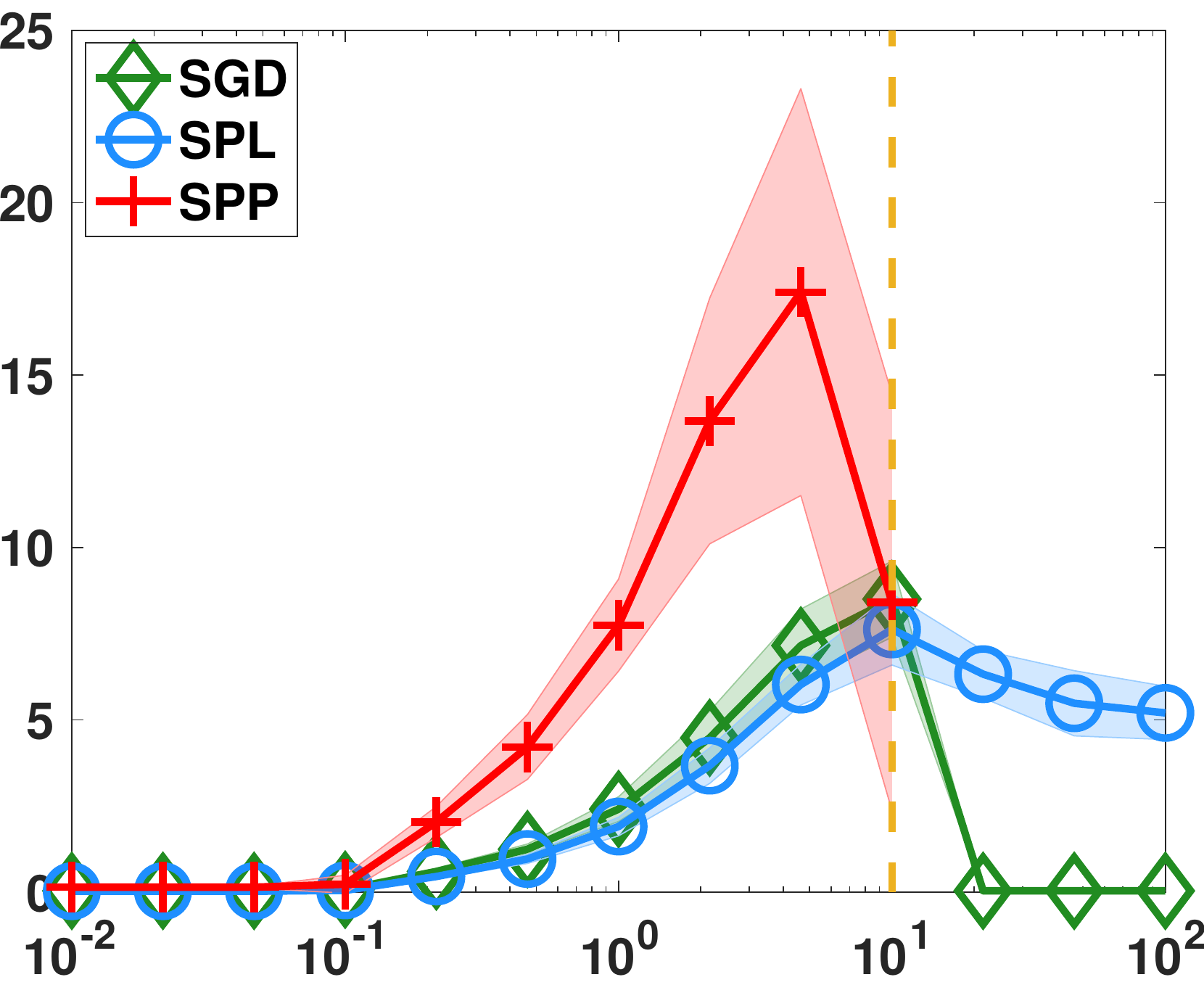}
	\includegraphics[scale=0.20]{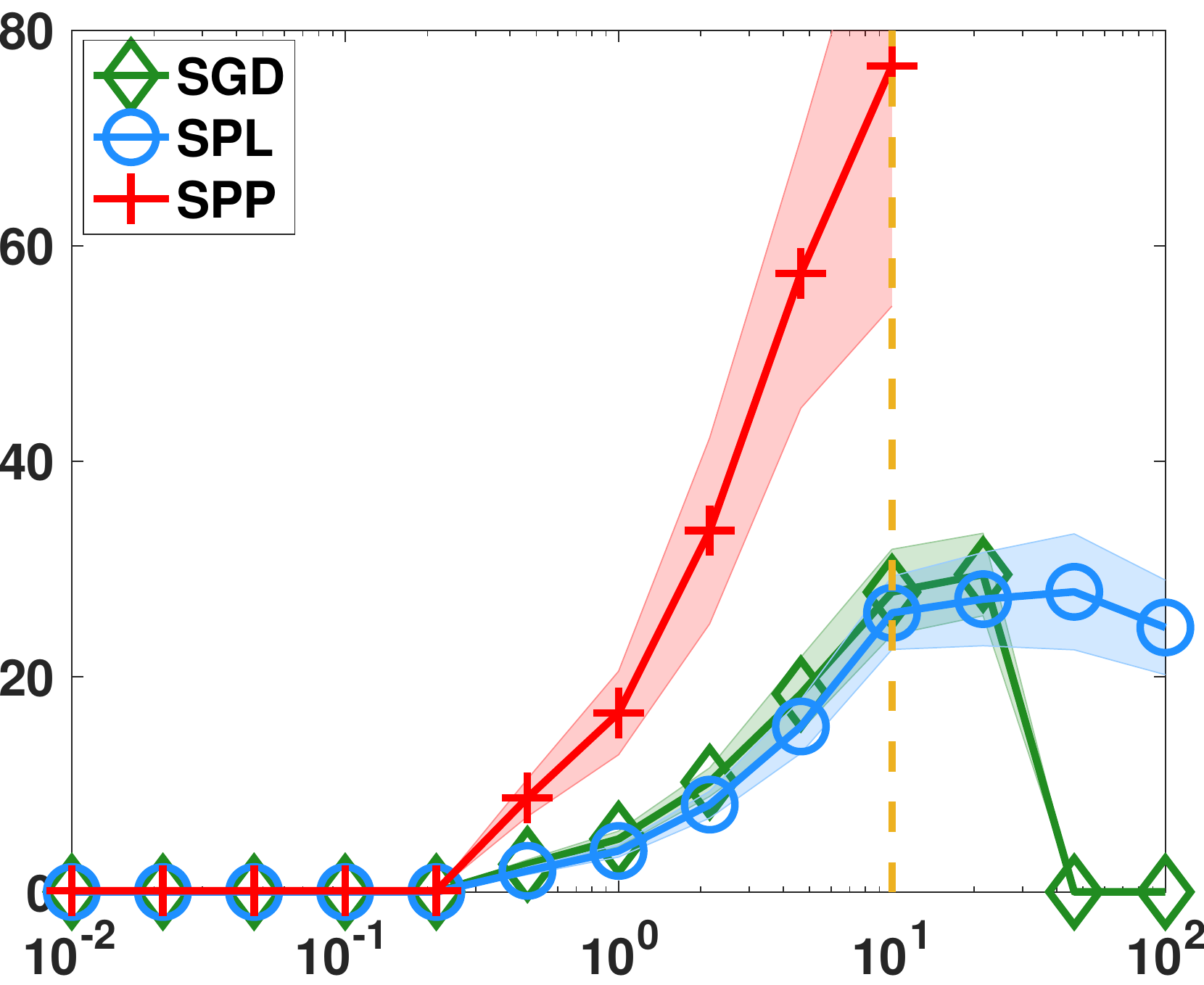}	
	\includegraphics[scale=0.20]{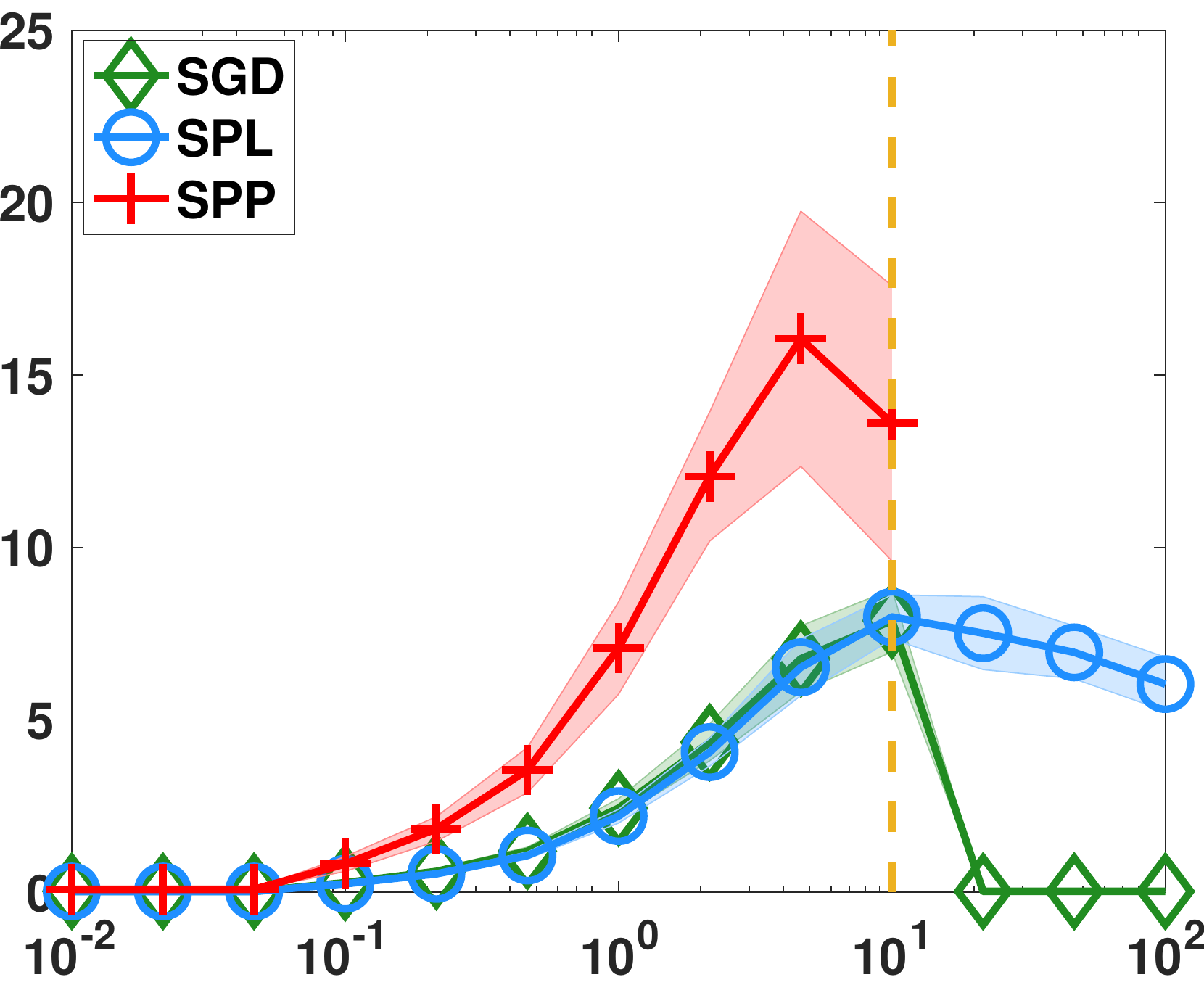}
	\includegraphics[scale=0.20]{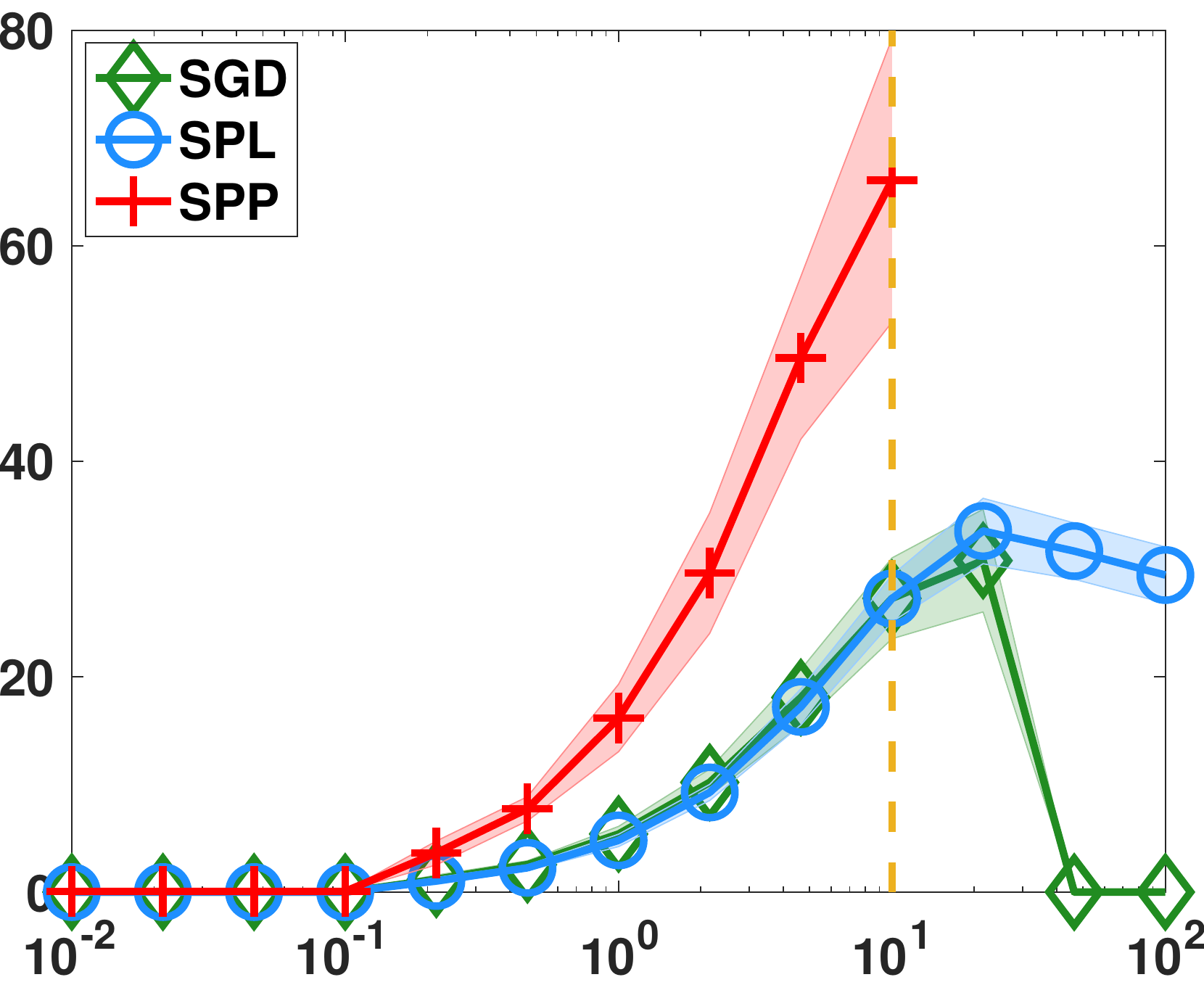}
		\includegraphics[scale=0.20]{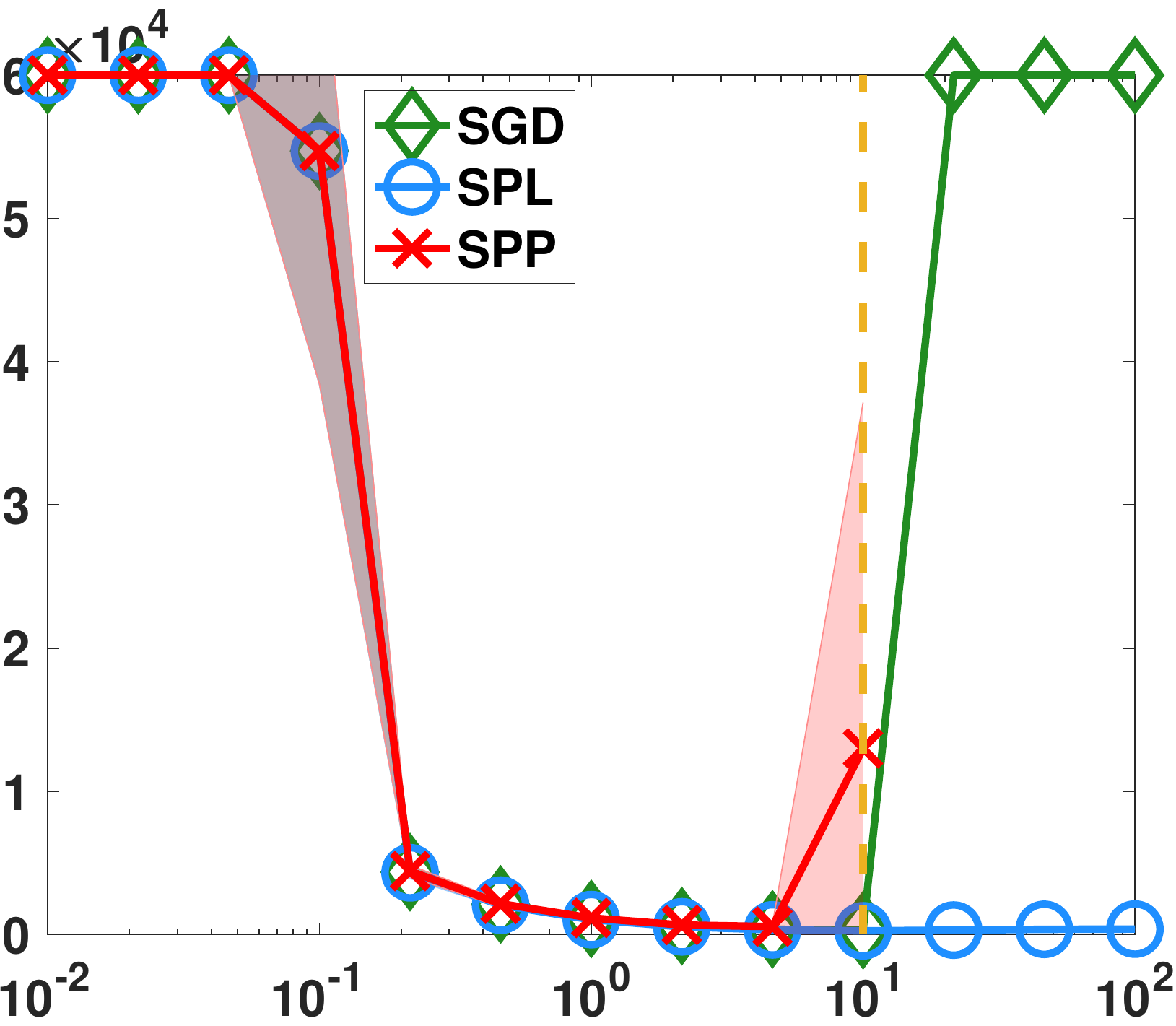}
		\includegraphics[scale=0.20]{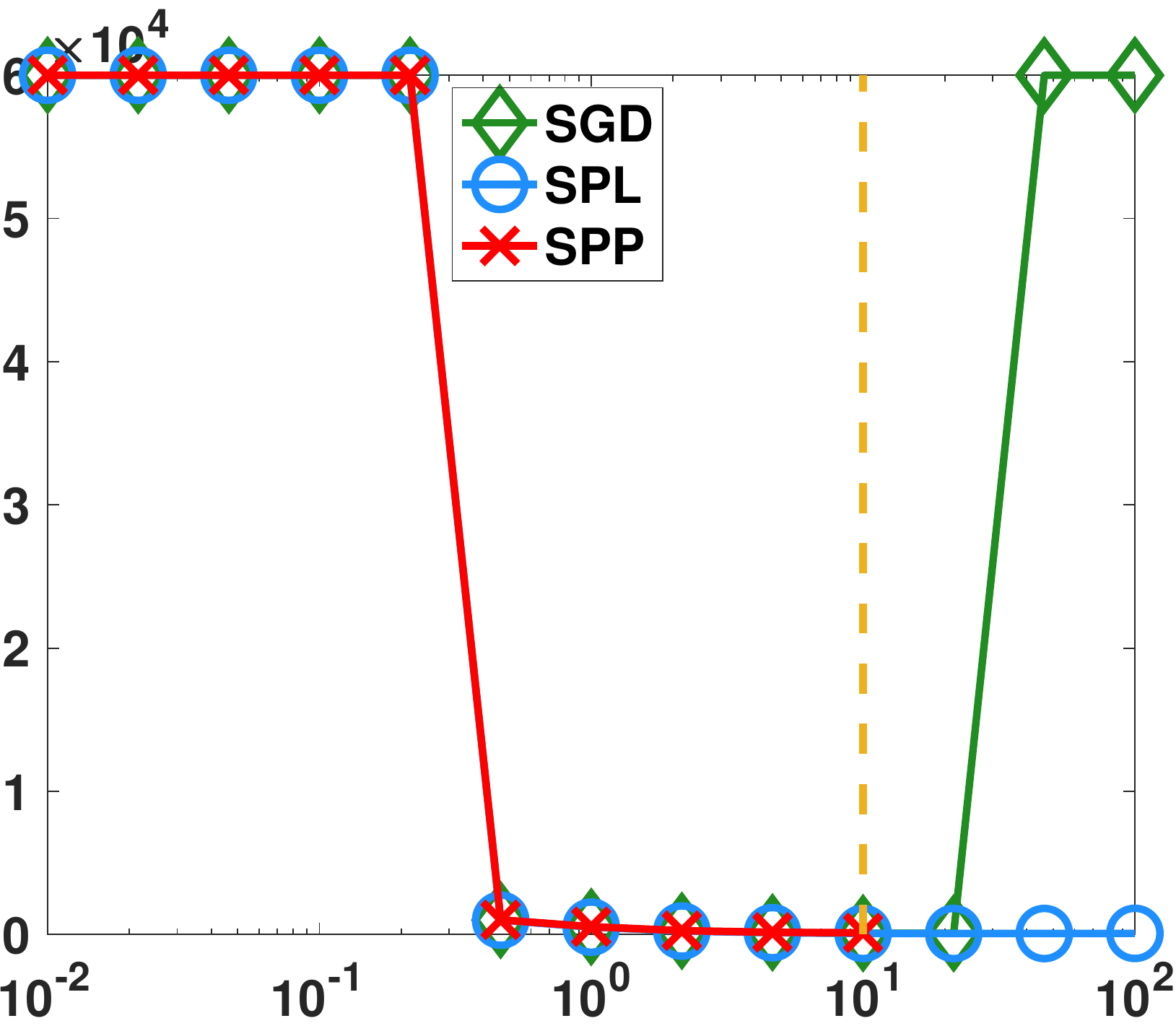}
		\includegraphics[scale=0.20]{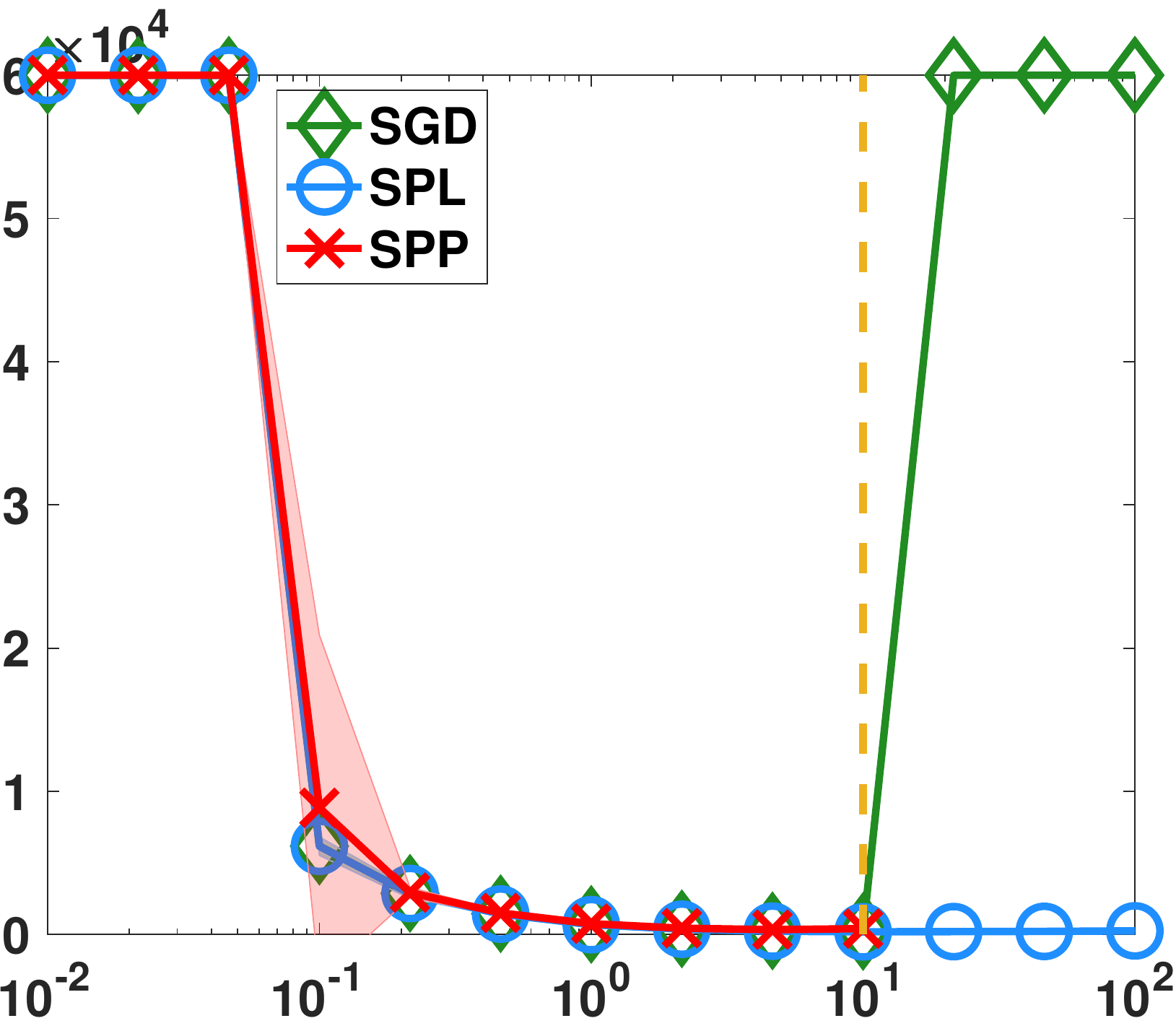}
		\includegraphics[scale=0.20]{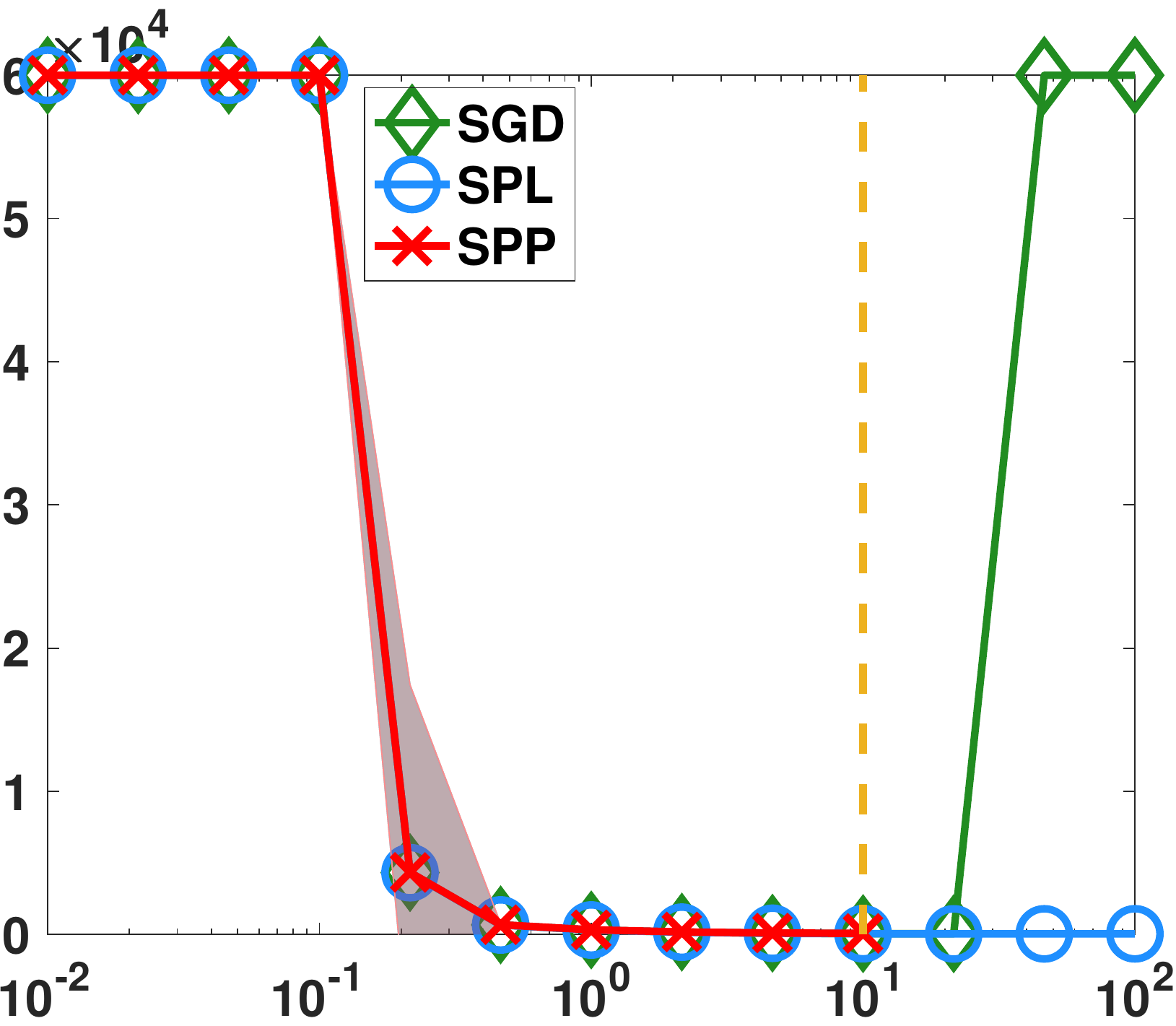}

		\caption{First row: Speedup vs. Stepsize $\alpha_0$. Second row: Iteration number on reaching desired accuracy vs. Stepsize $\alpha_0$. 
 From left to right: $\kappa=10, (p_\text{fail},m) = (0.2, 8), (0.2, 32), (0.3, 8), (0.3, 32)$. 
		\label{fig:bd-speedup-stepsize}}
\end{figure*}

\begin{figure*}[!h]
\centering
\includegraphics[scale=0.20]{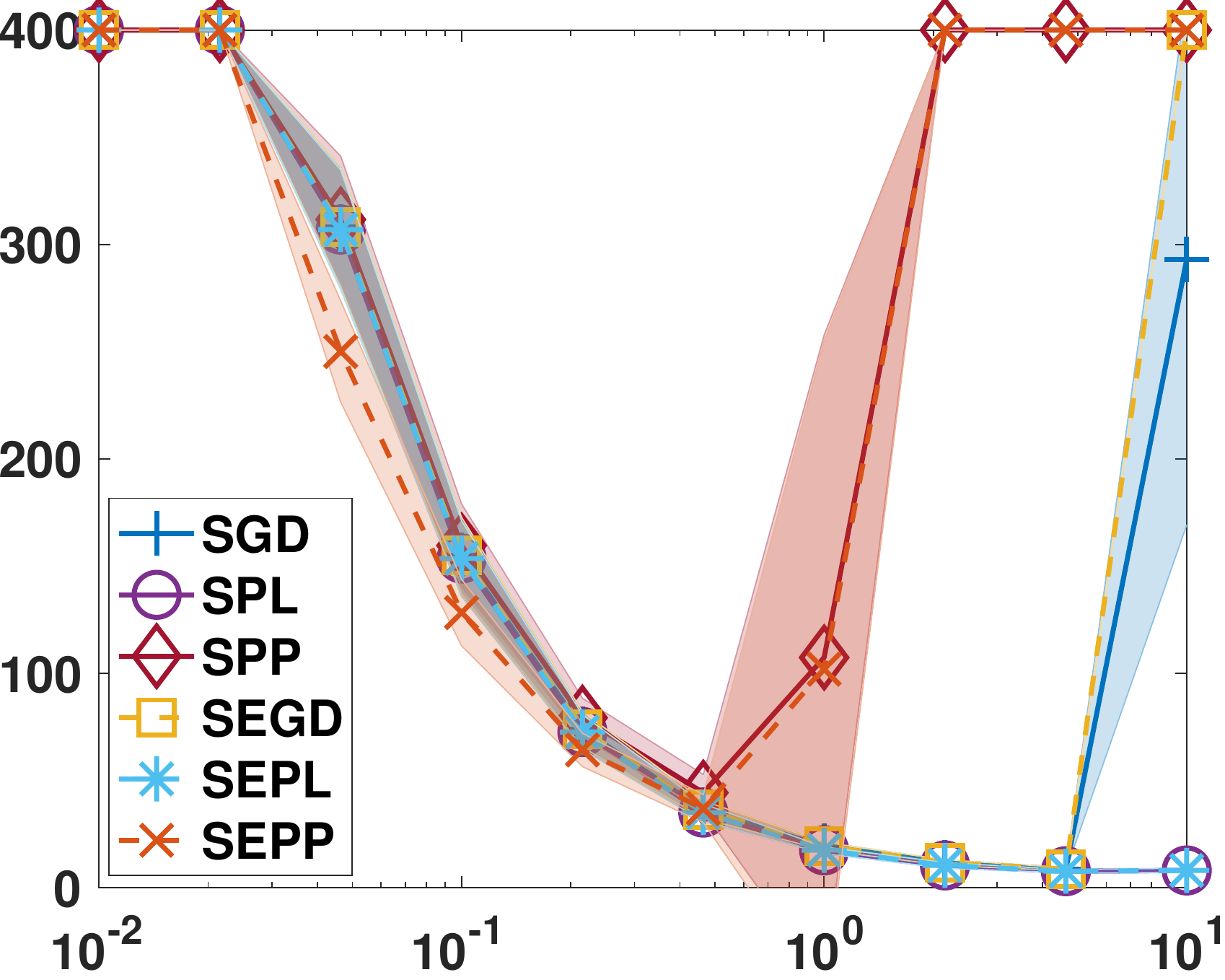}
 \includegraphics[scale=0.20]{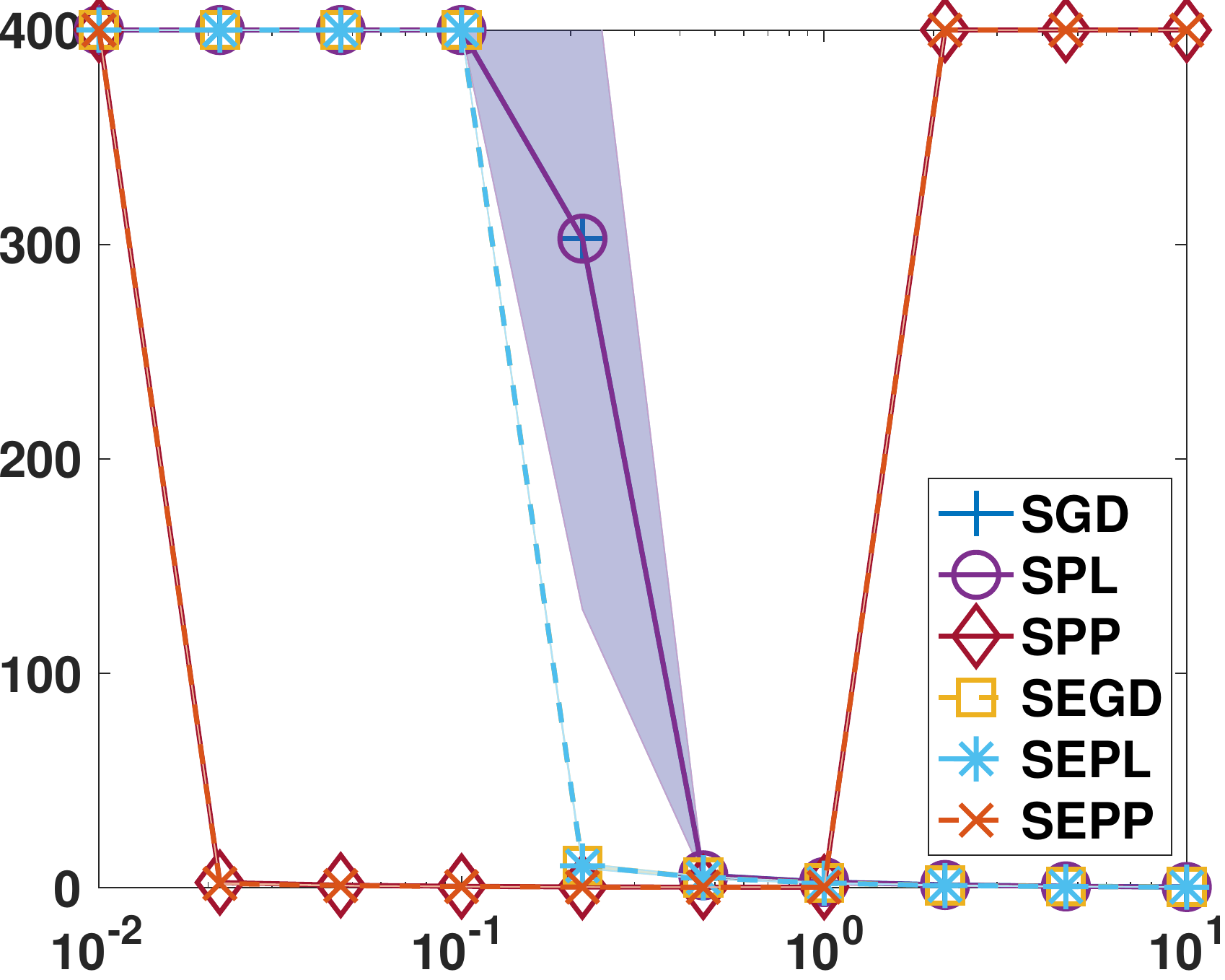}\includegraphics[scale=0.20]{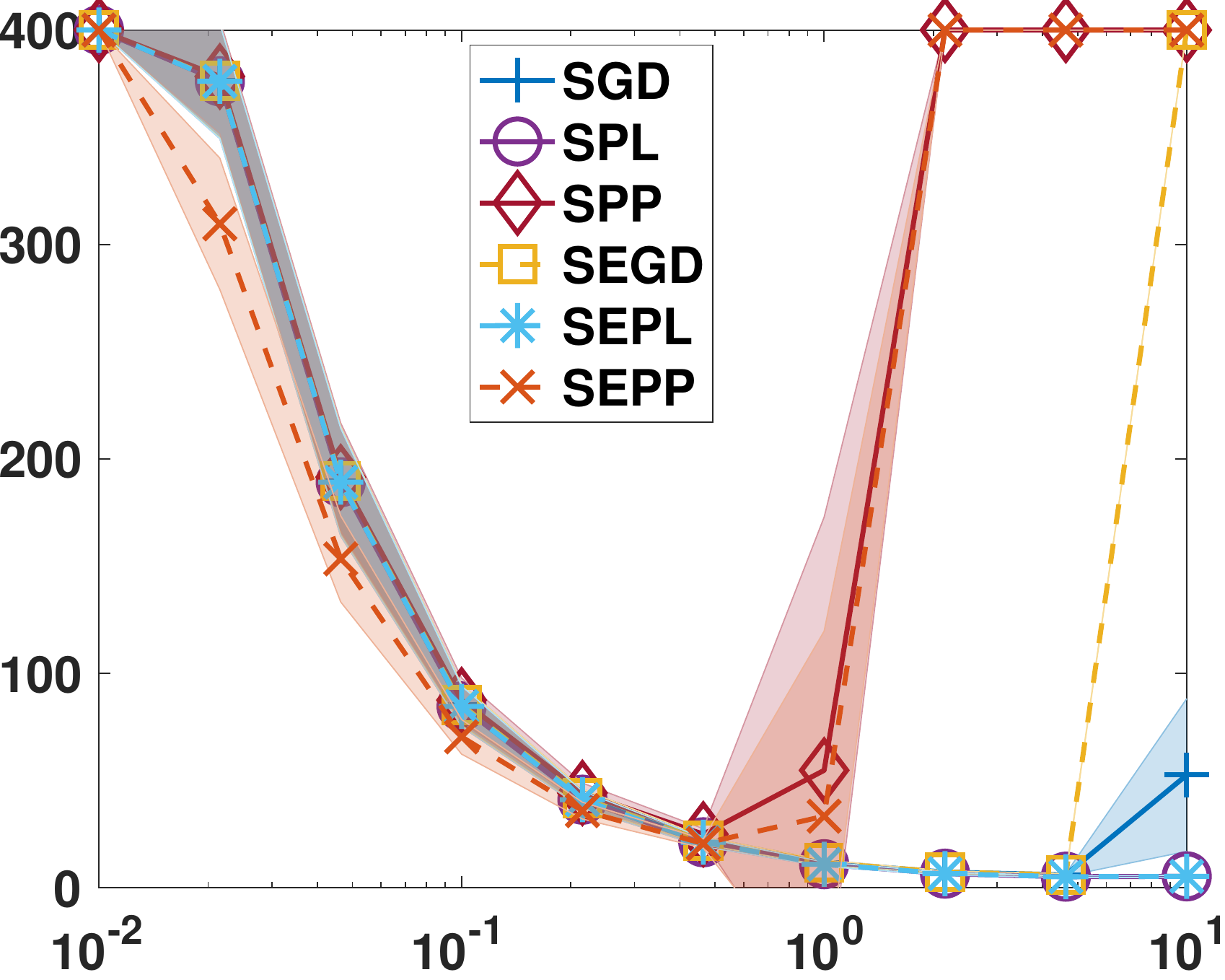}\includegraphics[scale=0.20]{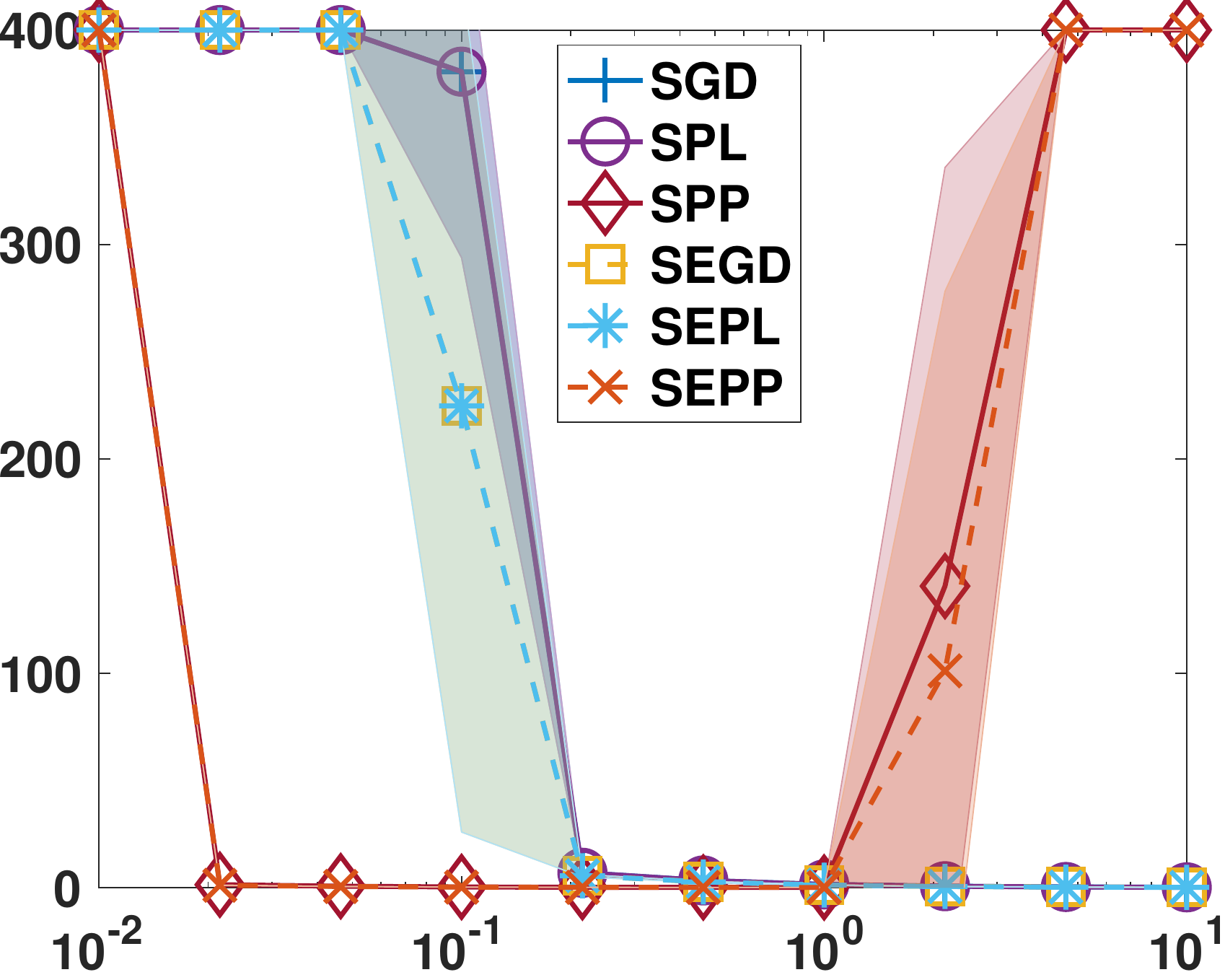}\\
\includegraphics[scale=0.20]{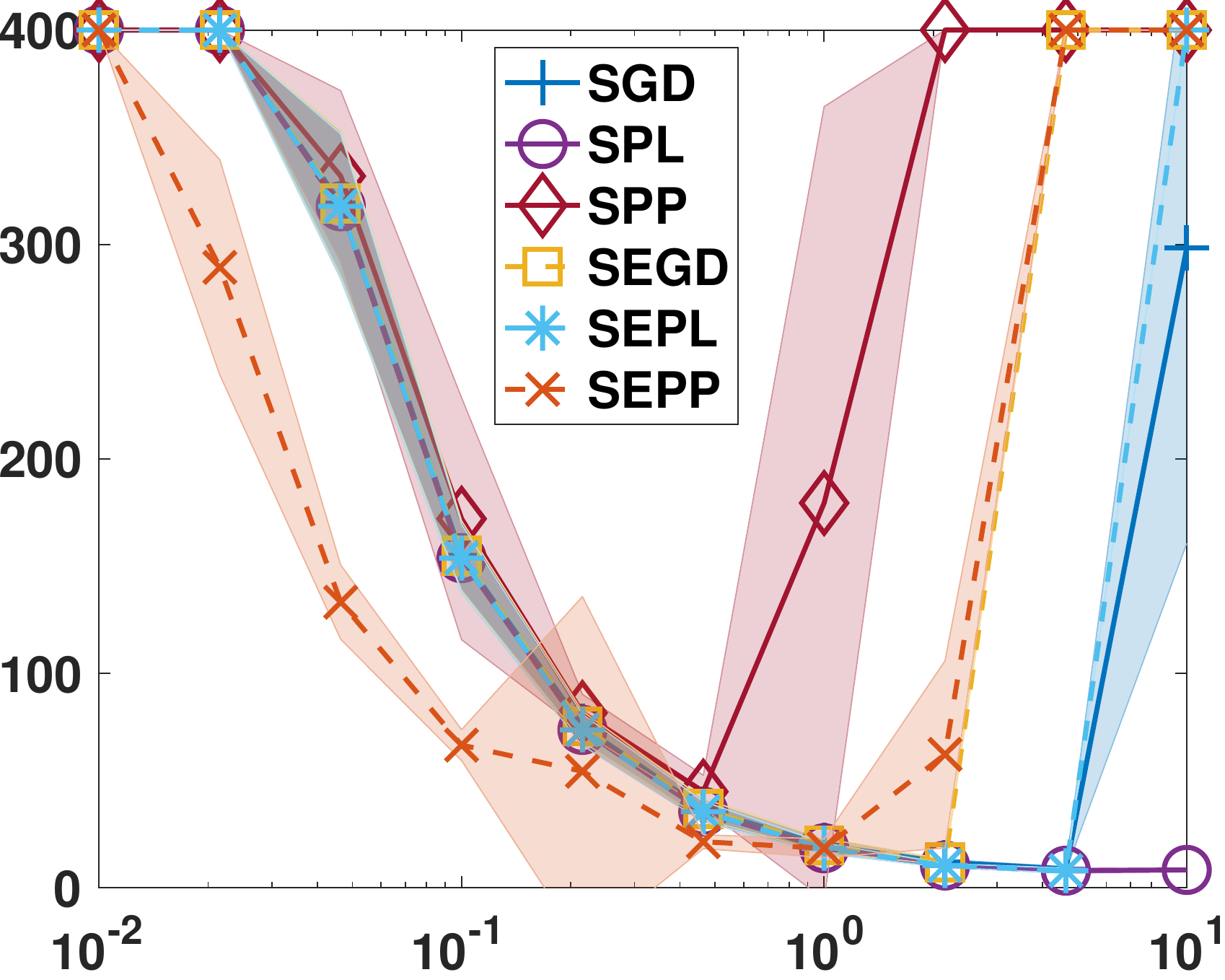} \includegraphics[scale=0.20]{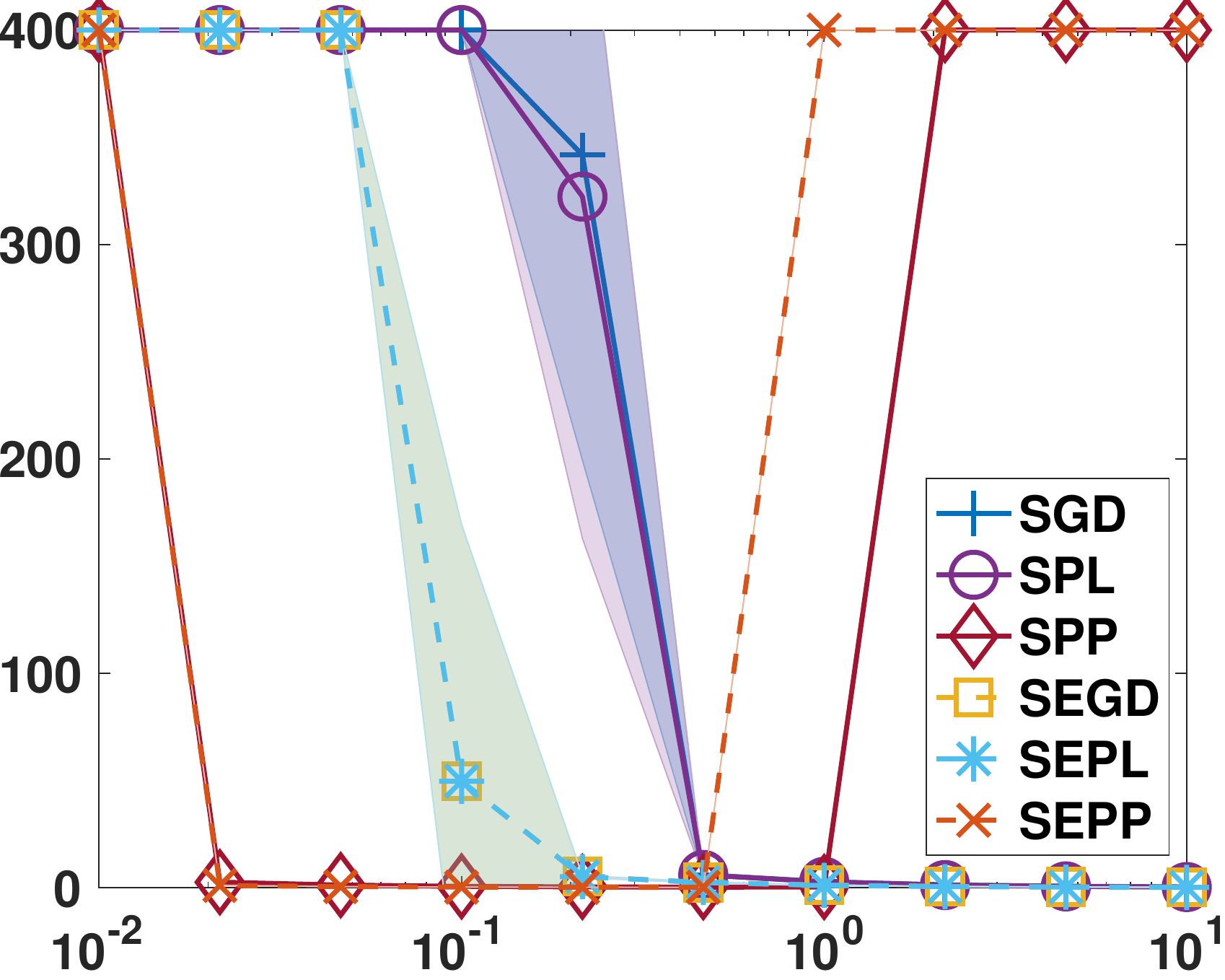}\includegraphics[scale=0.20]{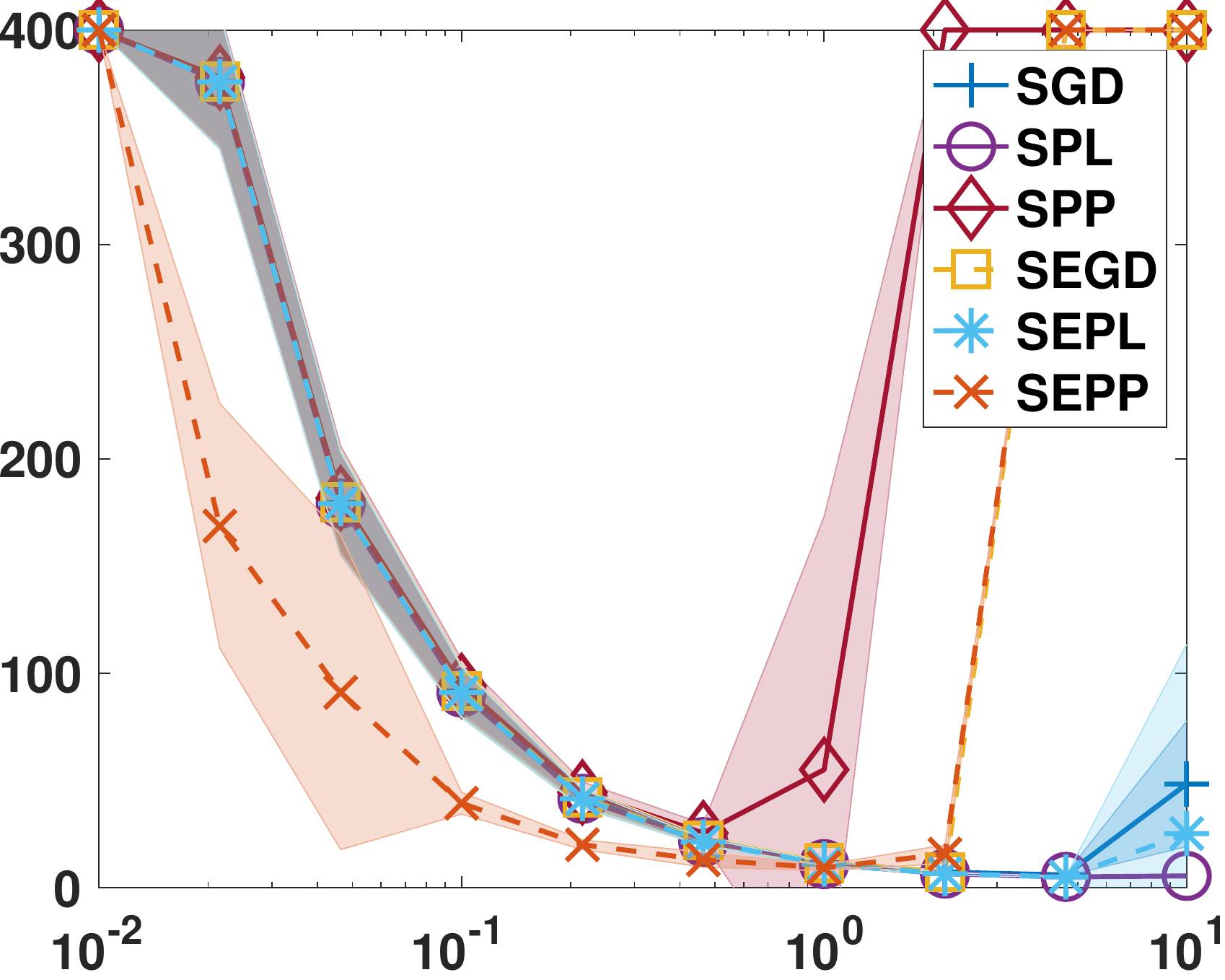}\includegraphics[scale=0.20]{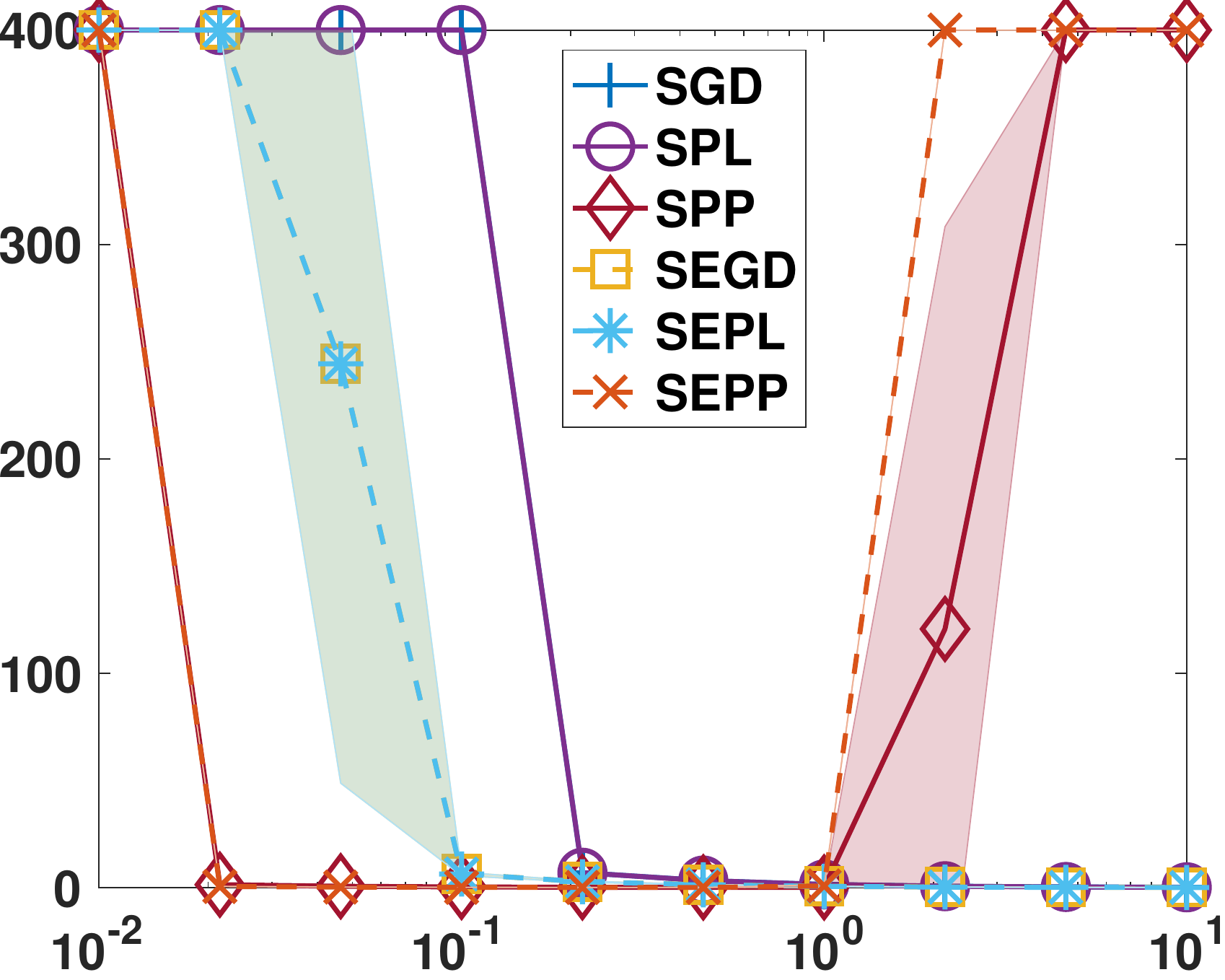}

	\caption{Epoch number on reaching desired accuracy vs. Stepsize $\alpha_0$. First row: $\beta = 0.2$. Second row: $\beta = 0.6$. From left to right: $ \kappa = 10, (p_\text{fail},m) = (0.2, 1), (0.2, 32), (0.3, 1), (0.3, 32)$.\label{fig:bd-itercount-stepsize}}
\end{figure*}

\newpage
\bibliographystyleapp{abbrvnat}
\bibliographyapp{citations}

\end{document}